\documentclass[11pt]{amsart}
\usepackage{amsmath}
\usepackage{amscd}
\usepackage{amssymb}
\usepackage{amsbsy}
\usepackage{amsfonts}
\usepackage{color}
\usepackage{latexsym}
\usepackage{graphics}
\usepackage{graphicx}
\usepackage{amsmath,amscd,latexsym}
\usepackage{array}
\usepackage{enumerate}
\usepackage{mathrsfs}
\usepackage[
  textheight=20cm,
  textwidth=6.0in,
  centering
]{geometry}

\theoremstyle{plain}
\newtheorem{theorem}{Theorem}[section]
\newtheorem{proposition}[theorem]{Proposition}
\newtheorem{lemma}[theorem]{Lemma}
\newtheorem{corollary}[theorem]{Corollary}
\newtheorem{main theorem}[theorem]{Main Theorem}

\newtheorem{question}[theorem]{Question}

\theoremstyle{definition}   
\newtheorem{remark}[theorem]{Remark}
\newtheorem{definition}[theorem]{Definition}

\newlength\savewidth

\newcommand{\svert}{\,|\,}

\newcommand{\rank}{{\rm rank}}

\newcommand{\Lab}{{\bf Lab}}


\begin{document}
\title[Maximal ranks of free abelian subgroups of generalized cactus groups associated with \\ finite Coxeter systems]
{Maximal ranks of free abelian subgroups of generalized cactus groups associated with \\ finite Coxeter systems}
\author{Jan Kim}
\address{Department of Mathematical Sciences \\
    KAIST, 291 Daehak-ro, Yuseong-gu, Daejeon, 34141, South Korea}
\email{jankim@kaist.ac.kr}	

\author{Junseok Kim}
\address{Department of Mathematics \\
        Technion - Israel Institute of Technology \\
        Technion City, Haifa, Israel, 3200003}
\email{jsk8818@campus.technion.ac.il}		


\begin{abstract}
We determine the maximum ranks of free abelian subgroups of generalized cactus groups associated with finite Coxeter systems.
To obtain the upper bounds, we use the cubical flat torus theorem of Wise and Woodhouse. We then verify the sharpness of these upper bounds by constructing free abelian subgroups of the corresponding ranks.
Our result shows that the maximal ranks of free abelian subgroups of ordinary cactus groups can be larger than conjectured by Genevois. 
As an application, we show that these groups are not quasi-isometric to any right-angled Artin group unless they are virtually free abelian.
\end{abstract}

\maketitle
\setcounter{tocdepth}{1}
\tableofcontents

\section{Introduction}
\label{sec:introduction}		

Let $n \ge 2$ be an integer. 
The {\em cactus group} $J_n$ is the group generated by
\[
\{s_{p,q} \svert 1 \le p < q \le n\},
\]
with the following relations:
{\setlength{\leftmargini}{2em}
\begin{itemize}
\item $s_{p,q}^2=1$ for all $1 \le p < q \le n$;
\item $s_{p,q}s_{m,r} = s_{m,r}s_{p,q}$ whenever $[p,q] \cap [m,r] = \emptyset$;
\item $s_{p,q}s_{m,r} = s_{p+q-r, p+q-m}s_{p,q}$ whenever $[m,r] \subset [p,q]$.
\end{itemize}
}
Here, $[m,r]$ and $[p,q]$ denote the sets $\{m,m+1,\dots,r\}$ and $\{p,p+1,\dots,q\}$, respectively.
Cactus groups were first introduced in~\cite{Henriques_Kamnitzer} in connection with crystals and coboundary categories. They are closely related to the topology of real moduli spaces. For each $n\ge 2$, the cactus group $J_n$ admits a natural epimorphism onto the symmetric group $S_n$, whose kernel, denoted by $PJ_n$ and
called the {\em pure cactus group}, is isomorphic to the fundamental
group of the real moduli space $\overline{M}_{0,n+1}(\mathbb{R})$~\cite[Theorem~9]{Henriques_Kamnitzer}. 

From the point of view of geometric group theory, Genevois~\cite{Genevois} studied cactus groups using their median geometry.
In particular, Genevois~\cite{Genevois} proved that the canonical Cayley graph of $J_n$ is a median graph, and hence $J_n$ acts properly and cocompactly on a $\mathrm{CAT}(0)$ cube complex. In fact, the median property of the canonical Cayley graphs of cactus groups can be deduced from earlier work in~\cite{Davis_Januszkiewicz_Scott,Scott}, while Genevois provided a direct proof.
Genevois further investigated several notions of hyperbolicity for cactus groups and proved that $J_n$ is virtually cocompact special.
In this setting, by the flat torus theorem~\cite[Theorem~II.7.1]{Bridson_Haefliger} and the flat closing theorem for compact special cube complexes~\cite[Corollary~2]{Caprace_Haglund}, the maximal possible rank of its free abelian subgroups is equals the maximum dimension of flats in this cube complex.
This gives a geometric motivation for determining that rank. In this direction,
Genevois~\cite[Conjecture~8.6]{Genevois} conjectured that $J_n$ contains
a free abelian subgroup of rank $k$ if and only if $k\leq \lfloor n/ 3 \rfloor$.
We prove that $J_n$ contains a free abelian subgroup of rank $k$ if and
only if
\[
k\leq \left\lfloor \frac{n-1}{2}\right\rfloor;
\]
see Theorem~\ref{thm:maximal_rank} below.
Thus, Genevois's conjecture does not hold in general.

\bigskip 

Our main result extends this computation to generalized
cactus groups associated with finite Coxeter systems, giving the exact
maximal rank of their free abelian subgroups. 
We begin with the definition of generalized cactus groups.

\begin{definition}[Generalized cactus groups and pure subgroups]
\label{def:pcw}
Let $(W,S)$ be a Coxeter system. For $I\subseteq S$, let
$W_I=\langle I\rangle$, and let $\mathcal I$ be the collection of
nonempty subsets $I\subseteq S$ such that $W_I$ is finite and the
Coxeter subdiagram spanned by $I$ is connected. For $I\in\mathcal I$,
let $\omega_I$ denote the longest element of $W_I$.

Then, the {\em generalized cactus group} $C_W$ is the group with generating set $\{\sigma_I \svert I\in\mathcal I\}$ and relations
\begin{enumerate}[\rm (i)]
\item $\sigma_I^2=1$ for all $I\in\mathcal I$;
\item $[\sigma_I,\sigma_J]=1$ for all $I,J \in\mathcal I$ such that $W_{I\cup J}=W_I\times W_J$;
\item $\sigma_I\sigma_J=\sigma_J\sigma_{\omega_J I\omega_J^{-1}}$ for all $I,J\in\mathcal I$ such that $I\subset J$,
\end{enumerate}

Let $g_W:C_W\to W$ be the homomorphism defined by
$g_W(\sigma_I)=\omega_I$. The \emph{generalized pure cactus group} is
the kernel
\[
PC_W:=\ker(g_W).
\]
If $W$ is finite, then $PC_W$ has finite index in $C_W$.
\end{definition}

Generalized cactus groups and related groups have been studied in a variety of contexts, including representation theory and geometric group theory.
Davis, Januszkiewicz, and Scott~\cite{Davis_Januszkiewicz_Scott_2} studied groups related to generalized cactus groups in connection with blow-ups of reflection tilings.
For Weyl groups, Losev~\cite{Losev} investigated generalized cactus groups in connection with Kazhdan--Lusztig cells and their actions on the corresponding Weyl groups. Bonnaf\'e~\cite{Bonnafe} extended this construction to arbitrary Coxeter systems.
Bellingeri, Godelle, and Paris~\cite{Bellingeri_Godelle_Paris} realized generalized cactus groups as examples of trickle groups, while Genevois~\cite{Genevois_3} realized them as examples of quandle products.
Genevois's results imply that the canonical Cayley graphs of generalized cactus groups are median.
When the defining Coxeter systems are finite, the cube-completions of these Cayley graphs are $\mathrm{CAT}(0)$ cube complexes on which the corresponding generalized cactus groups act geometrically.

We can now state our main theorem.

\begin{theorem}\label{thm:maximal_rank}
Let $(W,S)$ be a finite Coxeter system whose Coxeter diagram has
connected components $S_1,\dots,S_k$ and let $n_i=|S_i|$ for each $i$. Then the generalized cactus group $C_W$ contains a free abelian subgroup of rank $r$ if and only if
\[
r\leq \sum_{i=1}^k \left\lfloor \frac{n_i}{2} \right\rfloor.
\]
\end{theorem}

When $(W,S)$ is of type $A_{n-1}$, we have $|S|=n-1$, and the
generalized cactus group $C_W$ is isomorphic to the ordinary cactus
group $J_n$. Hence, Theorem~\ref{thm:maximal_rank} covers the above
result for $J_n$. Regarding the hyperbolicity, it is known that $J_n$ is hyperbolic for $n \le 4$, while $J_n$ is not hyperbolic for $n \ge 6$~\cite[Proposition~6.1]{Genevois}. 
For $n=5$, Theorem~\ref{thm:maximal_rank} gives an explicit algebraic proof of non-hyperbolicity by exhibiting a subgroup isomorphic to $\mathbb{Z}^2$ in terms of the standard generators of $J_5$.
The same conclusion can also be deduced from earlier topological results. Indeed, the pure cactus group $PJ_5$, which is a subgroup of $J_5$, is isomorphic to $\pi_1(\overline{M}_{0,6}(\mathbb{R}))$~\cite[Theorem~9]{Henriques_Kamnitzer}, and this fundamental group contains a subgroup isomorphic to $\mathbb{Z}^2$~\cite[Example~3.22]{Bonifant_Milnor}.

Moreover, Yu~\cite{Yu} proved that generalized cactus groups associated with finite Coxeter systems are linear. For general Coxeter systems, Yu also showed that the generalized pure cactus groups embed into right-angled Coxeter groups.
Right-angled Coxeter groups are closely related to right-angled Artin groups (RAAGs); in particular, every RAAG occurs as a finite-index subgroup of a right-angled Coxeter group~\cite{Davis_Januszkiewicz}.
In addition, RAAGs form a standard class of groups acting geometrically on \text{CAT(0)} cube complexes and also play an important role in the theory of special cube complexes~\cite{Haglund_Wise}.
The quasi-isometry classification of RAAGs has been studied for several classes, including tree RAAGs, atomic RAAGs, and RAAGs with finite outer automorphism groups~\cite{Behrstock_Neumann, Bestvina_Kleiner_Sageev, Huang_2}. Huang and Kleiner~\cite{Huang_Kleiner} further characterized finitely generated groups quasi-isometric to RAAGs with finite outer automorphism groups.
This motivates us to investigate whether generalized cactus groups can be quasi-isometric to RAAGs.

As an application of Theorem~\ref{thm:maximal_rank}, we prove the following result on when a generalized cactus group can be quasi-isometric to a RAAG.

\begin{theorem}\label{thm:mainRAAGQI}
Suppose that the Coxeter diagram of a finite Coxeter system $(W,S)$ has a connected component whose number of vertices is at least three. (Equivalently, the generalized cactus group $C_W$ is not virtually $\mathbb{Z}^d$ for any $d\geq 0$.)
Then $C_W$ is not quasi-isometric to any RAAG. 
\end{theorem}

The finiteness of $W$ plays an essential role in our proof of Theorem~\ref{thm:mainRAAGQI}.
It ensures that $PC_W$ has finite index in $C_W$, allowing us to obtain a compact special cube complex whose fundamental group has finite index in $C_W$.
We then apply Genevois's result~\cite[Theorem~4.6]{Genevois_2} to obtain free abelian subgroups from combinatorial flats.
When $W$ is infinite, $PC_W$ has infinite index in $C_W$, so this construction does not yield the compact special cube complex required to apply Genevois's theorem.

The paper is organized as follows. In Section~\ref{sec:preliminaries}, we collect the definitions and preliminary results needed throughout the paper. 
In Section~\ref{sec:maximal_rank}, we prove the ``only if'' direction of Theorem~\ref{thm:maximal_rank}, establishing an upper bound on the ranks of free abelian subgroups of generalized cactus groups. 
In Section~\ref{sec:finding_subgroups}, we verify the sharpness of these upper bounds by constructing free abelian subgroups of the corresponding ranks.
Finally, in Section~\ref{sec:application}, we study the large-scale geometry of generalized cactus groups and prove Theorem~\ref{thm:mainRAAGQI}.

\vspace{0.6cm} 

\noindent \textbf{Acknowledgment.} 
We thank Anthony Genevois for bringing Theorem~\ref{thm:mainRAAGQI} in the case of ordinary cactus groups to our attention.
The first author was supported by the Basic Science Research Program through the National Research Foundation of Korea(NRF) funded by the Ministry of Education (RS-2021-NR060138). The second author was supported by Israel Science Foundation Grant 3423/24 (PI Ilya Gekhtman).

\section{Preliminaries}
\label{sec:preliminaries}

In this section, we recall the definitions and preliminary results that will be used throughout the paper.

\subsection{Coxeter groups}

Let $W$ be a group together with a generating set $S$ such that
\[
W=\langle S \svert (st)^{m(s,t)}=1 \ \text{ for } s,t\in S \text{ with } m(s,t)<\infty \rangle,
\]
where $m(s,s)=1$ for all $s\in S$ and
$m(s,t)=m(t,s)\in\{2,3,\ldots,\infty\}$ for $s\neq t$.
Then $W$ is called a {\em Coxeter group}, and the pair $(W,S)$ is called a {\em Coxeter system}.

For $I\subset S$, let $W_I=\langle I\rangle$.
A subgroup of the form $W_I$ for some $I \subset S$ is called a {\em standard parabolic subgroup} of $W$. It is well known that $W_I$ is itself a Coxeter group with generating set $I$.
A subgroup $P \le W$ is called {\em parabolic} if
$P = w W_I w^{-1}$ for some $w \in W$ and some $I \subset S$.
A subset $I\subseteq S$ is called \emph{spherical} if $W_I$ is finite,
and \emph{irreducible} if the full subdiagram of the Coxeter diagram
spanned by $I$ is connected.
Let $\mathcal I$ denote the collection of nonempty spherical irreducible
subsets of $S$.
For a spherical subset $I\subseteq S$, we denote by
$\omega_I\in W_I$ its longest element.
The cardinality of $S$ is called the {\em rank} of $(W,S)$.
If $P\le W$ is a parabolic subgroup, write $P=wW_Tw^{-1}$ for some $T\subseteq S$ and $w\in W$.
We define the {\em rank} of $P$ to be $|T|$.
This definition is well-defined. To see this, suppose that
$P=wW_Tw^{-1}=w'W_{T'}w'^{-1}$.
Then $W_T$ and $W_{T'}$ are conjugate in $W$.
Hence, $|T|=|T'|$ by \cite[Lemma~3.2]{Qi}.

We will also use the following notion of orthogonality for spherical irreducible subsets and parabolic subgroups.

\begin{definition}[Orthogonality and nesting]
For $I,J\in\mathcal I$, we say that $I$ and $J$ are \emph{orthogonal}, and write $I\perp J$, if $W_{I\cup J}=W_I\times W_J$.
We say that $I$ and $J$ are \emph{nested} if either $I\subset J$ or $J\subset I$.

For parabolic subgroups $P,Q\leq W$, we say that $P$ and $Q$ are
\emph{orthogonal} if $PQ=QP$ and the multiplication map
$P\times Q\longrightarrow PQ$ defined by $(p,q)\longmapsto pq$
is an isomorphism.
We say that $P$ and $Q$ are \emph{nested} if either $P\subset Q$ or
$Q\subset P$.
\end{definition}

For $I,J\in\mathcal I$, the condition $I\perp J$ is equivalent to
$m(s,t)=2$ for every $s\in I$ and $t\in J$.
Equivalently, no vertex of $I$ is joined to a vertex of $J$ in the Coxeter diagram.
Let $\ell:W\to\mathbb{N}$ denote the word length with respect to $S$, where $\ell(w)$ is the minimum length of a word over $S$ representing $w$.

The following lemma will be useful for determining whether an element lies in a given standard parabolic subgroup.

\begin{lemma}\label{lem:parabolic-right-descent}
Let $(W,S)$ be a Coxeter system, and let $T\subseteq S$.
If $w\in W_T=\langle T\rangle$ and $s\in S$ satisfies $\ell(ws)<\ell(w)$, then $s\in T$.
\end{lemma}

\begin{proof}
Since $\ell(ws)<\ell(w)$, by \cite[Corollary~1.4.6(ii)]{Bjorner},
there exists a reduced expression of $w$ whose last letter is $s$.

On the other hand, since $w\in W_T$, by
\cite[Proposition~2.4.1(ii)]{Bjorner}, $w$ has a reduced expression
whose letters all lie in $T$.
By \cite[Corollary~1.4.8(ii)]{Bjorner}, the set of letters appearing
in a reduced expression of $w$ is independent of the chosen reduced
expression.
Therefore, $s\in T$.
\end{proof}


We say that a Coxeter system is \emph{finite} if $W$ is finite and $(W,S)$ is \emph{irreducible} if its Coxeter diagram is connected.
We use the classification of finite irreducible Coxeter systems in the form given by Davis~\cite[Table~6.1 and Theorem~C.1.2]{Davis}.
The possible types of finite irreducible Coxeter systems are
\[
A_n,\ B_n,\ D_n,\ E_6,\ E_7,\ E_8,\ F_4,\ H_3,\ H_4,\ I_2(p) \ (p\geq 5).
\]
Therefore, every finite Coxeter group is isomorphic to the product of finitely many finite irreducible Coxeter groups.
 Our indexing convention for the Coxeter diagram of type $D_n$ differs from that used in~\cite[Section~8.2]{Bjorner}.
\begin{figure}[h]
\begin{center}
	\includegraphics[width=0.72\columnwidth]{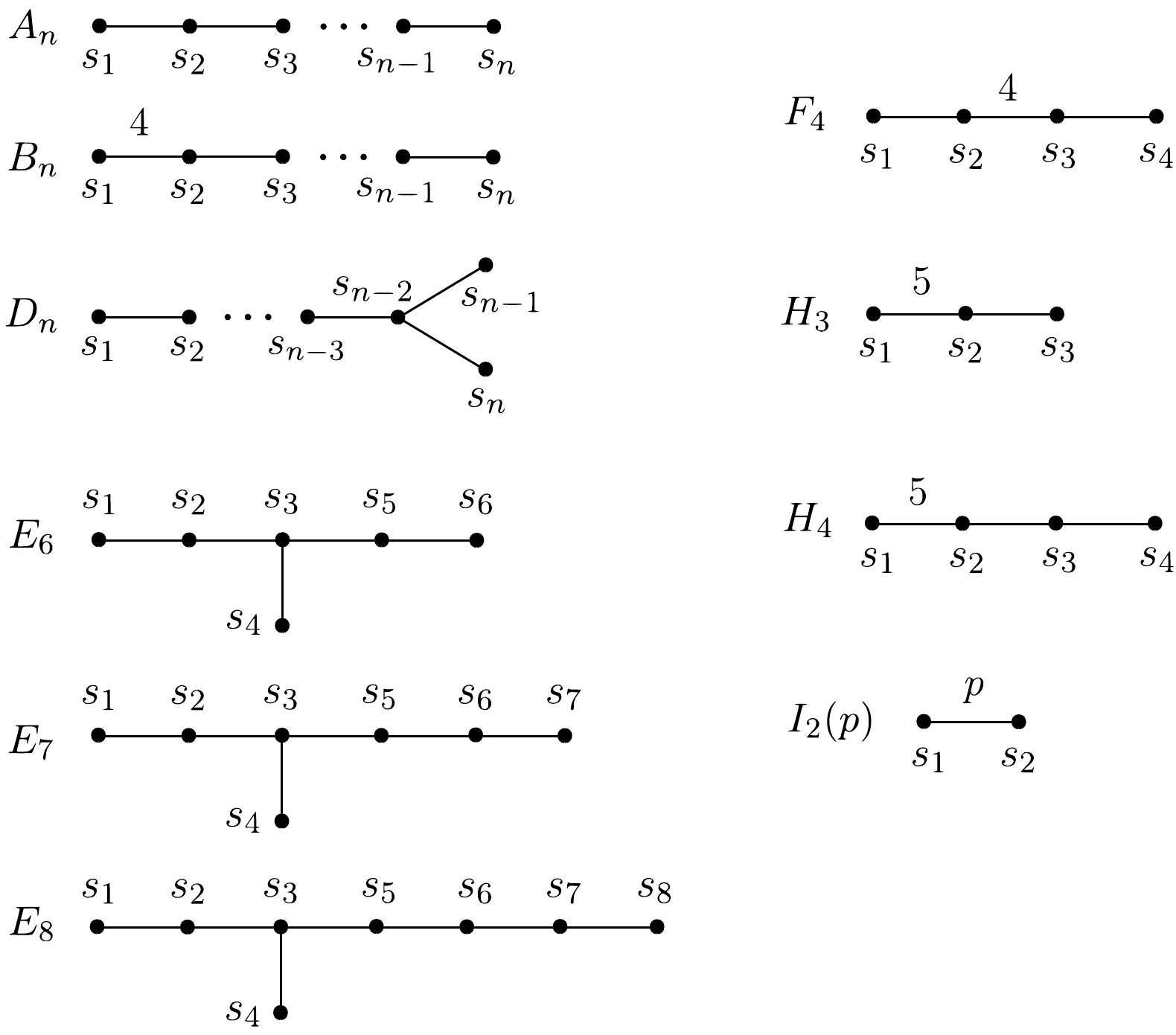}
\end{center}
\caption{\label{fig:diagrams}
Coxeter diagrams of finite irreducible Coxeter systems $(W, S)$}						
\end{figure}

\subsection{Longest elements of finite Coxeter groups}
\label{subsec:finitecoxetergroups}
In this subsection, we describe the longest elements of finite irreducible Coxeter groups type by type.
For types $A_n$, $B_n$, $D_n$, and $I_2(p)$, we use their standard permutation or reflection models.
Throughout this subsection, we write $S=\{s_1,\ldots,s_n\}$.

We describe the longest elements of finite Coxeter groups type by type.
For the infinite families $A_n,B_n,D_n$, we use the following concrete models.
In type $A_n$, we identify $W$ with the symmetric group $S_{n+1}$, where $s_i$ is the transposition $(i\ i+1)$.
See \cite[Example~1.2.3 and Proposition~1.5.4]{Bjorner}.
Under this identification, the longest element of the standard parabolic subgroup generated by $s_p, s_{p+1}, \ldots,s_q\in S$, where $1 \le p \le q \le n$, is the permutation which reverses the interval $\{p,\ldots,q+1\}$ and fixes its complement. (See \cite{Bjorner} the discussion following Proposition 2.3.4 for details.)

\medskip 

In type $B_n$, 
let $S^B_n$ be the group of all bijections $w$ of $[\pm n]=\{-n,\ldots,-2,-1,1,2,\ldots n\}$ satisfying
\[
w(-a)=-w(a)
\]
for every $a\in[\pm n]$.
We write
\[
w=[a_1,\ldots,a_n]
\]
to mean that $w(i)=a_i$ for $1\leq i\leq n$; this is called the \emph{window} notation.
We call $S^B_n$ the group of all {\em signed permutations} of $[n]=\{1,2,\ldots,n\}$.

Following Björner--Brenti~\cite[Section~8.1]{Bjorner}, let 
\[
S^B=\{s^B_0,s^B_1,\ldots,s^B_{n-1}\},
\]
where
\[
s^B_0=[-1,2,\ldots,n]
\]
and, for $1\leq j\leq n-1$,
\[
s^B_j=[1,\ldots,j-1,j+1,j,j+2,\ldots,n].
\] 
In other words, $s^B_0$ changes the sign of the first entry, while $s^B_j$ interchanges the $j$th and $(j+1)$st entries.
Then, by~\cite[Proposition~8.1.3]{Bjorner}, $(S^B_n,S^B)$ is a Coxeter system of type $B_n$.
Our generators are related to the above generators by the following:
\[
s_1=s^B_0,\qquad s_j=s^B_{j-1}\quad (j\geq 2).
\]
With the window notation, the longest element $w_0$ of the Coxeter group of type $B_n$ is the signed permutation which changes the sign of every entry, that is, 
\[
w_0=[-1,-2,\ldots,-n].
\]
Indeed, for $v\in S^B_n$ and $0\le i\le n-1$, one has
\[
\ell(vs_i^B)<\ell(v)\iff v(i)>v(i+1)\qquad (v(0):=0),
\]
see \cite[Proposition~8.1.2]{Bjorner}. Since $0>w_0(1)=-1$ and
$w_0(i)=-i>-(i+1)=w_0(i+1)$ for all $1\le i\le n-1$, we have
\[
\ell(w_0s_i^B)<\ell(w_0)\quad\text{for every }i=0,1,\ldots,n-1.
\]
Therefore, $w_0^{-1}$ satisfies $\ell(s_i^B w_0^{-1})<\ell(w_0^{-1})$ for every $i=0, 1, \dots, n-1$ because $\ell(u)=\ell(u^{-1})$.
By \cite[Proposition~2.3.1(ii)]{Bjorner}, it follows that
$w_0^{-1}$ is the longest element. 
Finally, $w_0^{-1}=w_0$ since $w_0^2=1$.

\medskip 

Similarly, in type $D_n$, let $S^D_n$ be the subgroup of $S^B_n$ consisting of signed permutations with an even number of negative entries:
\[
S^D_n=\{w\in S^B_n\mid \#\{i\mid w(i)<0\}\equiv 0 \pmod 2\}.
\]
Following Björner--Brenti~\cite[Section~8.2]{Bjorner}, let
\[
S^D=\{s^D_0,s^D_1,\ldots,s^D_{n-1}\},
\]
where
\[
s^D_0=[-2,-1,3,\ldots,n] \quad \text{and} \quad s^D_j=s^B_j \quad \text{for $1 \le j \le n-1$.}
\]
Thus, $s^D_j$ interchanges the $j$th and $(j+1)$st entries for $1\leq j\leq n-1$, while $s^D_0$ interchanges the first two entries and changes both signs.
Then, by~\cite[Proposition~8.2.3]{Bjorner}, $(S^D_n,S^D)$ is a Coxeter system of type $D_n$.
Moreover, the longest element of the Coxeter group of type $D_n$ is
\[
w_0=
\begin{cases}
[-1,-2,\ldots,-n] & \text{if } n \text{ is even}, \\
[1,-2,\ldots,-n] & \text{if } n \text{ is odd}.
\end{cases}
\]
As in type $B_n$, this follows from the standard signed-permutation model and the length criterion for type $D_n$ in the window notation; see \cite[Proposition~8.2.2]{Bjorner}.

Note that our indexing convention for type $D_n$ differs from that in~\cite[Section~8.2]{Bjorner} (see Figure~\ref{fig:diagrams}). 
With this convention, our generators correspond to those in~\cite[Section~8.2]{Bjorner} as follows:
\[
s_j=s^D_{n-j} \quad \text{for $1\leq j\leq n-2$},
\qquad
s_{n-1}=s^D_1,
\qquad
s_n=s^D_0.
\]
In~\cite[Section~8.2]{Bjorner}, $s_j^D$ interchanges the $j$th and $(j+1)$st entries in the window notation for $1 \le j \le n-1$, and
$s_0^D=[-2,-1,3,\ldots,n]$. Under the correspondence above, for $1\leq j\leq n-1$, the generator
$s_j$ interchanges the $(n-j)$th and $(n-j+1)$st entries, while
\[
s_n=[-2,-1,3,\ldots,n].
\]

\medskip 

For type $I_2(p)$, we use the standard dihedral realization.
Thus, $W(I_2(p))$ is generated by two reflections $r_1$ and $r_2$ in lines that meet at angle $\pi/p$; their product $r_1r_2$ is a rotation through angle $2\pi/p$.
Moreover, every element is represented by an alternating word in the two simple generators of length at most $p$, and the two alternating words of length $p$ represent the same element; see \cite[Example~1.2.7]{Bjorner}.
Hence, the element represented by these two alternating words of length $p$ is the longest element.

\medskip 

For the remaining finite irreducible types $E_6,E_7,E_8,F_4,H_3$, and $H_4$, we used SageMath as follows.
We first imported SageMath's standard library and used the following helper function to construct the Coxeter group associated with a Coxeter matrix $M$:
\[
\texttt{CoxeterGroup(M, implementation="reflection")}.
\]
In our code, this was written as follows:
\begin{verbatim}
from sage.all import *

def Coxeter_group(M):
    return CoxeterGroup(M, implementation="reflection")
\end{verbatim}

For an element $g$ of a Coxeter group, we computed a reduced expression using SageMath's method \texttt{g.reduced\_word()}:
\begin{verbatim}
def reduced_word_expression(g):
    return g.reduced_word()
\end{verbatim}
We also use this method to check whether a given element is the identity: the output is the empty word if and only if the element is the identity.

To check whether an element $g$ is the longest element, we compared it with SageMath's built-in longest element \texttt{W.long\_element()}:
\begin{verbatim}
def longest_check(W, g):
    return g == W.long_element()
\end{verbatim}

In each concrete computation, we then specified the relevant Coxeter matrix and the word in the Coxeter generators.

We shall use the following standard fact of longest elements of finite irreducible Coxeter systems.
It determines exactly when the longest element acts trivially by conjugation, or equivalently, when the opposition involution is trivial. See \cite[Fact~1]{Zibrowius}, which records the classification from \cite[\S 27--2]{Kane}.

\begin{lemma}
{\cite[Proposition~2.3.2(i)]{Bjorner}\textnormal{;}
\cite[Remark~3.1.2]{Davis_Januszkiewicz_Scott_2}}
\label{lem:central-longest-element}
Let $(W,S)$ be a finite irreducible Coxeter system, and let $w_0$ denote its longest element. Then $w_0^2=1$.
Moreover, $w_0$ is central in $W$ if and only if the Coxeter type of $(W,S)$ is one of
\[
A_1,\quad B_n\ (n\geq 2),\quad D_{2m}\ (m\geq 2),\quad
E_7,\quad E_8,\quad F_4,\quad H_3,\quad H_4,
\]
or a dihedral type $I_2(p)$ with $p$ even.
Equivalently, $w_0$ is not central precisely in types
\[
A_n\ (n\geq 2),\quad D_{2m+1}\ (m\geq 2),\quad E_6,
\]
and in dihedral type $I_2(p)$ with $p$ odd.
\end{lemma}

See also \cite[Proposition~2.3.2(i)]{Bjorner} for proof of the equality $w_0^2=1$. 
Furthermore, in each of the non-central cases, conjugation by the longest element induces the unique nontrivial automorphism of the Coxeter diagram. See \cite[Remark 3.1.2]{Davis_Januszkiewicz_Scott_2}.

\subsection{Generalized cactus groups}

We now discuss several basic properties of generalized cactus groups and their Cayley graphs. We refer to Definition~\ref{def:pcw} in the Introduction for the definitions of $C_W$ and $PC_W$.
If $W$ is finite, then $PC_W$ has finite index in $C_W$.

\begin{remark}
\label{rmk:J_n_isomorphic}
In~\cite[Proposition~2.2]{Yu}, it is noted that the ordinary cactus group $J_n$ is isomorphic to the generalized cactus group of type $A_{n-1}$.
\end{remark}

Let $\mathscr{C}_W$ be the Cayley graph of $C_W$ with respect to the generating set $\{\sigma_I\}$.
Then $\mathscr{C}_W$ is a median graph.
This follows from the following observations. 
In~\cite[Example~8.12]{Genevois_3}, Genevois realized generalized cactus groups as cactus products whose factors are all isomorphic to $\mathbb{Z}_2$. Since cactus products are quandle products, it follows from~\cite[Theorem~4.12]{Genevois_3} that $\mathscr{C}_W$ is quasi-median.
Moreover, by~\cite[Lemma~4.13]{Genevois_3}, every complete subgraph of $\mathscr{C}_W$ is contained in a coset of a factor, which has order two.
Hence $\mathscr{C}_W$ contains no triangle.
Therefore, $\mathscr{C}_W$ is a median graph by~\cite[Theorem~4.2]{Genevois_3}.

We next determine when two adjacent edges of $\mathscr{C}_W$ span a 4-cycle. As generalized cactus groups are special cases of quandle products, the corresponding result for quandle products applies directly.

\begin{lemma}{\cite[Lemma 4.15]{Genevois_3}}
\label{lem:squares}
Let $h\in C_W$ and $I,J\in\mathcal I$. The two adjacent edges
\[
(h,h\sigma_I)\quad\text{and}\quad(h,h\sigma_J)
\]
span a 4-cycle in $\mathscr{C}_W$ if and only if $I$ and $J$ are orthogonal or nested.
Moreover, any 4-cycle in $\mathscr{C}_W$ is obtained from one of the defining 4-cycles associated to a commuting relation or a nesting relation by the left translation action of $C_W$ on $\mathscr{C}_W$.
\end{lemma}

\subsection{Links, cube complexes and quasilines}

Let $X$ be a cube complex.
For a vertex $v \in X$, the {\em link} of $v$ in $X$ is the simplicial complex whose vertices correspond to the edges incident to $v$, where a set of $k+1$ vertices spans a $k$-simplex if and only if the corresponding $k+1$ edges are contained in a common $(k+1)$-cube of $X$.
A simplicial complex $L$ is called {\em flag} if, whenever $v_0, \dots, v_k$ are vertices of $L$ such that $\{v_i, v_j\}$ is an edge for each $i \neq j$, the vertices $v_0, \dots, v_k$ span a $k$-simplex of $L$.
A cube complex is {\em non-positively curved} if the link of every vertex is a flag simplicial complex.
A simply connected non-positively curved cube complex is called a {\em \text{CAT(0)} cube complex}.
A {\em (cubical) quasiline} is a $\text{CAT(0)}$ cube complex that is quasi-isometric to $\mathbb{R}$.

Let $Y$ be a subcomplex of $X$. We say that $Y$ is {\em convex} if $Y$ is connected and every
combinatorial geodesic in $X^{(1)}$ joining two vertices of $Y$
is contained in $Y^{(1)}$.
A subcomplex $L$ of a simplicial complex $K$ is {\em full} if every collection of vertices of $L$ spanning a simplex in $K$ also spans a simplex in $L$.
A virtually abelian subgroup $A\leq G$ is called \emph{highest} if no finite-index subgroup of $A$ is contained in a virtually abelian subgroup of $G$ of strictly larger rank.

\begin{theorem}\cite[Theorem 3.6]{Wise_Woodhouse}
\label{thm:wise_woodhouse}
    Let $G$ act properly and cocompactly on a \text{CAT(0)} cube complex $\widetilde{X}$. Let $A$ be a highest virtually abelian subgroup of $G$ and let $p=\rank(A)$, where the rank of a virtually abelian
group is the rank of a finite-index free abelian subgroup. Then $A$ acts properly and cocompactly on a convex subcomplex $\widetilde{Y}\subset \widetilde{X}$ such that $\widetilde{Y}\cong \prod_{i=1}^p C_i$ where each $C_i$ is a quasiline.
\end{theorem}

\subsection{Yu's embedding}\label{subsec:Yus-embedding}

Following Yu's construction~\cite{Yu}, we describe the right-angled Coxeter group as follows.
Let
\[
\mathcal S \ :=\ \{ wW_Iw^{-1}\mid w\in W,\ I\in\mathcal I\},
\]
where $\mathcal I$ is the collection of nonempty spherical irreducible subsets of $S$.
For each $U,V\in \mathcal{S}$ define $m_{UV}$ by
\[
m_{UV}\ =\
\begin{cases}
1 & \text{if } U=V,\\
2 & \text{if } U\subseteq V \text{ or } V\subseteq U,\\
2 & \text{if } U\cap V=\{1\}\ \text{and}\ U\subseteq Z_W(V),\\
\infty & \text{otherwise},
\end{cases}
\]
where $Z_W(V)$ denotes the centralizer of $V$ in $W$.
Let $\widehat{W}$ be the right-angled Coxeter group with generating set
$\{\tau_U\mid U\in\mathcal S\}$ and defining relations
\[
\tau_U^2=1 \ \ (U\in\mathcal S),
\qquad
[\tau_U,\tau_V]=1 \ \ \text{whenever } m_{UV}=2.
\]
Note that $m_{UV}=2$ if and only if $U$ and $V$ are orthogonal or nested.

For a Coxeter system $(W,S)$, Yu constructed a monomorphism $PC_W\hookrightarrow \widehat{W}$, which we denote by
\[
\Phi:PC_W\hookrightarrow \widehat{W}.
\]
Concretely, if $x\in PC_W$ is represented by a word
\[
x=\sigma_{I_1}\sigma_{I_2}\cdots \sigma_{I_m},
\]
and we set $u_0:=1$ and
\[
u_{k-1}:=\omega_{I_1}\omega_{I_2}\cdots \omega_{I_{k-1}}\in W
\qquad (1\le k\le m),
\]
then
\begin{align*}
   \Phi(x)
    &=
    \prod_{k=1}^m \tau_{\,u_{k-1}W_{I_k}u_{k-1}^{-1}}\\
    &=
\tau_{W_{I_1}}
\ \tau_{\omega_{I_1}W_{I_2}\omega_{I_1}}
\ \tau_{(\omega_{I_1}\omega_{I_2})W_{I_3}(\omega_{I_2}\omega_{I_1})}
\ \cdots\
\tau_{(\omega_{I_1}\cdots \omega_{I_{m-1}})\,W_{I_m}\,(\omega_{I_{m-1}}\cdots \omega_{I_1})}. 
\end{align*}
See \cite[Lemma~2.7]{Yu} for the details.

\subsection{Van Kampen diagrams}
A {\em map} $M$ is a finite $2$-dimensional cell complex embedded in $\mathbb{R}^2$ consisting of a finite collection of $0$-cells ({\em vertices}), $1$-cells ({\em edges}) and $2$-cells ({\em faces}).
If $M \subseteq \mathbb{R}^2$, then the boundary of $M$ is denoted by $\partial M$.
If $\pi$ is a face of $M$, we denote the boundary of $\pi$ by $\partial \pi$.
If a vertex (or an edge) of $M$ is not a vertex (or an edge, respectively) of $\partial M$, then it is called the {\em interior} vertex (or the {\em interior} edge, respectively).
A word is {\em reduced} if it has no subword of the form $zz^{-1}$ or $z^{-1}z$, and it is {\em cyclically reduced} if each of its cyclic permutations is reduced.
Now we discuss van Kampen diagrams.	

\begin{definition}[$\mathcal{R}$-diagrams, van Kampen Diagrams]
Let $G=\langle X \svert \mathcal{R} \rangle$ be a group presentation, where $\mathcal{R}$ is {\em symmetrized}, meaning that all its elements are cyclically reduced and it is closed under taking inverses and cyclic permutations.
Then a map $\Delta$ is called an {\em $\mathcal{R}$-diagram} if it satisfies the following properties.
\begin{enumerate}[\rm (i)]
	\item Let $\Delta_{(1)}$ denote the set of all oriented edges of $\Delta$.
	Each oriented edge of $\Delta$ is labeled by an element $x \in X$, and the oppositely oriented edge is labeled by $x^{-1}$, and hence there exists a labeling function $\Lab: \Delta_{(1)} \rightarrow X \cup X^{-1}$ such that $\Lab(e^{-1})=\Lab(e)^{-1}$ for all $e \in \Delta_{(1)}$.
	\item For each face $\pi$ of $\Delta$, every choice of a vertex on
    $\partial\pi$, and either orientation of $\partial\pi$, the label read
    along $\partial\pi$ from that vertex belongs to $\mathcal R$.
\end{enumerate}
If an $\mathcal{R}$-diagram $\Delta$ is connected and simply connected, then it is called a {\em van Kampen diagram} over $G=\langle X \svert \mathcal{R} \rangle$.
For any $\mathcal{R}$-diagram $\Delta$, the number of faces of $\Delta$ is called the {\em area} of $\Delta$, denoted by $Area(\Delta)$.	
\end{definition}

Let $\Delta$ be a van Kampen diagram.
Two distinct faces $\pi_1$ and $\pi_2$ of $\Delta$ form a
\emph{reducible pair} if they share an edge $e$ and, after reading
$\partial\pi_1$ from $e$ in one direction and $\partial\pi_2$ from
$e^{-1}$ in the opposite direction, the resulting boundary words are
inverse to each other letter for letter.
If there is no reducible pair in $\Delta$, then $\Delta$ is called a {\em reduced van Kampen diagram}.

The following fundamental result is van Kampen's Lemma.

\begin{theorem}[van Kampen's Lemma]
\label{thm:van_Kampen_Lemma}
Suppose $G=\langle X \svert \mathcal{R} \rangle$, where $\mathcal R$ is a
symmetrized set of cyclically reduced words.
Let $u$ be a cyclically reduced word in $X$. Then $u$ represents the identity element in $G$
if and only if
there exists a reduced van Kampen diagram $\Delta$ over $G=\langle X \svert \mathcal{R} \rangle$
such that
$(\Lab(\partial \Delta)) \equiv (u)$.
\end{theorem}

\section{Bounding the rank of free abelian subgroups of generalized cactus groups}
\label{sec:maximal_rank}

In this section, we prove the only if part of Theorem~\ref{thm:maximal_rank}, which determines
the maximal rank of free abelian subgroups of $C_W$.
We first show that such a maximal rank exists.

\begin{lemma}
\label{lem:rank_bounded}
Let $(W, S)$ be a finite Coxeter system. Then the ranks of free abelian subgroups of $C_W$ are bounded above.
\end{lemma}

\begin{proof}
Since $W$ is finite, $C_W$ acts properly by isometries on a finite-dimensional $\text{CAT(0)}$ cube complex.
Then, by~\cite[Corollary~4.11]{Wright}, we have $\operatorname{asdim}(C_W) < \infty$.

Let $A$ be a free abelian subgroup of rank $p$.
Then $A \cong \mathbb{Z}^p$, and hence $\operatorname{asdim}(A)=p$.
By~\cite[Theorem~2.11]{Dranishnikov_Smith}, the asymptotic dimension of a
countable group is the supremum of the asymptotic dimensions of its finitely
generated subgroups.
Hence, $p=\operatorname{asdim}(A) \le \operatorname{asdim}(C_W) < \infty$.
Therefore, the ranks of free abelian subgroups of $C_W$ are bounded above.
\end{proof}

The only if part of Theorem~\ref{thm:maximal_rank} follows from the following theorem:

\begin{theorem}
\label{thm:maximal_rank_only_if}
Let $(W,S)$ be a finite Coxeter system with connected components $S_1,\dots,S_k$, and let $n_i=|S_i|$ for each $i$. 
If the generalized cactus group $C_W$ contains a free abelian subgroup of rank $r$, then $r\leq \sum_{i=1}^k \lfloor n_i/2 \rfloor$.
\end{theorem}

\begin{proof}
Since $W$ is finite, $PC_W$ has finite index in $C_W$. Hence, the maximal ranks of the free abelian subgroups of $PC_W$ and $C_W$ are the same.

Let $\mathscr{C}_W$ be the Cayley graph with respect to the generating set
$\{\sigma_I\}_{I \in \mathcal{I}}$.
By \cite[Theorem~4.12]{Genevois_3}, $\mathscr{C}_W$ is quasi-median.
Since all factor groups from the corresponding quandle product structure are isomorphic to $\mathbb Z_2$, it follows from \cite[Lemma~4.13]{Genevois_3} that $\mathscr{C}_W$ is triangle-free.
Hence, by \cite[Theorem~4.2]{Genevois_3}, it is a median graph.
Let $\mathscr{A}_W$ be a cube-completion of $\mathscr{C}_W$. Since $\mathscr{C}_W$ is a median graph, $\mathscr{A}_W$ is a $\text{CAT(0)}$ cube complex.
Moreover, since $C_W$ acts properly and cocompactly on a median graph, this action of $C_W$ on $\mathscr{C}_W$ extends a cubical action on $\mathscr{A}_W$, and remains proper and cocompact.

By Lemma~\ref{lem:rank_bounded}, there exists a free abelian subgroup
$H$ of $C_W$ of maximal rank. Let $m$ denote its rank.
By Theorem~\ref{thm:wise_woodhouse}, $H$ acts properly and cocompactly on a
convex subcomplex of $\mathscr{A}_W$ which is isomorphic to $\prod_{i=1}^m C_i$, where $C_i$ is a quasiline for each $i=1, \dots, m$. 

Since $S$ is finite, the collection $\mathcal{I}$ of nontrivial subsets of $S$ is finite.
Hence, each vertex of $\mathscr{C}_W$ is incident to $|\mathcal{I}|$ edges.
Since $\mathscr{A}_W$ is a cube-completion of $\mathscr{C}_W$, no new vertices or edges are added.
Therefore, if $\mathscr{A}_W$ contains a $k$-cube, then there are $k$ distinct edges of $\mathscr{C}_W$ incident to each vertex of this cube.
Thus, $k \le |\mathcal{I}|$, and hence $\dim \mathscr{A}_W \le |\mathcal{I}| < \infty$.
It follows that $\prod_{i=1}^m C_i$ is finite-dimensional, and therefore, each $C_i$ is finite-dimensional.

\begin{lemma}\label{lem:quasiline}
For each $i$, if every two edges of $C_i$ sharing a vertex span a square, then $C_i$ is itself a cube.
\end{lemma}

\begin{proof}[Proof of Lemma~\ref{lem:quasiline}]
Let $Q$ be a {\em maximal cube} of $C_i$, that is, a cube whose dimension is the largest among all cubes contained in $C_i$.
Such $Q$ always exists since $C_i$ is finite-dimensional.
Set $k=\dim Q$ and choose a vertex $v$ of $Q$.
First, we claim that every edge of $C_i$ incident to $v$ is contained in $Q$.
Assume that there exists an edge $e$ of $C_i$ incident to $v$ such that $e$ is not contained in $Q$. Let $e_1, \dots, e_k$ be the edges of $Q$ incident to $v$.
Since $C_i$ is $\text{CAT(0)}$, $\operatorname{Lk}_{\mathscr C_i}(v)$ is a flag simplicial complex. 
Therefore, the vertices of $\operatorname{Lk}_{\mathscr C_i}(v)$ corresponding to $e, e_1, \dots, e_k$ span a $k$-simplex, and hence the corresponding $k+1$ edges are contained in a common $(k+1)$-cube, which contradicts to the maximality of $Q$.
Thus, every edge incident to a vertex of $Q$ is contained in $Q$.

Next, we show that $C_i=Q$. Since $C_i$ is quasi-isometric to $\mathbb{R}$, it is connected.
If $C_i \neq Q$, then there exists a path joining a vertex to $Q$ to a vertex outside $Q$. Hence there exists an edge of this path with one endpoint in $Q$ and the other outside $Q$. Then this edge is incident to a vertex of $Q$, and therefore, it must be contained in $Q$, which is impossible.
Thus, $C_i=Q$, so $C_i$ itself is a cube.
\end{proof}

By Lemma~\ref{lem:quasiline}, $C_i$ contains two distinct edges $e_i^-$ and $e_i^+$ which share a vertex $v_i$ but do not span a square. Otherwise, $C_i$ would be a single cube, and hence bounded, contradicting the fact that $C_i$ is a quasiline.
Now, let $L$ be the link of the vertex $v=(v_1, \dots, v_m)$ in $\prod_{i=1}^m C_i$. Since the link of a vertex in a product cube complex is the join of the links of the corresponding vertices, it follows that 
$L$ contains a full subcomplex $\{e_1^-, e_1^+\} * \cdots * \{e_m^-, e_m^+\}$ that is the join of $m$ copies of $\mathbb{S}^0$.
Since $\prod_{i=1}^m C_i$ is isomorphic to a subcomplex of $\mathscr{A}_W$,
we may regard as $v$ as a vertex of $\mathscr{A}_W$.

Now, let $\Gamma$ be the Coxeter diagram of $(W, S)$, 
and let $\Gamma_1, \dots, \Gamma_k$ be connected components of $\Gamma$ with vertex sets $S_1, \dots, S_k$, respectively, that is, for each $i$, $\Gamma_i$ is the full subdiagram of $\Gamma$ spanned by $S_i$.
Recall that $\mathcal{I}$ is the collection of subsets $I \subset S$ which are spherical and irreducible. 
Since $W$ is finite, every standard parabolic subgroup $W_I$ is finite, and hence every subset $I\subseteq S$ is spherical.
Moreover, $W_I$ is irreducible if and only if the full subdiagram of $\Gamma$ spanned by $I$ is connected. 
Let $K_\Gamma$ denote $\operatorname{Lk}_{\mathscr A_W}(v)$.
Then the vertices of $K_\Gamma$ correspond to the edges of $\mathscr A_W$ incident to that vertex. 
Since the $1$-skeleton of $\mathscr A_W$ is $\mathscr C_W$, 
these edges are naturally labeled by the generators $\sigma_I$ for $I \in \mathcal{I}$. Therefore, the vertices of $K_\Gamma$
are naturally identified with the connected full subdiagrams of $\Gamma$.
In particular, each such subdiagram is contained in a unique connected
component $\Gamma_i$.
By Lemma~\ref{lem:squares}, two vertices of $K_\Gamma$ are
adjacent if and only if the corresponding subdiagrams are nested or
orthogonal.

For each $i$, let $K_{\Gamma_i}$ be the full subcomplex of $K_\Gamma$
spanned by the connected subdiagrams contained in $\Gamma_i$.
Since connected subdiagrams contained in distinct components of
$\Gamma$ are orthogonal, we have
\[
K_\Gamma=K_{\Gamma_1}*\cdots*K_{\Gamma_k}.
\]

The following lemma gives the upper bound for each connected component of $\Gamma$.

\begin{lemma}
\label{lem:tree_upper_bound}
For each $i=1, \dots, k$, suppose that $K_{\Gamma_i}$ contains the join of $q$ copies of $\mathbb{S}^0$ as a full subcomplex, then $q \le \lfloor n_i/2 \rfloor$.
\end{lemma}

\begin{proof}[Proof of Lemma~\ref{lem:tree_upper_bound}]
Since $\Gamma_i$ is a connected component of the Coxeter diagram of a finite Coxeter system, it is a finite tree with $n_i$ vertices.
We proceed by induction on $n_i$.
If $n_i=1$, then $K_{\Gamma_i}$ has only one vertex. Hence it contains no copy of $\mathbb{S}^0$, and therefore, $q=0$. Thus, $0=q \le \lfloor 1/2 \rfloor$.
If $n_i=2$, then $K_{\Gamma_i}$ has three vertices since the nonempty connected subtrees of $\Gamma_i$ are $\{s\}$, $\{t\}$ and $\{s, t\}$.
Since the join of $q$ copies of $\mathbb{S}^0$ has $2q$ vertices, we have $2q \le 3$.
Hence $q\le 1$, and therefore $q \le 1 = \lfloor 2/2 \rfloor$.

Suppose that $q\leq \lfloor |\Lambda|/2 \rfloor$ holds for $|\Lambda|<n_i$, and assume that $n_i \ge 3$.
Let $M$ be a full subcomplex of $K_{\Gamma_i}$ isomorphic to the join of $q$ copies of $\mathbb{S}^0$. 
Moreover, let $\{P_1, Q_1, \dots, P_q, Q_q\}$ be the vertex set of $M$, where $P_j$ and $Q_j$ are not adjacent for each $j=1, \dots, q$, and every vertex in $\{P_j, Q_j\}$ is adjacent to every vertex in $\{P_\ell, Q_\ell\}$ if $j \neq \ell$.
Suppose that all vertices of $M$ are contained in a proper connected subtree $\Lambda$ of $\Gamma_i$. Then, $|\Lambda|<n_i$.
Since the connected subtrees contained in $\Lambda$ form the vertex
set of $K_\Lambda$, and since two vertices are adjacent
in either complex if and only if the corresponding subtrees are nested
or orthogonal, the full subcomplex of $K_{\Gamma_i}$ spanned by the connected subtrees contained in $\Lambda$ is naturally isomorphic to $K_\Lambda$.
Since all vertices of $M$ are contained in $\Lambda$ and $M$ is full
in $K_{\Gamma_i}$, $M$ is also a full subcomplex of this copy of
$K_\Lambda$.
Hence, by the induction hypothesis,
\[
q \le \lfloor |\Lambda|/2 \rfloor \le \lfloor n_i/2 \rfloor.
\]
Thus we may assume that the vertices of $M$ are not all contained in any
proper connected subtree of $\Gamma_i$.
Choose $P_1$ among the vertices of $M$ so that $|P_1|$ is maximal.
Since $P_1$ and $Q_1$ are not adjacent, they are neither nested nor orthogonal.
Let $j \ge 2$, and let $R \in \{P_j, Q_j\}$. Then since $R$ is adjacent to $P_1$, the subtrees $R$ and $P_1$ are nested or orthogonal.
By the maximality of $|P_1|$, $P_1 \not\subset R$. Hence either $R \subset P_1$ or $R$ and $P_1$ are orthogonal.
Similarly, since $R$ is adjacent to $Q_1$, the subtrees $R$ and $Q_1$ are nested or orthogonal. If $Q_1 \subset R$, then since either $R \subset P_1$ or $R$ and $P_1$ are orthogonal, it follows that $Q_1 \subset P_1$ or $Q_1$ and $P_1$ are orthogonal, which is impossible.
Hence either $R \subset Q_1$ or $R$ and $Q_1$ are orthogonal.
Therefore, for $j \ge 2$, every $R \in \{P_j, Q_j\}$ is contained in one of the connected components of the following subgraphs:
\begin{enumerate}[\rm (i)]
\item $P_1 \cap Q_1$,
\item the subgraph of $P_1$ spanned by the vertices which are not in $Q_1$ and
are not adjacent to any vertex of $Q_1$,
\item the subgraph of $Q_1$ spanned by the vertices which are not in $P_1$ and
are not adjacent to any vertex of $P_1$,
\item the subgraph of $\Gamma_i$ spanned by the vertices which are neither in
$P_1 \cup Q_1$ nor adjacent to any vertex of $P_1 \cup Q_1$.
\end{enumerate}

Let $\mathcal{E}$ be the collection of all connected components of (i)--(iv).
If $P_j$ and $Q_j$, with $j \ge 2$, are contained in two different members of $\mathcal{E}$, then $P_j$ and $Q_j$ are orthogonal, and hence they are adjacent in $K_{\Gamma_i}$, which is impossible. Thus, each pair $\{P_j, Q_j\}$, with $j \ge 2$, is contained in the single member of $\mathcal{E}$.
For $E \in \mathcal{E}$, let $m_E$ be the number of pairs contained in $E$.
Then
\[
q=1+\sum_{E \in \mathcal{E}} m_E,
\]
where the term $1$ corresponds to the pair $\{P_1, Q_1\}$.

For each $E \in \mathcal{E}$, let $N_E$ be the subcomplex of $K_{\Gamma_i}$ spanned by the connected subtrees contained in $E$.
Moreover, let $M_E$ be the subcomplex of $M$ spanned by the vertices belonging to the pairs contained in $E$. Then $M_E \cong (\mathbb{S}^0)^{*m_E}$.
We claim that $M_E$ is a full subcomplex of $N_E$.
To see this, let $\Delta$ be a simplex of $N_E$ whose vertices lie in $M_E$.
Since $N_E \subset K_{\Gamma_i}$, $\Delta$ is a simplex of
$K_{\Gamma_i}$ whose vertices lie in $M_E\subset M$. Since $M$ is full in
$K_{\Gamma_i}$, $\Delta$ is a simplex of $M$. Finally, since $M_E$ is the full
subcomplex of $M$ spanned by these vertices, it follows that $\Delta$ is a simplex
of $M_E$. Hence $M_E$ is full in $N_E$.

Since $E$ is a proper connected subtree
of $\Gamma_i$,
by the induction hypothesis, $m_E \le  \lfloor |E|/2 \rfloor$, and hence
\[
q \le 1+\sum_{E \in \mathcal{E}} \lfloor |E|/2 \rfloor \le  1+ \sum_{E\in\mathcal{E}} |E|/2.
\]

It remains to estimate $\sum_{E\in\mathcal{E}} |E|$.
Since $P_1$ and $Q_1$ are not nested, $P_1 \not\subset Q_1$ and $Q_1 \not\subset P_1$. Hence 
\[
P_1 \setminus Q_1 \neq \emptyset \quad \text{and} \quad
Q_1 \setminus P_1 \neq \emptyset.
\]
Moreover, since $P_1$ and $Q_1$ are not orthogonal, either $P_1 \cap  Q_1 \neq \emptyset$, or there is an edge of $\Gamma_i$ joining a vertex of $P_1$ to a vertex of $Q_1$. In either case, there exists a vertex in $P_1 \setminus Q_1$ which is adjacent to a vertex of $Q_1$, and there exists a vertex in $Q_1 \setminus P_1$ which is adjacent to a vertex of $P_1$. These two vertices are distinct.
By the definition of $\mathcal{E}$, no member of $\mathcal{E}$ contains
such a boundary vertex. 
Therefore, at least two vertices of $\Gamma_i$ are not contained in the union of the members of $\mathcal{E}$. 
Since $\Gamma_i$ has $n_i$ vertices, we obtain
\[
\sum_{E\in \mathcal E}|E|\le n_i-2.
\]
Thus,
\[
q \le 1+\frac{n_i-2}{2}=\frac{n_i}{2},
\]
and since $q$ is an integer, we have $q \le \lfloor n_i/2 \rfloor$.
\end{proof}

Since the subcomplex of $\mathscr{A}_W$ isomorphic to $\prod_{i=1}^m C_i$ is convex in $\mathscr A_W$, 
by~\cite[Theorem~1.1]{Sakai_Sakuma}, $\{e_1^-, e_1^+\} * \cdots * \{e_m^-, e_m^+\}$ is a full subcomplex of $K_\Gamma$.
For each $j$, the two vertices $e_j^-$ and $e_j^+$ are not adjacent.
Hence, the corresponding connected subdiagrams must be contained in
the same connected component of $\Gamma$, since subdiagrams contained
in distinct components are orthogonal and therefore adjacent in
$K_\Gamma$.
For each $i$, let $m_i$ be the number of pairs
$\{e_j^-,e_j^+\}$ contained in $\Gamma_i$. Then $m=\sum_{i=1}^k m_i$.
The vertices belonging to these $m_i$ pairs span a full subcomplex of
$K_{\Gamma_i}$ isomorphic to the join of $m_i$ copies of $\mathbb{S}^0$.
Therefore, by Lemma~\ref{lem:tree_upper_bound}, $m_i\le \lfloor n_i / 2 \rfloor$, and thus,
\[
m\le
\sum_{i=1}^k
\left\lfloor\frac{n_i}{2}\right\rfloor.
\]

We now prove that $m=\sum_{i=1}^k \lfloor n_i/2\rfloor$.
Fix a connected component $\Gamma_i$ of $\Gamma$, and let $n$ be the
number of vertices of $\Gamma_i$.
We distinguish two cases according to whether $\Gamma_i$ is a path or not.
First, we consider the case where the Coxeter diagram is a path.
This includes the following types:

\subsection{Types $A_n$, $B_n$, $F_4$, $H_3$, $H_4$ and $I_2(p)$}
\label{subsec:path_types}
In this case, let $(W,S)$ be an irreducible finite Coxeter system whose Coxeter diagram is a path. We label the vertices of $\Gamma_i$ as 
$s_1, \dots, s_n$ so that $s_i$ is adjacent to $s_{i+1}$ for $i=1, \dots, n-1$.
Note that $n=4$ for $F_4$ and $H_4$, $n=3$ for $H_3$, and $n=2$ for $I_2(p)$.

For $1 \le a \le b \le n$, set $I_{a,b}:=\{s_a, \dots, s_b\} \subseteq S$.
Since $\Gamma_i$ is a path, the irreducible subsets of $S$
are $I_{a,b}$ with $1 \le a \le b \le n$.
Thus, under the natural identification of the vertices of $K_{\Gamma_i}$ with the connected full subdiagrams of $\Gamma_i$,
the vertices of $K_{\Gamma_i}$ are naturally identified with the
intervals $I_{a,b}$. 
By Lemma~\ref{lem:squares}, two vertices $I_{a,b}$ and $I_{c,d}$ are
adjacent in $K_{\Gamma_i}$ if and only if the corresponding intervals
are nested or orthogonal.
Especially, if $I_{a,b}$ and $I_{c,d}$ are orthogonal, then $b+1<c$ or $d+1<a$.
If $I_{a,b}$ and $I_{c,d}$ are nested, then $a \le c \le d \le b$ or $c \le a \le b \le d$.

For $j=1,\dots,\lfloor n/2\rfloor$, let $P_j=I_{2j,2j}$ and $Q_j=I_{1,2j-1}$.
Then $P_j$ and $Q_j$ are not adjacent in $K_{\Gamma_i}$, since they
are neither nested nor orthogonal.
If $j<\ell$, then $P_j\subset Q_\ell$ and $Q_j\subset Q_\ell$,
while $P_\ell$ is orthogonal to both $P_j$ and $Q_j$.
Hence, every vertex in $\{P_j,Q_j\}$ is adjacent to every vertex in
$\{P_\ell,Q_\ell\}$.
Therefore, the full subcomplex of $K_{\Gamma_i}$ spanned by
\[
P_1,Q_1,\dots,
P_{\lfloor n/2\rfloor},Q_{\lfloor n/2\rfloor}
\]
is isomorphic to the join of $\lfloor n/2\rfloor$ copies of $\mathbb{S}^0$.
Thus, $K_{\Gamma_i}$ contains the join of
$\lfloor n/2\rfloor$ copies of $\mathbb{S}^0$ as a full subcomplex.

\subsection{Types $D_n$, $E_6$, $E_7$ and $E_8$}
\label{subsec:non_path_types}
We now consider the remaining possibilities for the fixed connected
component $\Gamma_i$. Suppose that $\Gamma_i$ is not a path.
Then $\Gamma_i$ is of type $D_n$, $E_6$, $E_7$, or $E_8$.
We show that $K_{\Gamma_i}$ contains the join of
$\lfloor n/2\rfloor$ copies of $\mathbb{S}^0$ as a full subcomplex.

For type $D_n$, assume that $\Gamma_i$ is of type $D_n$.
Recall that we label its vertices as $s_1,\dots,s_n$ so that $s_{n-2}$ is the
unique vertex of degree $3$, and both $s_{n-1}$ and $s_n$ are adjacent
to $s_{n-2}$.

First, suppose that $n=2m+1$.
For $1\le j\le m$, consider the pairs
\[
P_j=\{s_{2j}\}, \qquad Q_j=\{s_1,\dots,s_{2j-1}\}.
\]
In each pair, the two subdiagrams are neither nested nor orthogonal,
and hence they are not adjacent in $K_{\Gamma_i}$.
On the other hand, if two subdiagrams belong to different pairs, then
they are nested or orthogonal, and hence they are adjacent in
$K_{\Gamma_i}$.
Therefore, these pairs span a full subcomplex isomorphic to the join of
$m=\lfloor n/2\rfloor$ copies of $\mathbb{S}^0$.

Now suppose that $n=2m$. Consider the same pairs
\[
\begin{aligned}
P_j&=\{s_{2j}\}, \qquad Q_j=\{s_1,\dots,s_{2j-1}\}, \quad \text{for $1 \le j \le m-1$}, \\
P_m&=\{s_1,\dots,s_{n-1}\},
\qquad
Q_m=\{s_1,\dots,s_{n-2},s_n\}.
\end{aligned}
\]
Again, the two subdiagrams in each pair are neither nested nor
orthogonal. Moreover, for $j<m$, both $P_j$ and $Q_j$ are contained
in $P_m$ and $Q_m$. Thus, vertices belonging to distinct pairs are
adjacent. Therefore, these pairs span a full subcomplex isomorphic to
the join of $m=\lfloor n/2\rfloor$ copies of $\mathbb{S}^0$.

\medskip 
We next consider types $E_6$, $E_7$, and $E_8$. Recall that we use the numbering
in which $s_3$ is the unique vertex of degree $3$ and $s_4$ is the
leaf adjacent to $s_3$.

For types $E_6$ and $E_7$, consider the three pairs
\[
\begin{aligned}
P_1&=\{s_2\}, & Q_1&=\{s_1\}, \\
P_2&=\{s_4\}, & Q_2&=\{s_1,s_2,s_3\}, \\
P_3&=\{s_6\}, & Q_3&=\{s_1,s_2,s_3,s_4,s_5\}.
\end{aligned}
\]
For each $i$, $P_i$ and $Q_i$ are not adjacent, while vertices
belonging to distinct pairs are adjacent. Therefore, these pairs span
a full subcomplex isomorphic to the join of three copies of $\mathbb{S}^0$. Since $3=\lfloor 6/2 \rfloor=\lfloor 7/2 \rfloor$, this gives the required subcomplex for types $E_6$ and $E_7$.

For type $E_8$, consider the above three pairs together with
\[
P_4=\{s_1,\dots,s_7\}, \qquad Q_4=\{s_8\}.
\]
Again, $P_4$ and $Q_4$ are not adjacent, while each of $P_4$ and $Q_4$
is adjacent to every vertex in the first three pairs.
Therefore, these four pairs span a full subcomplex isomorphic to the
join of $4=\lfloor 8/2 \rfloor$ copies of $\mathbb{S}^0$.
\end{proof}

\section{Finding maximal free abelian subgroups of generalized cactus groups}
\label{sec:finding_subgroups}

Let $(W,S)$ be a finite Coxeter system with connected components $S_1,\dots,S_k$ of $S$, and let $n_i=|S_i|$ for each $i$.
The goal of this section is to construct free abelian subgroups of the generalized cactus group $C_W$ with the expected maximal rank $m=\sum_{i=1}^k \lfloor n_i/2 \rfloor$.
In fact, the subgroup that we construct lies in $PC_W$.

We first construct explicit elements and verify their pairwise commutativity in $C_W$ using the defining relations of $C_W$.
The defining presentation of $C_W$ makes van Kampen diagrams a natural tool for verifying that explicitly given words represent the identity. Accordingly, in Lemma~\ref{lem:An-commutativity}, we first use the defining relations of $C_W$ to establish the pairwise commutativity of the elements that we construct. On the other hand, properties such as having infinite order and being independent are less straightforward to establish directly from the presentation.
After showing that these elements lie in $PC_W$ in Lemma~\ref{lem:An-ai-in-PCW}, we then apply Yu's embedding $\Phi$ and explicitly determine the images $\Phi(a_i)$ in Lemma~\ref{lem:An-PQR}. We use these descriptions in Lemma~\ref{lem:An-infinite-order} to show that each $a_i$ has infinite order.
Finally, in Proposition~\ref{thm:An-free-abelian}, we use Yu's embedding to prove the independence of $a_1, \dots, a_m$. Combined with the preceding lemmas, this shows that $\langle a_1, \dots, a_m \rangle$ is a free abelian subgroup of $PC_W$, and hence of $C_W$.

\begin{remark}
\label{rmk:irreducible}
It is enough to construct the desired free abelian subgroups under the assumption that $(W, S)$ is irreducible.
The reason is as follows: suppose that $S=S_1 \sqcup \cdots \sqcup S_k$, where $S_i$ is an irreducible component for each $i=1, \dots, k$.
If, for each $i$, there is a free abelian subgroup $A_i \le PC_{W_{S_i}}$ of rank $r_i$, then we may regard $A_i$ as a free abelian subgroup of $PC_W$ since $S_i \subseteq S$ induces an embedding $PC_{W_{S_i}} \hookrightarrow PC_W$.

Since different irreducible components are orthogonal, the subgroups $A_1, \dots, A_k$ commute with each other. Moreover, each $A_i$ comes from a different irreducible component. Hence, the subgroup generated by $A_1, \dots, A_k$ is naturally isomorphic to $A_1 \times \cdots \times A_k$.
For each $i=1, \dots, k$, since $A_i \cong \mathbb{Z}^{r_i}$, we have
\[
A_1 \times \cdots \times A_k \cong \mathbb{Z}^{r_1} \times \cdots \times \mathbb{Z}^{r_k}.
\]
Therefore, $A_1\times\cdots\times A_k$ is a free abelian subgroup of rank $\sum_{i=1}^k r_i$.
\end{remark}

By Remark~\ref{rmk:irreducible}, we may assume throughout this section that $(W,S)$ is irreducible. Under this assumption, $S$ has only one connected component. Hence, we write $n=|S|$ and $m=\lfloor n/2 \rfloor$.
Let $S=\{s_1,\ldots,s_n\}$. For $1\le p\le q\le n$, let
\[
I_{p,q}:=\{s_p,\ldots,s_q\}\subseteq S,\qquad \Sigma_{p,q}:=\sigma_{I_{p,q}}\in C_W.
\]

We consider the finite irreducible Coxeter types separately.
In each case below, we take $(W, S)$ to be a Coxeter system of the corresponding type.
We first specify elements $a_1,\ldots,a_m$ of $C_W$.

\medskip
\begin{enumerate}[\noindent \rm (i)]
\item \textbf{Type $A_n$ for $n \ge 1$}: For each $i=1, \dots, m$,
\[
a_i=\Sigma_{2i,2i}\,\Sigma_{1,2i-1}\,\Sigma_{2,2i}\,\Sigma_{1,1}\,\Sigma_{2,2i}\,\Sigma_{1,2i-1}.
\]

\medskip 

\item \textbf{Type $B_n$ for $n \ge 2$}: For each $i=1, \dots, m$,
\[
a_i = (\Sigma_{1,2i-1}\Sigma_{2i,2i})^{4}.
\]

\medskip 

\item \textbf{Type $D_n$ for $n \ge 4$}: 
If $n=2m+1$, then
\[
a_i=\Sigma_{2i,2i}\,\Sigma_{1,2i-1}\,\Sigma_{2,2i}\,\Sigma_{1,1}\,\Sigma_{2,2i}\,\Sigma_{1,2i-1} \quad \text{for each $i=1, \dots, m$}.
\] 
If $n=2m$, then \[
a_i=\Sigma_{2i,2i}\,\Sigma_{1,2i-1}\,\Sigma_{2,2i}\,\Sigma_{1,1}\,\Sigma_{2,2i}\,\Sigma_{1,2i-1}, \quad \text{for each $i=1, \dots, m-1$},
\]
and
\[
J=\{s_1,\dots,s_{n-2},s_n\},\qquad
b=\Sigma_{1,n-1}\,\Sigma_{2,n}\,\sigma_J,\qquad
a_m=b^2.
\]

\medskip 

\item \textbf{Types $E_6$, $E_7$ and $E_8$}: 
\[
\begin{aligned}[t]
    a_1&=(\Sigma_{2,2}\Sigma_{1,1})^3,\\
    a_2&=\Sigma_{4,4}\ \Sigma_{1,3}\ \Sigma_{2,4}\ \Sigma_{1,1}\ \Sigma_{2,4}\ \Sigma_{1,3},\\
    a_3&=\Sigma_{6,6}\ \Sigma_{1,5}\ \sigma_{J}\ \Sigma_{2,6}\ \Sigma_{1,1}\ \Sigma_{2,6}\ \sigma_{J}\ \Sigma_{1,5},
    \qquad J:=\{s_1,s_2,s_3,s_5,s_6\},\\
    a_4&=(\Sigma_{1,7}\Sigma_{8,8})^6 \qquad \text{(only in type $E_8$).}
\end{aligned}
\]

\medskip 

\item \textbf{Type $F_4$}: $a_1=(\Sigma_{1,1}\Sigma_{2,2})^3$ and
$a_2=(\Sigma_{1,3}\Sigma_{4,4})^6$.
\medskip

\item \textbf{Types $H_3$ and $H_4$}: $a_1=(\Sigma_{1,1}\Sigma_{2,2})^5$ and $a_2=(\Sigma_{1,3}\Sigma_{4,4})^{10}$.

\medskip

\item \textbf{Type $I_2(p)$ for $p \ge 5$}: $a_1=(\Sigma_{1,1}\Sigma_{2,2})^p$.
\end{enumerate}

\begin{lemma}\label{lem:An-commutativity}
For each $i \neq j$, $[a_i,a_j]=1$ in $C_W$.
\end{lemma}

\begin{proof}
We prove the result using Theorem~\ref{thm:van_Kampen_Lemma}. In the cases where commutativity requires verification, we construct reduced van Kampen diagrams over the presentation of $C_W$ given in Definition~\ref{def:pcw}. Each $2$-cell in the diagrams corresponds to one of the defining relators of $C_W$, whose validity can be checked directly from the Coxeter diagram and the relevant longest element.
For simplicity, in the figures we omit edge orientations and do not draw the $2$-cells whose boundary labels are of the form $\sigma_I^2$. Since every generator $\sigma_I$ has order $2$ in $C_W$, reversing the orientation of an edge does not change its label.
If $i=j$, clearly, $[a_i, a_j]=1$ in $C_W$. For each type, suppose that $i < j$.

\medskip
\indent (i) {\bf Type $A_n$ for $n \ge 1$}:  
Using the defining relations, we construct a reduced van Kampen diagram $\Delta_1$ over the presentation of $C_W$ given in Definition~\ref{def:pcw} such that $(\Lab(\partial \Delta_1)) \equiv ([a_i,a_j])$ as shown in Figure~\ref{fig:type_A_n}.
\begin{figure}[h]
\begin{center}
	\includegraphics[width=0.85\columnwidth]{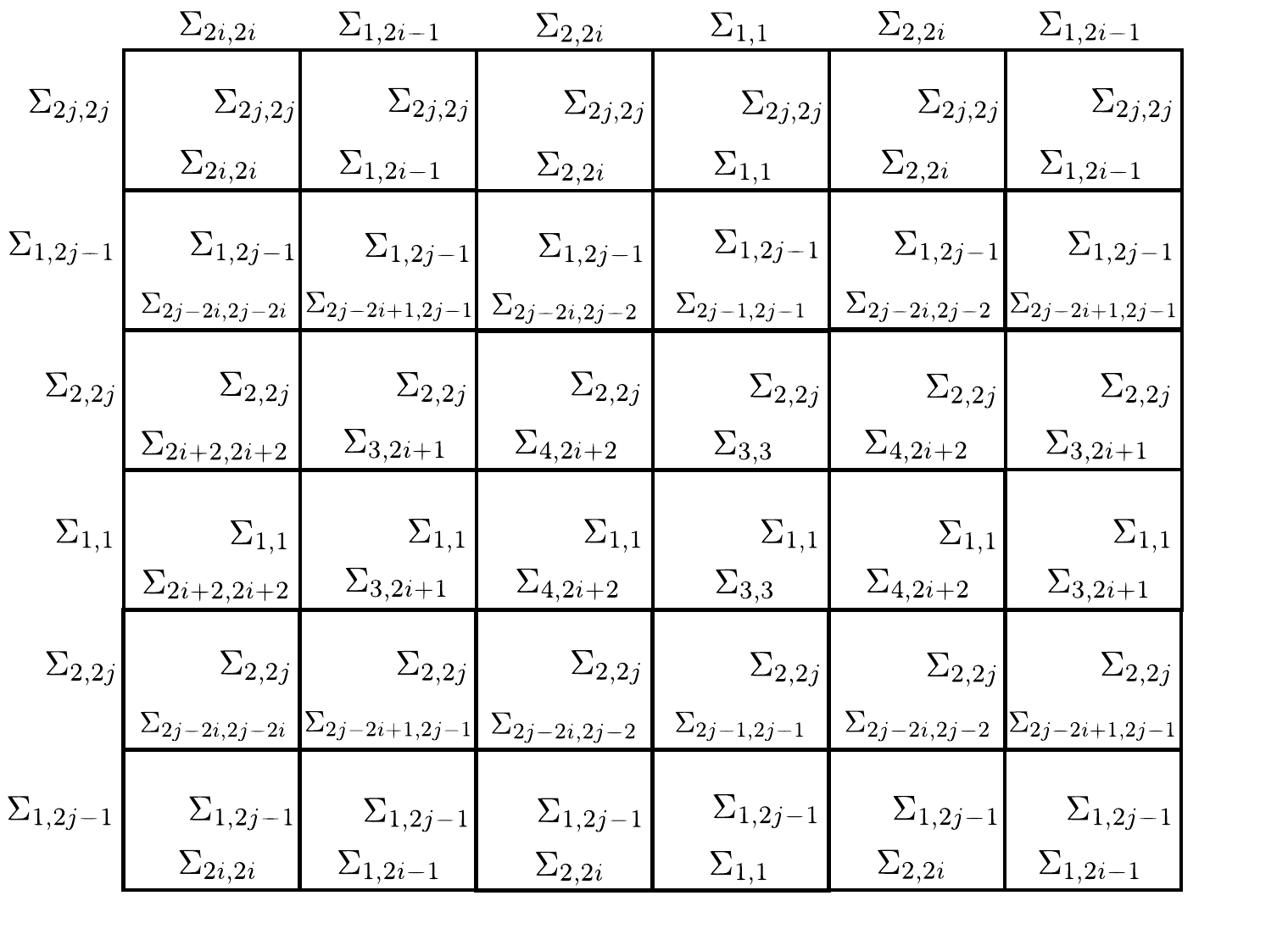}
\end{center}
\vspace{-4mm}
\caption{\label{fig:type_A_n} Type $A_n$ for $n \ge 1$: a van Kampen diagram $\Delta_1$
}						
\end{figure}
Therefore, by Theorem~\ref{thm:van_Kampen_Lemma}, $[a_i, a_j]=1$ in $C_W$.

\medskip
\indent (ii) {\bf Type $B_n$ for $n \ge 2$}:
The following relations hold in $C_W$: 
\[ 
\begin{aligned} 
\Sigma_{1,2i-1}\Sigma_{1,2j-1} &=\Sigma_{1,2j-1}\Sigma_{1,2i-1},& \Sigma_{2i,2i}\Sigma_{1,2j-1} &=\Sigma_{1,2j-1}\Sigma_{2i,2i},\\ \Sigma_{1,2i-1}\Sigma_{2j,2j} &=\Sigma_{2j,2j}\Sigma_{1,2i-1},& \Sigma_{2i,2i}\Sigma_{2j,2j} &=\Sigma_{2j,2j}\Sigma_{2i,2i}. 
\end{aligned} 
\] 
Hence, 
$[\Sigma_{1,2i-1}\Sigma_{2i,2i}, \Sigma_{1,2j-1}\Sigma_{2j,2j}]=1$ in $C_W$, and thus, $[a_i,a_j]=1$ in $C_W$.

\medskip
\indent (iii) {\bf Type $D_n$ for $n \ge 4$}: As long as neither $a_i$ nor $a_j$ is $a_m$ in Type $D_{2m}$, all generators involved lie in the standard parabolic subgroup generated by $\{s_1, \dots, s_{n-1} \}$, which is of type $A_{n-1}$. Hence, $[a_i, a_j]=1$ in $C_W$ for $i, j \in \{1, \dots, m-1\}$.

Suppose that $i \in \{1, \dots, m-1\}$. Since $a_m=b^2$, to prove that $[a_i, a_m]=1$ in $C_W$, we show that $[a_i, b]=1$ in $C_W$.
Then, the defining relations yield a reduced van Kampen diagram $\Delta_3$ over the presentation of $C_W$ given in Definition~\ref{def:pcw} satisfying $(\Lab(\partial \Delta_3)) \equiv ([a_i, b])$, as illustrated in Figure~\ref{fig:type_D_n}.
\begin{figure}[h]
\begin{center}
	\includegraphics[width=0.85\columnwidth]{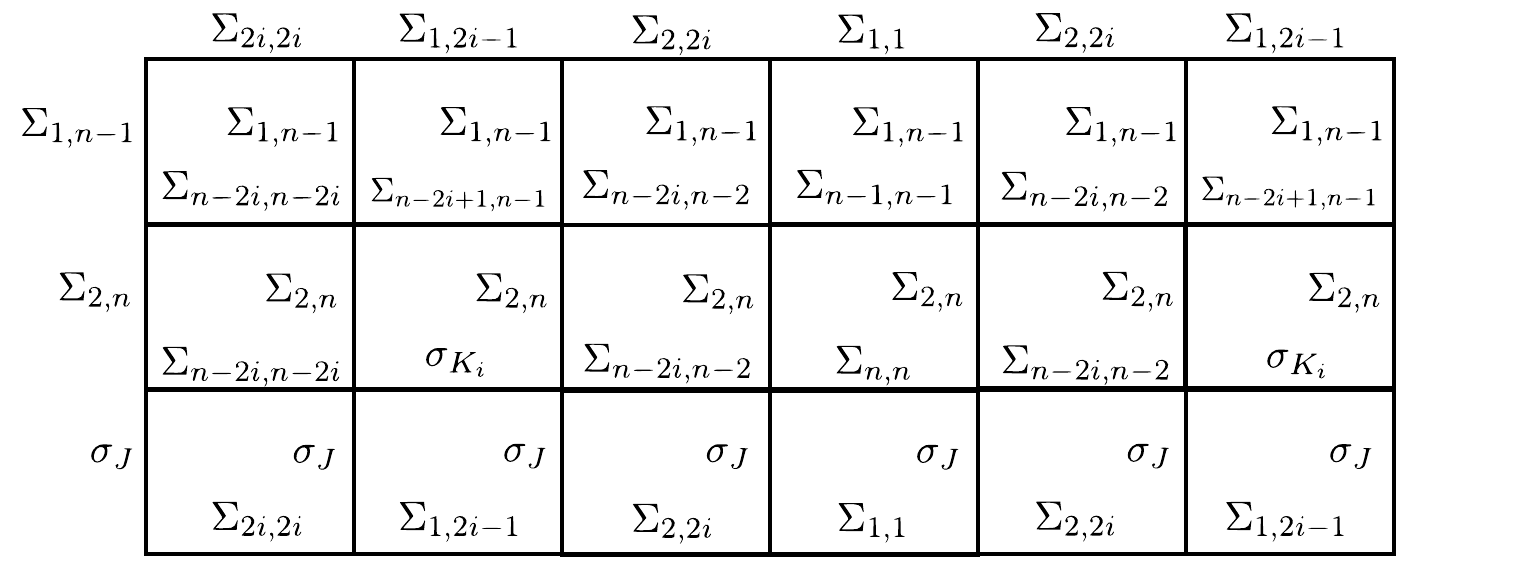}
\end{center}
\vspace{-4mm}
\caption{\label{fig:type_D_n} Type $D_n$ for $n \ge 4$: a van Kampen diagram $\Delta_3$, where $K_i=\{s_{n-2i+1}, \dots, s_{n-2},s_n\}$
}						
\end{figure}

\medskip
(iv) {\bf Types $E_6$, $E_7$, and $E_8$}: In these types, the defining relations used to verify the commutativity of $a_1, a_2$ and $a_3$ are the same, since the corresponding generators lie in the same subsystem of type $E_6$. Thus, it suffices to verify that $[a_1, a_2]=[a_1,a_3]=[a_2,a_3]=1$ in $C_W$, and in type $E_8$, we additionally verify that $a_4$ commutes with each of $a_1, a_2$ and $a_3$.

Using the defining relations of $C_W$, we construct reduced van Kampen diagrams $\Delta_4$ and $\Delta_5$ over the presentation of $C_W$ given in Definition~\ref{def:pcw}, shown from left to right in Figure~\ref{fig:type_E_n}, such that $(\Lab(\partial \Delta_4)) \equiv ([\Sigma_{2,2}\Sigma_{1,1}, a_2])$ and $(\Lab(\partial \Delta_5)) \equiv ([\Sigma_{2,2}\Sigma_{1,1}, a_3])$, respectively.
\begin{figure}[h]
\begin{center}
	\includegraphics[width=0.7\columnwidth]{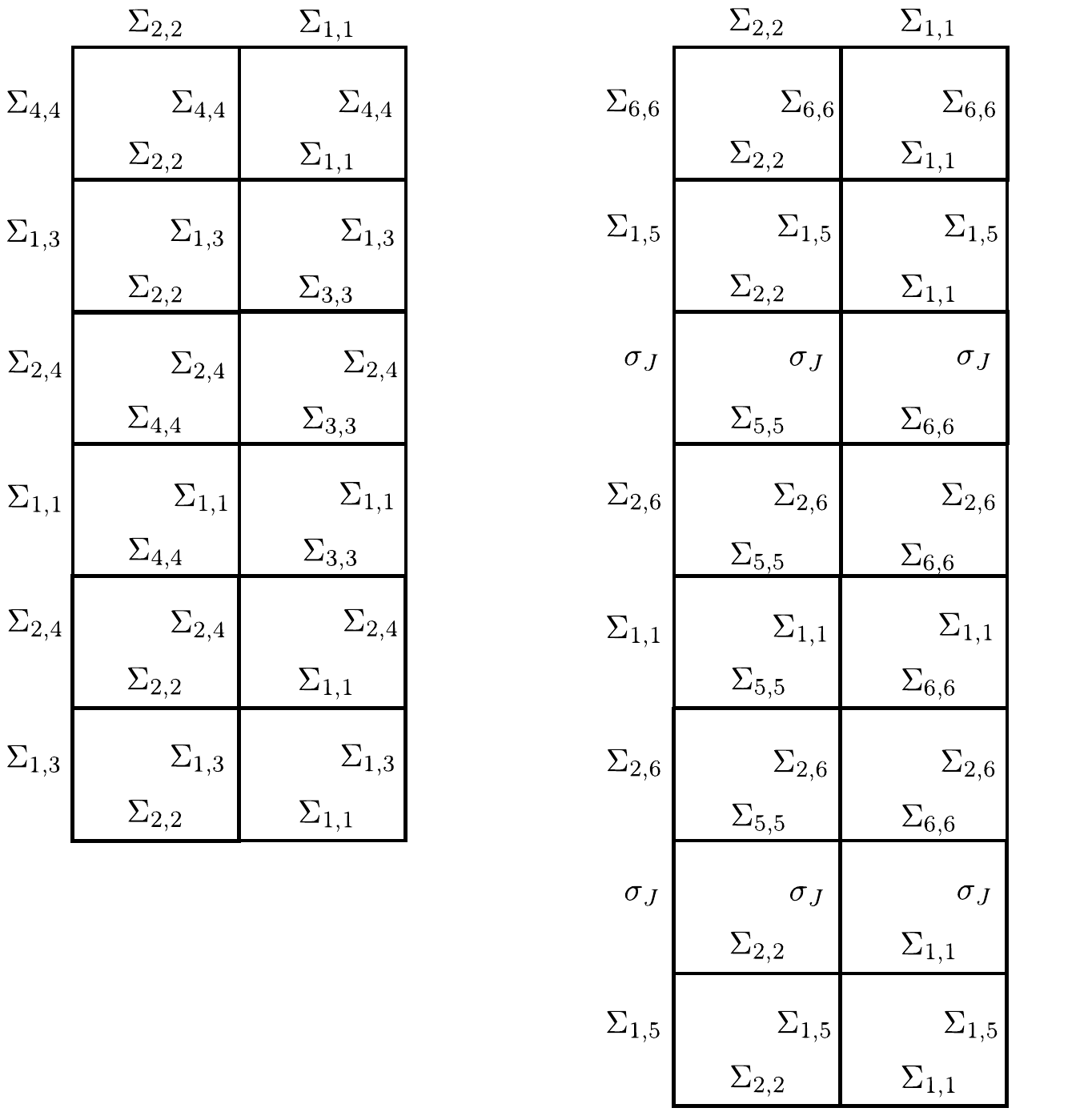}
\end{center}
\vspace{-4mm}
\caption{\label{fig:type_E_n} Types $E_6$ and $E_7$: van Kampen diagrams $\Delta_4$ and $\Delta_5$
}						
\end{figure}
by Theorem~\ref{thm:van_Kampen_Lemma}, $[\Sigma_{2,2}\Sigma_{1,1}, a_2]=1$ and $[\Sigma_{2,2}\Sigma_{1,1}, a_3]=1$ in $C_W$.
Therefore, $[a_1, a_2]=1$ and $[a_1, a_3]=1$ in $C_W$.

To show that $[a_2, a_3]=1$ in $C_W$, we draw a van Kampen diagram $\Delta_6$ over the presentation of $C_W$ given in Definition~\ref{def:pcw} such that $(\Lab(\partial \Delta_6)) \equiv ([a_2, a_3])$ as shown in Figure~\ref{fig:type_E_n_a2a3}.
\begin{figure}[h]
\begin{center}
	\includegraphics[width=0.8\columnwidth]{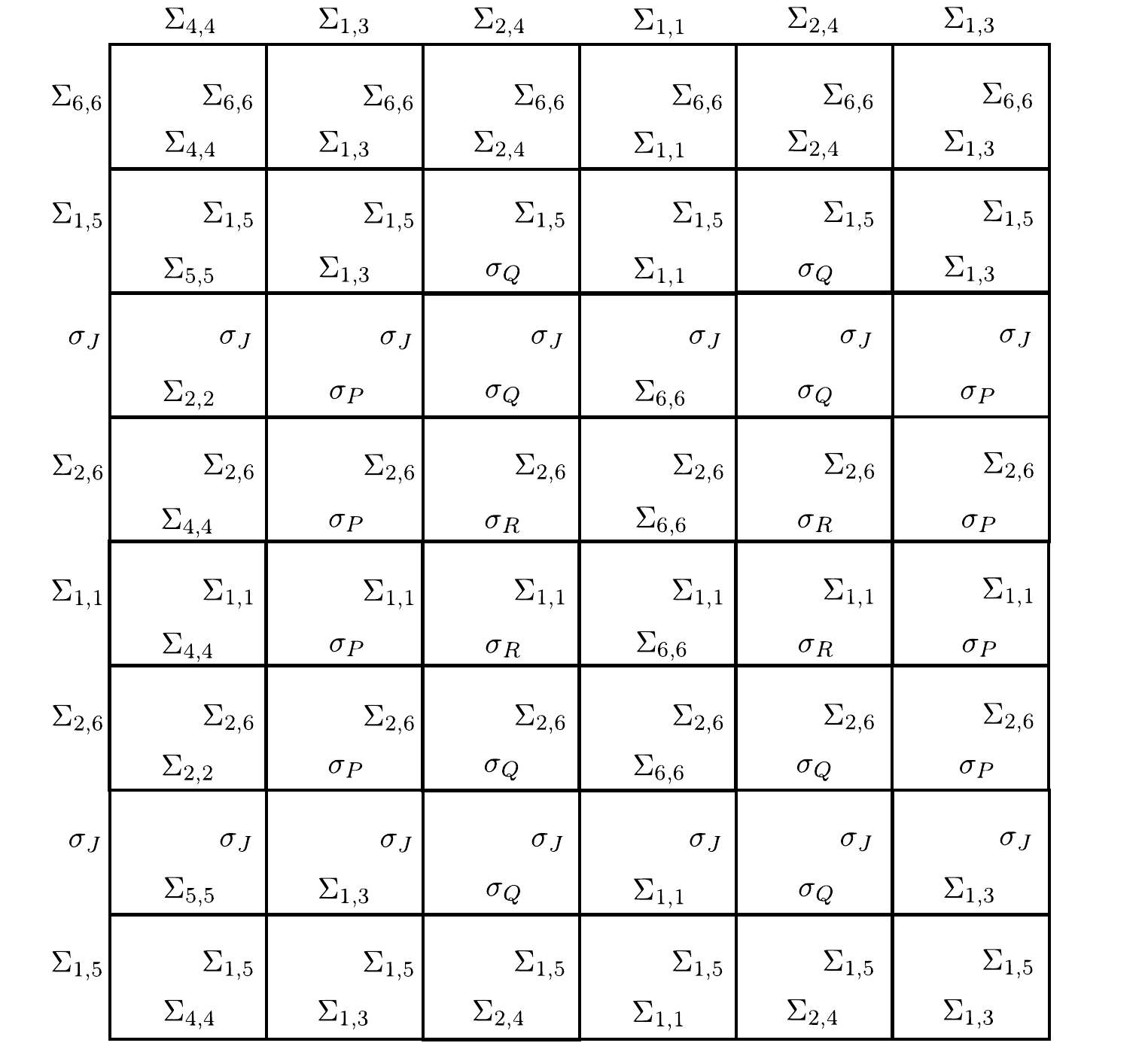}
\end{center}
\vspace{-4mm}
\caption{\label{fig:type_E_n_a2a3} Type $E_8$: a van Kampen diagram $\Delta_6$, where $P=\{s_3, s_5, s_6\}$, $Q=\{s_2, s_3, s_5\}$ and $R=\{s_3, s_4, s_5\}$
}						
\end{figure}
Then, by Theorem~\ref{thm:van_Kampen_Lemma}, we have $[a_2, a_3]=1$ in $C_W$.

The commutativity of $a_i$ and $a_4$ for $i=1, 2, 3$ follows easily from the defining relations of $C_W$.
Indeed, for every subset $I$ corresponding to a generator $\sigma_I$ appearing in $a_1$, $a_2$, or $a_3$, we have
\[
\sigma_I\Sigma_{1,7}
=
\Sigma_{1,7}\sigma_I
\qquad\text{and}\qquad
\sigma_I\Sigma_{8,8}
=
\Sigma_{8,8}\sigma_I.
\]
Therefore, $[a_i, a_4]=1$ in $C_W$ for $i=1, 2, 3$.

\medskip
(v) {\bf Type $F_4$}: The following relations hold in $C_W$: 
\[ 
\begin{aligned} 
\Sigma_{1,1}\Sigma_{1,3} &=\Sigma_{1,3}\Sigma_{1,1},& \Sigma_{2,2}\Sigma_{1,3} &=\Sigma_{1,3}\Sigma_{2,2},\\ \Sigma_{1,1}\Sigma_{4,4} &=\Sigma_{4,4}\Sigma_{1,1},& \Sigma_{2,2}\Sigma_{4,4} &=\Sigma_{4,4}\Sigma_{2,2}. 
\end{aligned} 
\]
Hence, $[\Sigma_{1,1}\Sigma_{2,2}, \Sigma_{1,3}\Sigma_{4,4}]=1$
in $C_W$, and therefore, $[a_i, a_j]=1$ in $C_W$.

\medskip
(vi) {\bf Types $H_3$ and $H_4$}: In type $H_3$, there is only one element $a_1$. Hence, there is no commutativity to verify.
In type $H_4$, the following relations hold in $C_W$: 
\[ 
\begin{aligned} 
\Sigma_{1,1}\Sigma_{1,3} &=\Sigma_{1,3}\Sigma_{1,1},& \Sigma_{2,2}\Sigma_{1,3} &=\Sigma_{1,3}\Sigma_{2,2},\\ \Sigma_{1,1}\Sigma_{4,4} &=\Sigma_{4,4}\Sigma_{1,1},& \Sigma_{2,2}\Sigma_{4,4} &=\Sigma_{4,4}\Sigma_{2,2}. 
\end{aligned} 
\]
Hence, $[\Sigma_{1,1}\Sigma_{2,2},\Sigma_{1,3}\Sigma_{4,4}]=1$ in $C_W$, and thus, $[a_1,a_2]=1$ in $C_W$.

\medskip
(vii) {\bf Type $I_2(p)$ for $p \ge 5$}: In this type, there is only one element $a_1$. Hence, we do not need to verify the commutativity.
\end{proof}

Next we show that the subgroup of $C_W$ generated by $a_1,\ldots,a_m$ is contained in $PC_W$.
Then we verify that these elements have infinite order, commute pairwise, and are independent to show that $\langle a_1,\ldots,a_m \rangle$ is isomorphic to $\mathbb{Z}^m$.

\begin{lemma}\label{lem:An-ai-in-PCW}
Each $a_i$ belongs to $P C_W$.
\end{lemma}

\begin{proof}
Recall that $g_W : C_W \rightarrow W$ is the homomorphism defined by $g_W(\sigma_I)=\omega_I$, where $\omega_I$ is the longest element in the standard parabolic subgroup $W_I$ of $W$.

\medskip
(i) {\bf Type $A_n$ for $n \ge 1$}: In this type, we identify the Coxeter group with $S_{n+1}$ (see Section~\ref{subsec:finitecoxetergroups} for details).
Then we have $g_W(\Sigma_{2i,2i})=(2i\ 2i+1)$ and $g_W(\Sigma_{1,1})=(1\ 2)$, while $g_W(\Sigma_{1,2i-1})$ reverses $\{1,\ldots,2i\}$ and $g_W(\Sigma_{2,2i})$ reverses $\{2,\ldots,2i+1\}$.
Hence,
\[
g_W(\Sigma_{2,2i})g_W(\Sigma_{1,1})g_W(\Sigma_{2,2i})=(1\ 2i+1)=g_W(\Sigma_{1,2i-1})g_W(\Sigma_{2i,2i})g_W(\Sigma_{1,2i-1}).
\]
It follows that
\begin{align*}
    g_W(a_i)&=g_W(\Sigma_{2i,2i}\,\Sigma_{1,2i-1}\,\Sigma_{2,2i}\,\Sigma_{1,1}\,\Sigma_{2,2i}\,\Sigma_{1,2i-1})\\
    &=g_W(\Sigma_{2i,2i})\,g_W(\Sigma_{2i,2i})\\
    &=1.
\end{align*}
Therefore, $a_i \in \ker(g_W)=PC_W$ for every $i=1, \dots, m$.

\medskip

(ii) {\bf Type $B_n$ for $n \ge 2$}: It suffices to prove that $g_W(a_i)=(\omega_{I_{1,2i-1}}s_{2i})^4=1$ in $W$.
We use the signed-permutation model of type $B_n$ described in Section~\ref{subsec:finitecoxetergroups}.
In this model, the standard parabolic subgroup $W_{I_{1,2i-1}}$ is generated by $s_1=s_0^B, \dots, s_{2i-1}=s_{2i-2}^B$. We recall that
$s^B_0=[-1,2,\ldots,n]$,
and for $1\leq j\leq 2i-2$,
\[
s^B_j=[1,\ldots,j-1,j+1,j,j+2,\ldots,n].
\]
None of these generators affects the entries in positions $2i, \dots, n$.
Hence, the longest element of $W_{I_{1,2i-1}}$ changes the sign of the first $2i-1$ entries. That is,
\[
\omega_{I_{1,2i-1}}=[-1,-2,\ldots,-(2i-1),2i,\ldots,n].
\]
Let $h=\omega_{I_{1,2i-1}}s_{2i}$.
Then $h$ acts on $[\pm n]$ by
\[
h(j)=-j\ (1\le j\le 2i-2),\quad
h(2i-1)=2i,\quad
h(2i)=-(2i-1),\quad
h(j)=j\ (j\ge 2i+1).
\]
Hence,
\[
h^2(j)=j\ (1\le j\le 2i-2),\quad
h^2(2i-1)=-(2i-1),\quad
h^2(2i)=-2i,\quad
h^2(j)=j\ (j\ge 2i+1).
\]
It follows that $h^4=1$.
Therefore, $g_W(a_i)=(\omega_{I_{1,2i-1}}s_{2i})^4=h^4=1$, and we have $a_i\in PC_W$ for every $i=1, \dots, m$.

\medskip
(iii) {\bf Type $D_n$ for $n \ge 4$}: Assume first that $n=2m+1$. 
Then, for each $i=1, \dots, m$, 
every generator $\Sigma_{p,q}$ appearing in the defining $a_i$ satisfies $q \le 2i \le 2m=n-1$.
Thus, each $a_i$ lies in the subgroup of $C_W$ generated by $\{\sigma_I \svert  I\subseteq I_{1,n-1}\}$.
Since the Coxeter diagram on $I_{1,n-1}=\{s_1,\dots,s_{n-1}\}$ is a path, the Coxeter system $(W_{I_{1,n-1}},I_{1,n-1})$ is the standard Coxeter system of type $A_{n-1}(=A_{2m})$.
Therefore, by the same argument as in type $A_{n-1}$, we have $a_i \in PC_W$ for every $i=1, \dots, m$.

Next, assume that $n=2m$.
For $i=1, \dots, m-1$, every symbol $\Sigma_{p,q}$ appearing in the defining $a_i$ satisfies $q \le 2i \le 2m-2=n-2$. Hence, each $a_i$ lies in the subgroup of $C_W$ generated by $\{\sigma_I \svert  I\subseteq I_{1,n-2}\}$, and therefore also lies in the subgroup generated by $\{\sigma_I \svert  I\subseteq I_{1,n-1}\}$. Then the same argument as in type $A_{n-1}$ shows that $a_i \in PC_W$ for every $i=1, \dots, m$.
For the remaining case $a_m=b^2$, we may write
\[
g_W(b)=g_W(\Sigma_{1,n-1}\,\Sigma_{2,n}\,\sigma_J)=\omega_A\omega_B\omega_C,
\]
where
\begin{equation}
\label{equ:ABC}
A=I_{1,n-1},\quad B=I_{2,n},\quad C=J=\{s_1,\dots,s_{n-2},s_n\}.
\end{equation}
As described in Section~\ref{subsec:finitecoxetergroups}, we use the signed-permutation model for type $D_n$, but with our indexing convention.
Recall that one can represent the generators using the window notation:
\[
s_j=[1,\ldots,n-j-1,n-j+1,n-j,n-j+2,\ldots,n]\quad (1\le j\le n-1),
\]
and $s_n=[-2,-1,3,\ldots n]$.

We compute $\omega_A,\omega_B,\omega_C$ in this model.
First, $W_A=\langle s_1,\ldots,s_{n-1}\rangle$ is the subgroup of signed permutations that do not change any signs.
Hence, it is naturally identified with the symmetric group $S_n$.
Therefore, the longest element of $W_A$ reverses the order of the entries.
Thus, we have
\begin{equation}
\label{equ:omega_A}
\omega_A=[n,n-1,\ldots,1].
\end{equation}
Next, $W_B=\langle s_2,\ldots,s_n\rangle$ fixes the first entry and is
naturally identified with a Coxeter group of type $D_{n-1}$ on the coordinates $2,\ldots,n$. 
Since $n-1$ is odd, its longest element changes the signs of the coordinates $2,\ldots,n-1$ and fixes the coordinate $n$. Therefore,
\[
\omega_B=[1,-2,-3,\ldots,-(n-1),n].
\]
Finally, consider $W_C=\langle s_1,\ldots,s_{n-2},s_n\rangle$.
With our indexing convention for type $D_n$, the generator $s_j$ $(1\le j\le n-2)$ interchanges $n-j$ and $n-j+1$, while $s_n=[-2,-1,3,\ldots,n]$ interchanges $2$ and $-1$.
Thus, $W_C$ is naturally identified with the symmetric group on the set $\{n,n-1,\ldots, 2,-1\}$ with respect to this ordering, where $s_1,\ldots,s_{n-2},s_n$ act as the adjacent transpositions.
(Note that $s_n(2)=-1$ and $s_n(-1)=2$.)
Hence, $W_C$ is a Coxeter group of type $A_{n-1}$, and the longest element of $W_C$ reverses this ordered set.
Therefore,
\[
\omega_C=[-n,n-1,n-2,\ldots,2,-1].
\]

We now compute $\omega_A\omega_B\omega_C$. We use the convention that products act from right to left.
For $k=1$, we have
\[
\omega_C(1)=-n,\qquad \omega_B(-n)=-n,\qquad \omega_A(-n)=-1.
\]
Thus, $(\omega_A\omega_B\omega_C)(1)=-1$.
For $2\le k\le n-1$, we have
\[
\omega_C(k)=n+1-k,\qquad
\omega_B(n+1-k)=-(n+1-k),
\]
and hence $(\omega_A\omega_B\omega_C)(k)=\omega_A(-(n+1-k))=-k$.
Finally,
\[
\omega_C(n)=-1,\qquad \omega_B(-1)=-1,\qquad \omega_A(-1)=-n.
\]
Thus, $(\omega_A\omega_B\omega_C)(n)=-n$.
Therefore, $\omega_A\omega_B\omega_C=[-1,-2,\ldots,-n]$.
Recall that since $n$ is even, the longest element of the Coxeter group of type $D_n$ (see Section~\ref{subsec:finitecoxetergroups}) is
\[
w_0=[-1,-2,\ldots,-n].
\]
Thus, $\omega_A\omega_B\omega_C$ is the longest element of $W$.
By Lemma~\ref{lem:central-longest-element}, it is central and has order two. Hence $g_W(a_m)=g_W(b^2)=1$, and therefore, $a_m \in \ker(g_W)=PC_W$.

\medskip

(iv) {\bf Types $E_6$, $E_7$ and $E_8$}: We show that $a_1,a_2,a_3 \in PC_W$ in types $E_6$ and $E_7$, and that $a_1,a_2,a_3,a_4 \in PC_W$ in type $E_8$.
Since $\omega_{I_{1,1}}=s_1$ and $\omega_{I_{2,2}}=s_2$, we get $g_W(a_1)=(s_2s_1)^3=1$ since $m(s_1,s_2)=3$.
Moreover, all letters in $a_2$ are in the standard parabolic subgroup $W_{I_{1,4}}$ of type $A_4$.
Therefore, $g_W(a_2)=1$ follows from the calculation of type $A_n$.

All letters in $a_3$ belong to the standard parabolic subgroup $W_{I_{1,6}}$ whose Coxeter diagram has type $E_6$.
Let $y=\omega_{I_{1,5}}\omega_{J}\omega_{I_{2,6}}\in W_{I_{1,6}}$. 
The SageMath computation described in Section~\ref{subsec:finitecoxetergroups} gives $ys_1y^{-1}=s_6$.
Then we can rewrite
\[
g_W(a_3)
= s_6\omega_{I_{1,5}}\omega_J\ \omega_{I_{2,6}}s_1\omega_{I_{2,6}}\omega_J\omega_{I_{1,5}}
= s_6ys_1y^{-1}
=1.
\]
This implies $a_3\in PC_W$.
Moreover, using the direct SageMath computation described in Section~\ref{subsec:finitecoxetergroups}, we have $g_W(a_4)=(\omega_{I_{1,7}}\,s_8)^6=1$.
Thus, $a_4\in PC_W$.

\medskip
(v) {\bf Type $F_4$}: Since $m(s_1, s_2)=3$, we have $g_W(a_1)=(\omega_{I_{1,1}}\omega_{I_{2,2}})^3=(s_1s_2)^3=1$.
Hence $a_1\in PC_W$.
For $a_2$, as in the proof of (iv), the SageMath computation gives $g_W(a_2)=(\omega_{I_{1,3}} s_4)^6=1$, and thus, $a_2\in PC_W$.

\medskip
(vi) {\bf Types $H_3$ and $H_4$}: Since $\omega_{I_{1,1}}=s_1$ and $\omega_{I_{2,2}}=s_2$, we get $g_W(a_1)=(s_1s_2)^5=1$ since $m(s_1, s_2)=5$. Hence $a_1\in PC_W$.
For $a_2$, as in (iv) and (v), the SageMath computation gives $g_W(a_2)=(\omega_{I_{1,3}}s_4)^{10}=1$.
Therefore, we have $a_2\in PC_W$.

\medskip
(vii) {\bf Type $I_2(p)$ for $p \ge 5$}: This is immediate since $g_W(a_1)=(s_1s_2)^p=1$.
\end{proof}

Recall from Section~\ref{subsec:Yus-embedding} that, for a Coxeter system $(W,S)$, Yu constructed a monomorphism $PC_W\hookrightarrow \widehat{W}$, which we denote by
\[
\Phi:PC_W\hookrightarrow \widehat{W}.
\]
Now we compute the images $\Phi(a_i)$ for each $i$.

\begin{lemma}\label{lem:An-PQR}
The images $\Phi(a_i)$ are as follows.

\medskip 

{\rm (i) {\bf Type $A_n$ for $n \ge 1$}:} We have 
\[
\Phi(a_i)=(\tau_{P_i}\tau_{Q_i}\tau_{R_i})^2\in \widehat W \qquad (1\le i\le m),
\]
where
\begin{equation}
\label{equ:P_iQ_iR_i}
P_i=W_{I_{2i,2i}},\quad
Q_i=\omega_{I_{2i,2i}}\,W_{I_{1,2i-1}}\,\omega_{I_{2i,2i}}^{-1},\quad
R_i=\omega_{I_{2i,2i}}\omega_{I_{1,2i-1}}\,W_{I_{2,2i}}\,
(\omega_{I_{2i,2i}}\omega_{I_{1,2i-1}})^{-1}.
\end{equation}

\medskip 

{\rm (ii) {\bf Type $B_n$ for $n \ge 2$}:} We have
\[
\Phi(a_i)=\prod_{k=0}^{3}\tau_{U_{i,k}}\tau_{V_{i,k}}
\qquad (1\le i\le m),
\]
where, for \(k=0,1,2,3\),
\[
\begin{aligned}
U_{i,k}
&=(\omega_{I_{1,2i-1}}s_{2i})^k\,
W_{I_{1,2i-1}}\,
(\omega_{I_{1,2i-1}}s_{2i})^{-k},\\
V_{i,k}
&=(\omega_{I_{1,2i-1}}s_{2i})^k
\omega_{I_{1,2i-1}}\,
W_{I_{2i,2i}}\,
\omega_{I_{1,2i-1}}^{-1}
(\omega_{I_{1,2i-1}}s_{2i})^{-k}.
\end{aligned}
\]

\medskip

{\rm (iii) {\bf Type $D_n$ for $n \ge 4$}:} The formula in {\rm (i)} applies to all $a_i$ when $n=2m+1$, and to $a_1, \dots, a_{m-1}$ when $n=2m$. For the remaining element $a_m$, we have
\[
\Phi(a_m)=(\tau_P\,\tau_Q\,\tau_R)^2
\quad\text{in }\widehat{W},
\]
where, with $A,B,C$ are defined in~{\rm (\ref{equ:ABC})},
\[
P=W_A,\quad
Q=\omega_A\,W_B\,\omega_A^{-1},\quad
R=W_C.
\]

\medskip

{\rm (iv) {\bf Types $E_6$, $E_7$, and $E_8$}:} We have
\[
\Phi(a_1)=(\tau_{P_1}\tau_{Q_1}\tau_{R_1})^2 \qquad \text{and} \qquad \Phi(a_2)=(\tau_{P_2}\tau_{Q_2}\tau_{R_2})^2,
\]
where
\[
\begin{aligned}
P_1&=W_{I_{2,2}}=\langle s_2\rangle, \qquad
Q_1=\omega_{I_{2,2}}W_{I_{1,1}}\omega_{I_{2,2}}^{-1}
=\langle s_2s_1s_2\rangle, \\
R_1&=\omega_{I_{2,2}}\omega_{I_{1,1}}W_{I_{2,2}}
(\omega_{I_{2,2}}\omega_{I_{1,1}})^{-1}
=\langle s_1\rangle, \\
P_2&=W_{I_{4,4}}=\langle s_4\rangle,\qquad
Q_2=\omega_{I_{4,4}}W_{I_{1,3}}\omega_{I_{4,4}}^{-1}, \\
R_2&=\omega_{I_{4,4}}\omega_{I_{1,3}}\,W_{I_{2,4}}\,
(\omega_{I_{4,4}}\omega_{I_{1,3}})^{-1}.
\end{aligned}
\]
For $a_3$, we have 
\[ 
\Phi(a_3)= \prod_{k=1}^{8} \tau_{\,u_{k-1}W_{I_k}u_{k-1}^{-1}},
\] 
where $(I_1,\ldots,I_8) = (I_{6,6},I_{1,5},J,I_{2,6},I_{1,1},I_{2,6},J,I_{1,5})$, and 
\[ 
u_0:=1,\qquad u_k:=\omega_{I_1}\cdots\omega_{I_k} \quad (1\le k\le 8).
\]
Note that the fifth factor in this product is $\tau_{\langle s_6\rangle}$. 
In type $E_8$, we have 
\[ 
\Phi(a_4)=\prod_{k=0}^{5}\tau_{U_k}\tau_{V_k} \quad\text{in } \widehat W, \] 
where, for $k=0,1,\dots,5$, 
\[ 
\begin{aligned} 
U_k&=(\omega_{I_{1,7}}s_8)^k\,W_{I_{1,7}}\, (\omega_{I_{1,7}}s_8)^{-k},\\ V_k&=(\omega_{I_{1,7}}s_8)^k\omega_{I_{1,7}}\, W_{I_{8,8}}\,\omega_{I_{1,7}}^{-1} (\omega_{I_{1,7}}s_8)^{-k}. \end{aligned} 
\]

\medskip

{\rm (v) {\bf Type $F_4$}:} We have
\[
\Phi(a_1)=\prod_{k=0}^{2}\tau_{U^{(1)}_k}\tau_{V^{(1)}_k},
\qquad
\Phi(a_2)=\prod_{k=0}^{5}\tau_{U^{(2)}_k}\tau_{V^{(2)}_k}
\quad\text{in }\widehat W,
\]
where, for $k\ge 0$,
\[
\begin{aligned}
U^{(1)}_k&=(s_1s_2)^k\,W_{I_{1,1}}\,(s_1s_2)^{-k},\qquad
V^{(1)}_k=(s_1s_2)^k\,s_1\,W_{I_{2,2}}\,s_1\,(s_1s_2)^{-k}, \\
U^{(2)}_k&=(\omega_{I_{1,3}}s_4)^k\,W_{I_{1,3}}\,
(\omega_{I_{1,3}}s_4)^{-k},\qquad
V^{(2)}_k=(\omega_{I_{1,3}}s_4)^k\,\omega_{I_{1,3}}\,
W_{I_{4,4}}\,\omega_{I_{1,3}}^{-1}\,
(\omega_{I_{1,3}}s_4)^{-k}.
\end{aligned}
\]

\medskip

{\rm (vi) {\bf Types $H_3$ and $H_4$}:} We have
\[
\Phi(a_1)=\prod_{k=0}^{4}\tau_{U^{(1)}_k}\tau_{V^{(1)}_k}
\quad\text{in }\widehat W,
\]
where, for $k\ge 0$,
\[
U^{(1)}_k=(s_1s_2)^k\,W_{I_{1,1}}\,(s_1s_2)^{-k},\qquad
V^{(1)}_k=(s_1s_2)^k\,s_1\,W_{I_{2,2}}\,s_1\,(s_1s_2)^{-k}.
\]
In type $H_4$, we also have
\[
\Phi(a_2)=\prod_{k=0}^{9}\tau_{U^{(2)}_k}\tau_{V^{(2)}_k}
\quad\text{in }\widehat W,
\]
where, for $k\ge 0$,
\[
U^{(2)}_k=(\omega_{I_{1,3}}s_4)^k\,W_{I_{1,3}}\,
(\omega_{I_{1,3}}s_4)^{-k},\qquad
V^{(2)}_k=(\omega_{I_{1,3}}s_4)^k\,\omega_{I_{1,3}}\,
W_{I_{4,4}}\,\omega_{I_{1,3}}^{-1}\,
(\omega_{I_{1,3}}s_4)^{-k}.
\]

\medskip

{\rm (vii) {\bf Type $I_2(p)$ for $p \ge 5$}:} We have
\[
\Phi(a_1)=\prod_{k=0}^{p-1}\tau_{U_k}\tau_{V_k}
\quad\text{in }\widehat W,
\]
where, for $k=0,1,\dots,p-1$,
\[
U_k=(s_1s_2)^k\,W_{I_{1,1}}\,(s_1s_2)^{-k},\qquad
V_k=(s_1s_2)^k s_1\,W_{I_{2,2}}\,s_1^{-1}(s_1s_2)^{-k}.
\]
\end{lemma}

\begin{proof}
In type $A_n$ for $n \ge 1$, we apply Yu's embedding formula in Section~\ref{subsec:Yus-embedding}.
Then the first three factors are exactly $\tau_{P_i}\tau_{Q_i}\tau_{R_i}$. 
A straightforward check in $W\cong S_{n+1}$ shows that the remaining three conjugates occurring in $\Phi(a_i)$ coincide again with $\tau_{P_i}\tau_{Q_i}\tau_{R_i}$, in this order. Therefore, 
$\Phi(a_i)=(\tau_{P_i}\tau_{Q_i}\tau_{R_i})^2$.

In types $B_n$ for $n \ge 2$, $F_4$, $H_3$, $H_4$ and $I_2(p)$ for $p \ge 5$, the formulas are obtained by direct substitution into Yu's embedding formula for $\Phi$.

In type $D_n$ for $n \ge 4$, applying Yu's embedding formula to $(\sigma_A\sigma_B\sigma_C)(\sigma_A\sigma_B\sigma_C)$, we obtain
\[
\Phi(a_m)=\tau_{W_A}\,\tau_{\omega_A W_B\omega_A^{-1}}\,\tau_{\omega_A\omega_B W_C\omega_B^{-1}\omega_A^{-1}}\cdot
\tau_{uW_Au^{-1}}\,\tau_{u(\omega_A W_B\omega_A^{-1})u^{-1}}\,\tau_{u(\omega_A\omega_B W_C\omega_B^{-1}\omega_A^{-1})u^{-1}},
\]
where $u=\omega_A\omega_B\omega_C$.
We now simplify $R$ and the last three parabolic subgroups.
As shown in the proof of Lemma~\ref{lem:An-ai-in-PCW}(iii), $u$ is central. Since $\omega_C^2=1$, we have $\omega_A\omega_B=u\omega_C^{-1}=u\omega_C$.
Therefore,
\[
R=(\omega_A\omega_B)\,W_C\,(\omega_A\omega_B)^{-1}
=(u\omega_C)\,W_C\,(u\omega_C)^{-1}
=\omega_C W_C\omega_C^{-1}=W_C,
\]
where the last equality follows from $\omega_C\in W_C$.
Applying Yu's formula to $a_m$, the first three factors are $\tau_P\tau_Q\tau_R$, and hence
\[
\Phi(a_m)=\tau_P\,\tau_Q\,\tau_R\cdot \tau_{uPu^{-1}}\,\tau_{uQu^{-1}}\,\tau_{uRu^{-1}}.
\]
Moreover, since $u$ is central, $uPu^{-1}=P$, $uQu^{-1}=Q$, and $uRu^{-1}=R$.
Consequently, $\Phi(a_m)=(\tau_P\,\tau_Q\,\tau_R)^2$, as claimed.

In types $E_6$, $E_7$ and $E_8$, the formula for $\Phi(a_3)$ follows directly from Yu's formula. It remains only to identify the fifth factor.
The prefix before the fifth letter is $u_4=s_6\omega_{I_{1,5}}\omega_J\omega_{I_{2,6}}=s_6y$.
By the identity $ys_1y^{-1}=s_6$ from the proof of
Lemma~\ref{lem:An-ai-in-PCW}(iv), we obtain
\[ 
u_4\,s_1\,u_4^{-1} = s_6\,(y s_1 y^{-1})\,s_6 = s_6\,s_6\,s_6 = s_6.
\]
Thus, $u_4W_{I_{1,1}}u_4^{-1}=\langle s_6\rangle$, so the fifth factor is $\tau_{\langle s_6\rangle}$.
The proof for $\Phi(a_4)$ follows directly from Yu's formula, since $a_4=(\sigma_{I_{1,7}}\sigma_{I_{8,8}})^6$.
\end{proof}

Once the torsion-freeness of $PC_W$ is known, it suffices to prove that each $a_i$ is nontrivial to show that $a_i$ has infinite order.
This torsion-freeness is proved in Corollary~\ref{cor:pcw-torsionfree}.
It is straightforward to prove that a word represents the identity using van Kampen diagrams, since one only needs to construct a suitable diagram. 
In the present situation, however, the argument is more delicate because we must prove that no van Kampen diagram with the prescribed boundary label exists.
Therefore, we instead prove that each $\Phi(a_i)$ has infinite order in
$\widehat W$.
To prove this, the following elementary fact will be used several times.

\begin{remark}
\label{rmk:rho_A}
Let $\mathcal W$ be a right-angled Coxeter group whose defining graph is $\mathcal S$.
For any induced subgraph $\mathcal A$ of $\mathcal S$, the map
\[
\rho_{\mathcal A}:\mathcal W \to \langle v\mid v\in V(\mathcal A)\rangle.
\]
defined on generators by
\[
\rho_{\mathcal A}(v)=
\begin{cases}
v & (v\in V(\mathcal A)),\\
1 & (v\notin V(\mathcal A)),
\end{cases}
\]
extends to a well-defined homomorphism.
\end{remark}

\begin{lemma}\label{lem:An-infinite-order}
For each $i$, the element $a_i$ has infinite order in $PC_W$.
\end{lemma}

\begin{proof}
For each type, to show that each $a_i$ has infinite order, it suffices to show that each $\Phi(a_i)$ has infinite order in $\widehat{W}$.

\medskip
(i) {\bf Type $A_n$ for $n \ge 1$}: Assume $(W,S)$ is of type $A_n$.
We need to understand the parabolic subgroups $P_i$, $Q_i$, and $R_i$ defined in~{\rm (\ref{equ:P_iQ_iR_i})} of Lemma~\ref{lem:An-PQR}.
To observe that $\Phi(a_i)=(\tau_{P_i}\tau_{Q_i}\tau_{R_i})^2$ has infinite order, we first prove that, for each $1\le i\le m$, the parabolic subgroups $P_i,Q_i,R_i$ defined in~{\rm (\ref{equ:P_iQ_iR_i})} of Lemma~\ref{lem:An-PQR} are pairwise neither orthogonal nor nested.
It is then immediate to see that the element $(\tau_{P_i}\tau_{Q_i}\tau_{R_i})^2\in \langle \tau_{P_i},\tau_{Q_i},\tau_{R_i}\rangle\cong \mathbb{Z}_2*\mathbb{Z}_2*\mathbb{Z}_2$ has infinite order.

We first consider the case $i=1$.
Identifying $W$ with $S_{n+1}$, we have
\[
P_1=\langle(2\,3)\rangle,\qquad
Q_1=\langle(1\,3)\rangle,\qquad
R_1=\langle(1\,2)\rangle.
\]
These three subgroups are pairwise distinct and their generators are pairwise noncommuting.
Since each subgroup has order two, they are pairwise neither orthogonal nor nested.

We now assume that $i\geq 2$, and we rule out nesting and the orthogonality for each of the three pairs $(P_i, Q_i)$, $(Q_i, R_i)$ and $(P_i, R_i)$.
Hence we consider the three pairs separately.

\medskip
\noindent {\bf Case A.} {\em The pair $(P_i,Q_i)$.}

\medskip In this case, since $s_{2i-1} \in W_{I_{1,2i-1}}$, the subgroup $Q_i$ contains $s_{2i}s_{2i-1}s_{2i}$.
However,
\[
s_{2i}(s_{2i}s_{2i-1}s_{2i})=s_{2i-1}s_{2i} \neq s_{2i}s_{2i-1}=(s_{2i}s_{2i-1}s_{2i})s_{2i},
\]
since the edge joining $s_{2i-1}$ and $s_{2i}$ has label $3$.
Hence $s_{2i}\notin Z_W(Q_i)$, and therefore $P_i\not\subseteq Z_W(Q_i)$ which means $P_i$ and $Q_i$ are not orthogonal.

If $P_i\subseteq Q_i=s_{2i}W_{I_{1,2i-1}}s_{2i}^{-1}$, then after conjugating by $s_{2i}$, we have $P_i \subseteq W_{I_{1,2i-1}}$ $s_{2i} \in W_{I_{1,2i-1}}$. This is impossible, since $P_i=\langle s_{2i} \rangle$ and $W_{I_{1,2i-1}}=\langle s_1,\ldots,s_{2i-1}\rangle$ does not contain $s_{2i}$.
The reverse inclusion $Q_i\subseteq P_i$ is also impossible since $Q_i$ contains $s_1$.

\medskip
\noindent {\bf Case B.} {\em The pair $(Q_i,R_i)$.}

\medskip 
In this case, let $v_i=\omega_{I_{2i,2i}}\omega_{I_{1,2i-1}}
=s_{2i}\omega_{I_{1,2i-1}}$ so that $R_i=v_iW_{I_{2,2i}}v_i^{-1}$.
Since $W_{I_{2,2i-1}}\le W_{I_{2,2i}}$, we have $v_iW_{I_{2,2i-1}}v_i^{-1}\le v_iW_{I_{2,2i}}v_i^{-1}=R_i$. Moreover, since $\omega_{I_{1,2i-1}} W_{I_{2,2i-1}} \omega_{I_{1,2i-1}}^{-1}=W_{I_{1,2i-2}} \le W_{I_{1,2i-1}}$, we have
\[
v_i\,W_{I_{2,2i-1}}\,v_i^{-1}=s_{2i}\,W_{I_{1,2i-2}}\,s_{2i}\ \le\ s_{2i}W_{I_{1,2i-1}}s_{2i}=Q_i.
\]
Hence $v_i\,W_{I_{2,2i-1}}\,v_i^{-1} \in Q_i\cap R_i$, and hence $Q_i\cap R_i$ is not trivial, which means $Q_i$ and $R_i$ are not orthogonal.
Suppose that $Q_i\subseteq R_i$. Since
\[
v_i^{-1} Q_i v_i=\omega_{I_{1,2i-1}}\,W_{I_{1,2i-1}}\,\omega_{I_{1,2i-1}}^{-1}=W_{I_{1,2i-1}}
\qquad \text{and} \qquad
u^{-1}R_i u=W_{I_{2,2i}},
\]
we have $W_{I_{1,2i-1}}\subseteq W_{I_{2,2i}}$. This is impossible since
$s_1\in W_{I_{1,2i-1}}$ whereas $s_1\notin W_{I_{2,2i}}$.
The reverse inclusion is also impossible since $s_{2i}\in W_{I_{2,2i}}\setminus W_{I_{1,2i-1}}$.

\medskip
\noindent {\bf Case C.} {\em The pair $(P_i,R_i)$.}

\medskip 
We first show that $P_i$ and $R_i$ are not orthogonal.
Since $R_i=v_iW_{I_{2,2i}}v_i^{-1}$, $s_{2i}$ centralizes $R_i$ if and only if $v_i^{-1}s_{2i}v_i$ centralizes $W_{I_{2,2i}}$.
Thus, to show that $P_i\not\subseteq Z_W(R_i)$, it suffices to show that 
$v_i^{-1}s_{2i}v_i$ does not centralize $W_{I_{2,2i}}$.
Identifying $W$ with $S_{n+1}$, where
$s_j=(j\ j+1)$, we have $s_{2i}=(2i\ 2i+1)$.
Since $\omega_{I_{1,2i-1}}$ reverses $\{1,\ldots,2i\}$ and fixes $2i+1$, we obtain
\[
v_i^{-1}s_{2i}v_i=\omega_{I_{1,2i-1}}s_{2i}\omega_{I_{1,2i-1}}=(1 \ 2i+1),
\]
which does not commute with $s_{2i}$.
Hence $s_{2i}$ does not centralize $R_i$, and therefore, $P_i$ and $R_i$ are not orthogonal.

If $P_i\subseteq R_i$, then we have $v_i^{-1} P_i v_i \subseteq W_{I_{2,2i}}$.
Hence
\[
v_i^{-1} P_i v_i
=
\langle v_i^{-1} s_{2i} v_i\rangle
=
\langle (1\ 2i+1)\rangle \le W_{I_{2,2i}}.
\]
This is impossible, since $W_{I_{2,2i}}=\langle s_2,\ldots,s_{2i}\rangle$
fixes $1$, whereas $(1\ 2i+1)$ does not.
The reverse inclusion is also impossible since 
$P_i \cong \mathbb{Z}_2$ while $R_i$ is non-abelian.

\medskip
In each of the three cases, the corresponding pair is neither orthogonal nor nested.
It follows that $\Phi(a_i)=(\tau_{P_i}\tau_{Q_i}\tau_{R_i})^2$ has infinite order.
Therefore, $a_i$ has infinite order in $PC_W$.

\medskip 

(ii) {\bf Type $B_n$ for $n\geq 2$}: We first consider the case $i=1$.
Let $r=s_1s_2$. Since $m(s_1,s_2)=4$, it follows that the element $r$ has order four.
By Lemma~\ref{lem:An-PQR}(ii), we have
\[
U_{1,k}=\langle r^{2k}s_1\rangle,
\qquad
V_{1,k}=\langle r^{2k+1}s_1\rangle
\]
for $k=0,1,2,3$. Set $U=U_{1,0}=\langle s_1\rangle$ and $V=V_{1,0}=\langle s_1s_2s_1\rangle$.
The subgroups $U$ and $V$ are distinct subgroups of order two, so they are not nested. Moreover, they are not orthogonal, since $s_1(s_1s_2s_1)=s_2s_1
\neq s_1s_2=(s_1s_2s_1)s_1$.
Thus, $\langle\tau_U,\tau_V\rangle
\cong\mathbb Z_2*\mathbb Z_2$.
Since $r^4=1$, we have $U_{1,2}=U_{1,0}$ and $V_{1,2}=V_{1,0}$.
Moreover, since $r$ has order four,
$s_1, rs_1, r^2s_1, r^3s_1$ are pairwise distinct.
In particular, $U_{1,1}=U_{1,3}=\langle r^2s_1\rangle$ and $V_{1,1}=V_{1,3}=\langle r^3s_1\rangle$
are distinct from both $U$ and $V$.
Therefore,
\[
\rho_{\{U,V\}}(\Phi(a_1))
=(\tau_U\tau_V)^2,
\]
which has infinite order.
Hence $a_1$ has infinite order in $PC_W$.

We now assume $i\geq2$.
Let
\[
U:=U_{i,0}=W_{I_{1,2i-1}} \quad \text{and} \quad U':=U_{i,1}=(\omega_{I_{1,2i-1}}s_{2i})\,W_{I_{1,2i-1}}\,(\omega_{I_{1,2i-1}} s_{2i})^{-1}.
\]
We first show that $U$ and $U'$ are not orthogonal and nested.
If so, once we verify that $\rho_{\{U,U'\}}:\widehat{W} \rightarrow \langle\tau_U,\tau_{U'}\rangle\cong \mathbb{Z}_2*\mathbb{Z}_2$ defined in Remark~\ref{rmk:rho_A} sends $\Phi(a_i)$ to $(\tau_U\tau_{U'})^2$, then we can conclude that $\Phi(a_i)$ has infinite order since $(\tau_U\tau_{U'})^2$ has infinite order.

Since $s_{2i}$ is not adjacent to any generator in $I_{1,2i-2}$,
we have $s_{2i}W_{I_{1,2i-2}}s_{2i}^{-1}=W_{I_{1,2i-2}}$.
This implies that
\[
(\omega_{I_{1,2i-1}} s_{2i})W_{I_{1,2i-2}}(\omega_{I_{1,2i-1}} s_{2i})^{-1}=\omega_{I_{1,2i-1}} W_{I_{1,2i-2}} \omega_{I_{1,2i-1}}^{-1}\ \le\ W_{I_{1,2i-1}}=U,
\]
and since $W_{I_{1,2i-2}}\le W_{I_{1,2i-1}}$, we have $(\omega_{I_{1,2i-1}} s_{2i})W_{I_{1,2i-2}}(\omega_{I_{1,2i-1}} s_{2i})^{-1} \le U'$.
Thus, $U\cap U'\neq\{1\}$, and therefore, $U$ and $U'$ are not orthogonal.

Since $U$ and $U'$ have the same rank, they are nested only if $U=U'$.
However, the subgroup $U'$ contains
\[
y:=(\omega_{I_{1,2i-1}}s_{2i})s_{2i-1}(\omega_{I_{1,2i-1}}s_{2i})^{-1}=\omega_{I_{1,2i-1}}(s_{2i}s_{2i-1}s_{2i})\omega_{I_{1,2i-1}}^{-1}.
\]
Set $x=s_{2i}s_{2i-1}s_{2i}$. Since $m(s_{2i-1}, s_{2i})=3$, we have $\ell(x)=3$ and $\ell(x\,s_{2i})=\ell(s_{2i}s_{2i-1})=2<3=\ell(x)$.
If $x \in U=W_{I_{1, 2i-1}}$, then Lemma~\ref{lem:parabolic-right-descent} implies $s_{2i} \in I_{1,2i-1}$, a contradiction. Hence $x \notin U$. 
If $y \in U$, then, since $\omega_{I_{1,2i-1}}\in U$, $x=\omega_{I_{1,2i-1}}^{-1} \, y \, \omega_{I_{1,2i-1}}\in U$ contradicting $x \notin U$.
Therefore, $y \notin U$, and hence $U'\neq U$.

Thus, $U$ and $U'$ are neither orthogonal nor nested. 
By the defining relation of $\widehat{W}$, the generators $\tau_U$ and $\tau_{U'}$ do not commute.

Next, we prove that $a_i$ has infinite order.
Let $\rho_{\{U,U'\}}:\widehat{W} \rightarrow \langle\tau_U,\tau_{U'}\rangle$ 
be the homomorphism defined in Remark~\ref{rmk:rho_A}.
Then we may represent $(\omega_{I_{1,2i-1}}s_{2i})^2$ by its window notation as follows:
\[
(\omega_{I_{1,2i-1}}s_{2i})^2
=
[1, \ 2, \dots, 2i-2, \ -(2i-1), \ -2i, \ 2i+1, \dots, n].
\]
Since, in $(\omega_{I_{1,2i-1}}s_{2i})^2 U (\omega_{I_{1,2i-1}}s_{2i})^{-2}$, each conjugate of a generator of $U$ involves only the first $2i-1$ coordinates, it follows that $U_{i,2}=(\omega_{I_{1,2i-1}}s_{2i})^2 U (\omega_{I_{1,2i-1}}s_{2i})^{-2} \subseteq U$. Since $U=W_{I_{1,2i-1}}$, we have $U_{i,2} = U$.
and
\[
U_{i,3}=(\omega_{I_{1,2i-1}}s_{2i})^3\,U\,(\omega_{I_{1,2i-1}}s_{2i})^{-3}=(\omega_{I_{1,2i-1}}s_{2i})\,U\,(\omega_{I_{1,2i-1}}s_{2i})^{-1}=U'.
\]
On the other hand, each $V_{i,k}$ is a conjugate of the rank-$1$ parabolic subgroup $W_{I_{2i,2i}}=\langle s_{2i}\rangle$.
Thus, for $i\ge 2$, $V_{i,k}$ cannot coincide with $U$ or $U'$, since $U$ and $U'$ are of rank $2i-1$. The case $i=1$ can be checked directly.
Hence, $\rho_{\{U,U'\}}(\tau_{V_{i,k}})=1$ for $k=0, 1, 2, 3$.
Moreover, from $\tau_{U_{i,0}}=\tau_U$, $\tau_{U_{i,1}}=\tau_{U'}$, $\tau_{U_{i,2}}=\tau_U$, and $\tau_{U_{i,3}}=\tau_{U'}$,
we obtain 
\[
\rho_{\{U,U'\}}\bigl(\Phi(a_i)\bigr)
=\tau_U\tau_{U'}\tau_U\tau_{U'}
=(\tau_U\tau_{U'})^2,
\]
which has infinite order in $\langle \tau_U, \tau_{U'} \rangle\cong \mathbb{Z}_2*\mathbb{Z}_2$ since $\tau_U$ and $\tau_{U'}$ do not commute.
It follows that $\Phi(a_i)$ has infinite order in $\widehat W$, and therefore $a_i$ has infinite order in $PC_W$.

\medskip

(iii) {\bf Type $D_n$ for $n \ge 4$}:
If $n=2m+1$, then the same argument as in type $A_{n-1}$ shows that $a_i$ has infinite order in $PC_W$. If $n=2m$, the same argument applies to $a_1, \dots, a_{m-1}$. Thus, it remains only to show that $\Phi(a_m)$ has infinite order in $\widehat{W}$ for the case $n=2m$.

We first claim that $\tau_P$ and $\tau_Q$ do not commute in $\widehat{W}$.
If so, once we verify that $\rho_{\{P,Q\}}:\widehat{W} \rightarrow \langle\tau_P,\tau_Q\rangle\cong \mathbb{Z}_2*\mathbb{Z}_2$ sends $\Phi(a_i)$ to $(\tau_P\tau_Q)^2$, then we can conclude that $\Phi(a_i)$ has infinite order since $(\tau_P\tau_Q)^2$ has infinite order.

Since $\omega_A=[n,n-1,\ldots,1]$, we obtain $\omega_A s_j \omega_A^{-1}=s_{n-j}$ for $j=2, \dots, n-2$.
Hence, $W_{I_{2,n-2}}=\omega_A W_{I_{2,n-2}} \omega_A^{-1}$, and
this implies $W_{I_{2,n-2}}\le\omega_A W_B\omega_A^{-1}=Q$.
Moreover, clearly, $W_{I_{2,n-2}}\le W_A=P$. 
Therefore, we have $P\cap Q \neq \{1\}$. Thus, $P$ and $Q$ are not orthogonal.

Now we show that $P$ and $Q$ are not nested.
Note that both $P$ and $Q$ are parabolic subgroups of rank $n-1$.
Hence, if $P \le Q$ or $Q \le P$, then $P=Q$.
However, if $P=Q$, then $W_B=\omega_A^{-1} P \omega_A=P=W_A$,
contradicting $A\neq B$ by \cite[Proposition~2.4.1(v)]{Bjorner}.
Therefore, $P$ and $Q$ are neither orthogonal nor nested. 
Thus, $[\tau_P,\tau_Q]\neq 1$ and $\langle \tau_P, \tau_Q \rangle \cong \mathbb Z_2 * \mathbb Z_2$.

By Lemma~\ref{lem:An-PQR}(iii), we have $\Phi(a_m)=(\tau_P\tau_Q\tau_R)^2$.
Let $\rho_{\{P, Q\}} : \widehat{W} \rightarrow \langle\tau_P,\tau_Q\rangle$ be the homomorphism defined in Remark~\ref{rmk:rho_A}. Then $\rho_{\{P, Q\}}(\Phi(a_m))=(\tau_P\tau_Q)^2$ if $R \neq P$ and $R \neq Q$.
Hence, next, we claim that $R\neq P$ and $R\neq Q$. 
Since $s_n\in R$ but $s_n\notin P$, it follows that $R\neq P=W_A$.
For $n \ge 6$, we show that $R \neq Q$ by comparing their orders.
Since $Q=\omega_A W_B \omega_A^{-1}$, $Q$ is conjugate to $W_B$, while $R \cong W(A_{n-1})$.
Moreover, since $W_B \cong W(D_{n-1})$ is of order $2^{n-2}(n-1)!$ and  of $W(A_{n-1})$ is of order $n!$, if $n \ge 6$, then $R \neq Q$.
It remains to consider the case $n=4$.
Since $s_4 \in R$, it is enough to show that $s_4 \notin Q$.
Suppose that $s_4 \in Q$. Then $\omega_A^{-1}s_4\omega_A \in W_B$.
By~(\ref{equ:omega_A}), if $n=4$, then $\omega_A=[4,3,2,1]$, and a direct computation using SageMath gives that $\omega_A=s_1s_2s_3s_1s_2s_1$,
and using SageMath, we obtain that one reduced expression for $\omega_A^{-1}s_4\omega_A$ is
\[
s_2 s_3 s_1 s_2 s_4 s_2 s_3 s_1 s_2.
\]
Here and below, a {\em reduced expression} means that a word in $S$ of minimal length.
Then by~\cite[Corollary~1.4.8(ii)]{Bjorner}, every reduced expression for $\omega_A^{-1}s_4\omega_A$ contains the letter $s_1$.
However, every element of the standard parabolic subgroup
$W_B=\langle s_2,s_3,s_4\rangle$ admits a reduced expression using only the generators $s_2,s_3,s_4$.
This shows that $s_4 \notin Q$. Since $s_4 \in R$, we conclude that $R \neq Q$.
 
In $\langle \tau_P, \tau_Q \rangle \cong \mathbb Z_2*\mathbb Z_2$, the element $\tau_P\tau_Q$ has infinite order. Hence, $(\tau_P\tau_Q)^2$ has infinite order.
Therefore, $\Phi(a_m)$ has infinite order in $\widehat{W}$, and thus, $a_m$ has infinite order in $PC_W$.

\medskip
(iv) {\bf Types $E_6$, $E_7$ and $E_8$}: For $a_1$ and $a_2$, the same argument in the proof of Lemma~\ref{lem:An-infinite-order}(i) shows that they have infinite order.

For $a_3$, let $P_{3,k}:=u_{k-1}W_{I_k}u_{k-1}^{-1}$ for $1 \le k \le 8$.
Then Lemma~\ref{lem:An-PQR} gives that
$\Phi(a_3)=\tau_{P_{3,1}}\tau_{P_{3,2}}\cdots\tau_{P_{3,8}}$.
Since $I_2=I_{1,5}$ and $u_1=\omega_{I_1}=s_6$, we obtain $P_{3,2}=s_6W_{1,5}s_6^{-1}$. Set 
\[
U:=W_{I_{1,5}},\qquad V:=s_6 U s_6^{-1}.
\]
We use the homomorphism $\rho_{\{U,V\}} : \widehat{W} \rightarrow \langle \tau_U, \tau_V \rangle$ defined in Remark~\ref{rmk:rho_A}.

First, we show that $\tau_U$ and $\tau_V$ do not commute in $\widehat{W}$.
Once we verify that $\rho_{\{U,V\}}(\Phi(a_3))=(\tau_V\tau_U)^2$, then it follows that $\Phi(a_3)$ has infinite order.
Since $U$ and $V$ are conjugate, they have the same order.
Hence, they are nested if and only if $U=V$.
Since $s_5 \in U$ and $s_6$ normalize $U$, it follows that $s_6s_5s_6 \in U$. However, 
\[
\ell\bigl((s_6s_5s_6)s_6\bigr)=\ell(s_6s_5)=2<3=\ell(s_6s_5s_6),
\]
and this implies that $s_6 \in U=W_{I_{1,5}}$ by Lemma~\ref{lem:parabolic-right-descent}, which is a contradiction. Hence, $U\neq V$, that is, $U,V$ are not nested.
By the Coxeter diagrams of types $E_6$, $E_7$ and $E_8$, $s_6$ commute with $s_1, s_2, s_3$ and $s_4$. Hence $W_{I_{1,4}} \le U \cap V$, and therefore, $U \cap V \neq \{1\}$, and thus, $U$ and $V$ are neither nested nor orthogonal.

We now show that $\rho_{\{U,V\}}(\Phi(a_3))=(\tau_V\tau_U)^2$.
Since $P_{3,1}=u_0W_{I_1}u_0^{-1}=W_{I_{6,6}}=\langle s_6 \rangle$ and
$P_{3,5}=u_4W_{I_5}u_4^{-1}=u_4W_{I_{1,1}}u_4^{-1}=u_4\langle s_1 \rangle u_4^{-1}$, they have rank $1$, while $U$ and $V$ have rank~$5$. Hence, they cannot coincide with $U$ or $V$.
As mentioned above, $P_{3,2}=V$. 
Next, since $P_{3,3}$ and $P_{3,7}$ are conjugates of $W_J$, and since $W_J$ is isomorphic to $W(A_5)$, $P_{3,3}$ and $P_{3,7}$ are both isomorphic to $W(A_5)$. However, since $U$ and $V$ are isomorphic to $W(D_5)$, they cannot coincide with $U$ or $V$.
Next, using SageMath, one can check the following equalities in the Coxeter group $W_{I_{1,6}}$ of type $E_6$:
\[
\begin{aligned}
&u_3 s_2 u_3^{-1}=s_4,\quad
u_3 s_3 u_3^{-1}=s_3,\quad
u_3 s_4 u_3^{-1}=s_5,\quad
u_3 s_5 u_3^{-1}=s_2,\quad
u_3 s_6 u_3^{-1}=s_1,\\
&u_5 s_2 u_5^{-1}=s_6s_5s_6,\quad
u_5 s_3 u_5^{-1}=s_3,\quad
u_5 s_4 u_5^{-1}=s_4,\quad
u_5 s_5 u_5^{-1}=s_2,\quad
u_5 s_6 u_5^{-1}=s_1.
\end{aligned}
\]
Therefore,
\[
P_{3,4}=u_3 W_{I_{2,6}}u_3^{-1}=\langle s_1,s_2,s_3,s_4,s_5\rangle=W_{I_{1,5}}=U, \quad \text{and}
\]
\[
P_{3,6}=u_5 W_{I_{2,6}}u_5^{-1}
=\langle s_1,s_2,s_3,s_4,s_6s_5s_6\rangle=s_6\langle s_1,s_2,s_3,s_4,s_5\rangle s_6=s_6Us_6=V.
\]
Finally, since $a_3\in PC_W$ by Lemma~\ref{lem:An-ai-in-PCW}(iv),
we have $u_8=u_7\omega_{I_8}=\omega_{I_1}\cdots\omega_{I_8}=g_W(a_3)=1$.
Hence, $u_7=\omega_{I_8}^{-1}=\omega_{I_{1,5}}^{-1}=\omega_{I_{1,5}} \in U$.
Therefore,
\[
P_{3,8}=u_7W_{I_{1,5}}u_7^{-1}=\omega_{I_{1,5}}U\omega_{I_{1,5}}^{-1}=U.
\]
Thus,
\[
\rho_{\{U,V\}}
\bigl(\Phi(a_3)\bigr)=\tau_{P_{3,2}}\tau_{P_{3,4}}\tau_{P_{3,6}}\tau_{P_{3,8}}
=\tau_V\tau_U\tau_V\tau_U
=(\tau_V\tau_U)^2,
\]
which has infinite order in $\langle \tau_U, \tau_V \rangle$ since $\tau_U$ and $\tau_V$ do not commute.
Therefore, $\Phi(a_3)$ has infinite order in $\widehat W$, and so does $a_3$ in $PC_W$.

For $a_4$, recall from Lemma~\ref{lem:An-PQR}(iv) that
$\Phi(a_4)=\prod_{k=0}^{5}\tau_{U_k}\tau_{V_k}$ where 
\[
U_k=(\omega_{I_{1,7}}s_8)^k W_{I_{1,7}}(\omega_{I_{1,7}}s_8)^{-k} \quad \text{and} \quad V_k=(\omega_{I_{1,7}}s_8)^k \omega_{I_{1,7}}W_{I_{8,8}}\omega_{I_{1,7}}^{-1}(\omega_{I_{1,7}}s_8)^{-k}.
\]
First we show that $\tau_{U_0}$ and $\tau_{V_0}$ do not commute.
Note that $U_0=W_{I_{1,7}}=\langle s_1,\dots,s_7\rangle$ and $V_0=\omega_{I_{1,7}}\langle s_8\rangle \omega_{I_{1,7}}^{-1}$ are not nested since $s_8 \notin W_{I_{1,7}}$.
Moreover, they are not orthogonal since $s_8$ does not centralize $W_{I_{1,7}}$ (it fails to commute with $s_7$).
Hence, $\tau_{U_0}$ and $\tau_{V_0}$ do not commute.

Let $\rho_{\{U_0, V_0\}}:\widehat W\to \langle\tau_{U_0},\tau_{V_0}\rangle$ be the homomorphism defined in Remark~\ref{rmk:rho_A}.
We show that $\Phi(a_4)$ maps to $(\tau_{U_0}\tau_{V_0})^2$ under the retraction keeping only $\tau_{U_0}$ and $\tau_{V_0}$.
To see this, we first claim that $U_k \neq U_0$ and $V_k \neq V_0$ for $k \in \{1,2,4,5\}$.
Indeed, if $U_k=U_0$ for $k\in \{1,2,4,5\}$, then $(\omega_{I_{1,7}}s_8)^k$ normalizes $W_{I_{1,7}}$.
Hence, we have $(\omega_{I_{1,7}}s_8)^k s_7 (\omega_{I_{1,7}}s_8)^{-k} \in W_{I_{1,7}}$ since $s_7 \in W_{I_{1,7}}$.
However, a direct computation using SageMath shows that a reduced word for $(\omega_{I_{1,7}}s_8)^k s_7 (\omega_{I_{1,7}}s_8)^{-k}$ contains the letter $s_8$ for every $k\in\{1,2,4,5\}$.
By~\cite[Corollary~1.4.8(ii)]{Bjorner}, any two reduced expressions of a
given Coxeter group element contain the same set of generators. 
Hence, once $s_8$ occurs in one reduced expression of $(\omega_{I_{1,7}}s_8)^k s_7 (\omega_{I_{1,7}}s_8)^{-k}$, it occurs in every reduced expression. Consequently, 
\[
(\omega_{I_{1,7}}s_8)^k s_7 (\omega_{I_{1,7}}s_8)^{-k}\notin W_{I_{1,7}}
\]
and this contradiction proves $U_k \neq U_0$ for $k \in \{1, 2, 4, 5 \}$.


The argument for $V_k \neq V_0$ is analogous. Suppose to the contrary that $V_k=V_0$. Then $(\omega_{I_{1,7}}s_8)^k$ normalizes $\omega_{I_{1,7}} \langle s_8\rangle \omega_{I_{1,7}}^{-1}$.
However, a direct computation using SageMath in the Coxeter group of type $E_8$ rules this out. Therefore, $V_k \neq V_0$.

Next, we show that $U_3=U_0$ and $V_3=V_0$.
By the proof of Lemma~\ref{lem:An-ai-in-PCW}(iv), $\omega_{I_{1,7}}s_8$ has order $6$.
Hence, the subgroup $\langle \omega_{I_{1,7}}, s_8 \rangle$ is the dihedral group of order $12$.
In particular, $(\omega_{I_{1,7}}s_8)^3$ is central in $\langle \omega_{I_{1,7}}, s_8\rangle$.
This implies that $(\omega_{I_{1,7}}s_8)^3$ commutes with $\omega_{I_{1,7}}$ and with $s_8$.
Moreover, using SageMath, a direct computation in the Coxeter group of type $E_8$ shows that $(\omega_{I_{1,7}}s_8)^3$ normalizes $W_{I_{1,7}}$.
Therefore, we have $U_3=(\omega_{I_{1,7}}s_8)^3 W_{I_{1,7}} (\omega_{I_{1,7}}s_8)^{-3}=W_{I_{1,7}}=U_0$.
For $V_3$, using the centrality of $(\omega_{I_{1,7}}s_8)^3$ in $\langle \omega_{I_{1,7}},s_8\rangle$, we get
\[
\begin{aligned}
V_3&=(\omega_{I_{1,7}}s_8)^3\bigl(\omega_{I_{1,7}}\langle s_8\rangle \omega_{I_{1,7}}^{-1}\bigr)(\omega_{I_{1,7}}s_8)^{-3} \\
&=\omega_{I_{1,7}}\,\bigl((\omega_{I_{1,7}}s_8)^3\langle s_8\rangle (\omega_{I_{1,7}}s_8)^{-3}\bigr)\,\omega_{I_{1,7}}^{-1} \\
&=\omega_{I_{1,7}}\langle s_8\rangle \omega_{I_{1,7}}^{-1} 
=V_0.
\end{aligned}
\]

Thus, $\rho(\tau_{U_k})=\rho(\tau_{V_k})=1$ for $k\in \{1,2,4,5\}$, while $\rho(\tau_{U_0})=\rho(\tau_{U_3})=\tau_{U_0}$ and $\rho(\tau_{V_0})=\rho(\tau_{V_3})=\tau_{V_0}$.
This implies that
\[
\rho\bigl(\Phi(a_4)\bigr)=(\tau_{U_0}\tau_{V_0})^2,
\]
which has infinite order in $\langle \tau_{U_0}, \tau_{V_0} \rangle \cong \mathbb{Z}^2 * \mathbb{Z}^2$ since $\tau_{U_0},\tau_{V_0}$ do not commute.
Therefore, $\Phi(a_4)$ has infinite order in $\widehat W$, and so does $a_4$ in $PC_W$.

\medskip
(v) {\bf Type $F_4$}: For $a_1$, the same argument in the proof of Lemma~\ref{lem:An-infinite-order}(ii) shows that $a_1$ has infinite order.

For $a_2$, let $U=U^{(2)}_0=W_{I_{1,3}}$ and $V=V^{(2)}_0=\omega_{I_{1,3}}\langle s_4\rangle\omega_{I_{1,3}}^{-1}$.
Then $\tau_U$ and $\tau_V$ do not commute in $\widehat W$ for the following reason.
Since $V$ is a conjugate of the subgroup $\langle s_4 \rangle$ of rank $1$, while $U$ has rank $3$, it follows that $U \not\subset V$.
Moreover, $V \not\subset U$. Otherwise, we have $s_4 \in U$, which is a contradiction. Hence $U$ and $V$ are not nested.
Since $s_4$ does not commute to $s_3 \in I_{1,3}$, $\omega_{I_{1,3}}s_4\omega_{I_{1,3}}^{-1}$ does not commute with $\omega_{I_{1,3}}s_3\omega_{I_{1,3}}^{-1}=s_3\in U$.
Therefore, $U$ and $V$ are not orthogonal.

It is therefore enough to show that $\rho_{{U,V}}:\widehat W\to \langle \tau_U,\tau_V\rangle$ defined in Remark~\ref{rmk:rho_A} sends $\Phi(a_2)$ to $(\tau_U\tau_V)^2$.
For this purpose, we compare the sets $U_k^{(2)}$ and $V_k^{(2)}$ for all $k$.

Using SageMath, one can check that $(\omega_{I_{1,3}}s_4)^3$ is the longest element of the Coxeter group of type $F_4$. 
Then, by Lemma~\ref{lem:central-longest-element}, $(\omega_{I_{1,3}}s_4)^3$ is central.
Thus, we have $V^{(2)}_3=V$ and $U^{(2)}_3=U$.

For the remaining indices $k\in\{1,2,4,5\}$, we show that $U^{(2)}_k\neq U$ and $V^{(2)}_k\neq V$.
If $U^{(2)}_k=U$, then $(\omega_{I_{1,3}}s_4)^k$ normalizes $U=W_{I_{1,3}}$. Hence, $(\omega_{I_{1,3}}s_4)^k s_3 (\omega_{I_{1,3}}s_4)^{-k} \in U$. For $k=1, 2$, using SageMath, one can check directly that a reduced expression of $(\omega_{I_{1,3}}s_4)^k s_3 (\omega_{I_{1,3}}s_4)^{-k}$ contains the letter $s_4$.
Therefore, $(\omega_{I_{1,3}}s_4)^k s_3 (\omega_{I_{1,3}}s_4)^{-k} \notin U$, a contradiction.
Note that by the proof of Lemma~\ref{lem:An-ai-in-PCW}(v), $\omega_{I_{1,3}}s_4$ has order $6$ in the Coxeter group of type $F_4$.
For $k=4, 5$, it follows that $(\omega_{I_{1,3}}s_4)^{-k}=(\omega_{I_{1,3}}s_4)^{6-k}$, and if $U^{(2)}_{k}=U$, then $U^{(2)}_{6-k}=U$.
Thus, $U^{(2)}_k\neq U$.

To see the equality among $V_k^{(2)}$'s, since each $V_k^{(2)}$ is a parabolic subgroup of rank 1, it is enough to check whether the generators are the same. Using Sagemath, one can check that the element $(\omega_{I_{1,3}}s_4)^k$ centralizes $\omega_{I_{1,3}} s_4\omega_{I_{1,3}}^{-1}$ generating $V$ if and only if $k\equiv 0$ or $3\pmod 6$.
Consequently, $V^{(2)}_k\neq V$ for $k\in\{1,2,4,5\}$.

Let $\rho_{\{U, V\}}:\widehat W\to \langle\tau_U,\tau_V\rangle$ be the homomorphism
defined in Remark~\ref{rmk:rho_A}.
Then we have
\[
\rho_{\{U, V\}}\bigl(\Phi(a_2)\bigr)
=\rho_{\{U, V\}}(\tau_{U^{(2)}_0}\tau_{V^{(2)}_0}\tau_{U^{(2)}_3}\tau_{V^{(2)}_3})
=\tau_U\tau_V\tau_U\tau_V
=(\tau_U\tau_V)^2.
\]
Since $\tau_U$ and $\tau_V$ do not commute, $(\tau_U\tau_V)^2$ has infinite order in $\langle\tau_U,\tau_V\rangle$, and therefore $\Phi(a_2)$ has infinite order in $\widehat W$.

\medskip
(vi) {\bf Types $H_3$ and $H_4$}: 
To show that $a_1$ and $a_2$ have infinite order in $PC_W$, we show that $\Phi(a_1)$ and $\Phi(a_2)$ have infinite order in $\widehat{W}$ in type $H_4$.

For $a_1$, we claim that $U_0^{(1)}$ and $V_0^{(1)}$ are neither orthogonal nor nested, and the map $\rho_{\{U_0^{(1)}, V_0^{(1)}\}}:\widehat W \rightarrow \langle\tau_U,\tau_V\rangle \cong \mathbb{Z}_2*\mathbb{Z}_2$ defined in Remark~\ref{rmk:rho_A} sends $\Phi(a_1)$ to $(\tau_{U_0^{(1)}} \tau_{V_0^{(1)}})^2$.
 
Note that $W_{I_{1,2}}$ is the dihedral group $W(I_2(5))$ of order 10, that is, $W_{I_{1,2}}=\langle s_1,s_2\mid s_1^2=s_2^2=1,\ (s_1s_2)^5=1\rangle$.
Let $r=s_1s_2$.
Then $r$ has order $5$ and $s_1rs_1=r^{-1}$, and hence $s_1r^k=r^{-k}s_1$ for every $k\in\mathbb Z$.
In particular,
\[
r^k s_1 r^{-k} \;=\; r^k (s_1 r^{-k}) \;=\; r^k(r^k s_1)\;=\; r^{2k}s_1.
\]
Therefore, $U^{(1)}_k=\langle r^ks_1r^{-k}\rangle=\langle r^{2k}s_1\rangle$ and
\[
V^{(1)}_k=\langle r^k(s_1s_2s_1)r^{-k}\rangle
=\langle r^k(r s_1)r^{-k}\rangle
=\langle r^{2k+1}s_1\rangle,
\]
since $s_1\langle s_2\rangle s_1$ is generated by $s_1s_2s_1=r s_1$.

For $j=0,1,2,3,4$, the five reflections $r^j s_1$ are pairwise distinct in $W(I_2(5))$.
Then the parabolic subgroups $U^{(1)}_0,\dots,U^{(1)}_4$ are pairwise distinct, and similarly $V^{(1)}_0,\dots,V^{(1)}_4$ are pairwise distinct.

Comparing the exponents, we have $V^{(1)}_k=U^{(1)}_\ell$ if and only if $2k+1 \equiv 2\ell \pmod 5$, that is, $\ell \equiv k+3 \pmod 5$.
Thus, we have $V^{(1)}_2=U^{(1)}_0$ and $U^{(1)}_3=V^{(1)}_0$, and these are the only coincidences involving $U^{(1)}_0$ or $V^{(1)}_0$.

Let $\rho_{\{U^{(1)}_0, V^{(1)}_0\}}:\widehat W \rightarrow \langle\tau_{U^{(1)}_0},\tau_{V^{(1)}_0}\rangle$ be the homomorphism defined in Remark~\ref{rmk:rho_A}.
Then 
\[
\rho_{\{U^{(1)}_0, V^{(1)}_0\}}\bigl(\Phi(a_1)\bigr)
=\tau_{U^{(1)}_0}\tau_{V^{(1)}_0}\tau_{U^{(1)}_0}\tau_{V^{(1)}_0}
=(\tau_{U^{(1)}_0}\tau_{V^{(1)}_0})^2.
\]
Since $U^{(1)}_0$ and $V^{(1)}_0$ are distinct subgroups of rank $1$, they are not nested. Moreover, $U^{(1)}_0$ and $V^{(1)}_0$ are not orthogonal since $s_1(s_1s_2s_1)=s_2s_1 \neq s_1s_2=(s_1s_2s_1)s_1$, where the inequality follows from $m(s_1, s_2)=5$.
Thus, $\tau_{U^{(1)}_0}$ and $\tau_{V^{(1)}_0}$ do not commute. Therefore, $\rho_{\{U^{(1)}_0, V^{(1)}_0\}}\bigl(\Phi(a_1)\bigr)$ has infinite order in $\langle\tau_{U^{(1)}_0},\tau_{V^{(1)}_0}\rangle \cong \mathbb{Z}_2 * \mathbb{Z}_2$, and hence $\Phi(a_1)$ has infinite order.

Next, for $a_2$, let $U=U^{(2)}_0=W_{I_{1,3}}$ and $V=V^{(2)}_0=\omega_{I_{1,3}}\langle s_4\rangle\omega_{I_{1,3}}^{-1}$.
It suffices to show that $U$ and $V$ are neither orthogonal nor nested, and the map $\rho_{\{U, V\}}:\widehat W \rightarrow \langle\tau_U,\tau_V\rangle\cong\mathbb{Z}_2*\mathbb{Z}_2$ defined in Remark~\ref{rmk:rho_A} sends $\Phi(a_2)$ to $(\tau_U \tau_V)^2$.
Using SageMath again, one can check that $(\omega_{I_{1,3}}s_4)^5$ is the longest element of the Coxeter group of type $H_4$.
Then, by Lemma~\ref{lem:central-longest-element}, $(\omega_{I_{1,3}}s_4)^5$ is central.
Thus, we have $U^{(2)}_5=U$ and $V^{(2)}_5=V$.
Using an argument similar to the proof of Lemma~\ref{lem:An-infinite-order}(v), we obtain $U^{(2)}_k\neq U$ and $V^{(2)}_k \neq V$ for $k\in\{1,2,3,4,6,7,8,9\}$.

Moreover, $U$ and $V$ are not nested since $U$ has rank $3$, whereas $V$ has rank $1$ and $s_4 \notin W_{I_{1,3}}$.
Since $s_4$ does not commute with $s_3\in I_{1,3}$, it followst that $\omega_{I_{1,3}}s_4\omega_{I_{1,3}}^{-1}$ does not commute with $\omega_{I_{1,3}}s_3\omega_{I_{1,3}}^{-1} \in U$.
Thus, $U$ and $V$ are not orthogonal.
Therefore, $\tau_U$ and $\tau_V$ do not commute.

Let $\rho_{\{U, V\}}:\widehat W \rightarrow \langle\tau_U,\tau_V\rangle$ be the homomorphism defined in Remark~\ref{rmk:rho_A}.
Then 
\[
\rho_{\{U, V\}}\bigl(\Phi(a_2)\bigr)=\tau_U\tau_V\tau_U\tau_V=(\tau_U\tau_V)^2,
\]
which has infinite order in $\langle\tau_U,\tau_V\rangle$ since $\tau_U$ and $\tau_V$ do not commute.
Therefore, $\Phi(a_2)$ has infinite order in $\widehat W$.

\medskip

(vii) {\bf Type $I_2(p)$ for $p \ge 5$}:
We show that $\Phi(a_1)$ has infinite order in $\widehat{W}$.
To do this, we claim that $U_0$ and $V_0$ are neither orthogonal nor nested, and the map $\rho_{\{U_0, V_0\}}:\widehat W \rightarrow \langle\tau_U,\tau_V\rangle \cong \mathbb{Z}_2*\mathbb{Z}_2$ defined in Remark~\ref{rmk:rho_A} sends $\Phi(a_1)$ to $(\tau_{U_0} \tau_{V_0})^2$.

Note that $W(I_2(p))$ is the dihedral group of order $2p$, i.e. $W(I_2(p))=\langle s_1,s_2\mid s_1^2=s_2^2=1,\ (s_1s_2)^{p}=1\rangle$.
Let $r=s_1s_2$.
Then $r$ has order $p$ and $s_1rs_1=r^{-1}$, hence $s_1r^k=r^{-k}s_1$ for every $k\in\mathbb Z$.
In particular,
\[
r^k s_1 r^{-k} \;=\; r^k (s_1 r^{-k}) \;=\; r^k(r^k s_1)\;=\; r^{2k}s_1.
\]
Thus, $U_k=\langle r^ks_1r^{-k}\rangle=\langle r^{2k}s_1\rangle$. Moreover, $s_1\langle s_2\rangle s_1$ is generated by $s_1s_2s_1=r s_1$, so
\[
V_k=\langle r^k(s_1s_2s_1)r^{-k}\rangle
=\langle r^k(r s_1)r^{-k}\rangle
=\langle r^{2k+1}s_1\rangle.
\]
Since $r$ has order $p$, the reflections $r^js_1$ for $0\leq j<p$ are pairwise distinct.
If $p$ is odd, the subgroups $U_0,\ldots,U_{p-1}$ are pairwise distinct, and so are $V_0,\ldots,V_{p-1}$.
If $p$ is even, then
\[
U_{k+p/2}=U_k,\qquad
V_{k+p/2}=V_k
\]
for $0\leq k<p/2$, and each family contains exactly $p/2$ distinct subgroups.

Now, let $U=U_0=\langle s_1\rangle$ and $V=V_0=s_1\langle s_2\rangle s_1=\langle rs_1\rangle$.
Then $U$ and $V$ are not nested since $U \neq V$, and they are not orthogonal by an argument similar to that in Lemma~\ref{lem:An-infinite-order}(vi). Hence the two generators $\tau_U$ and $\tau_V$ do not commute.

Let $\rho_{\{U, V\}}:\widehat W \rightarrow \langle\tau_U,\tau_V\rangle$ be the homomorphism defined in Remark~\ref{rmk:rho_A}.
To compute $\rho_{\{U,V\}}(\Phi(a_1))$, we determine exactly when $U_k$ or $V_k$ equals $U$ or $V$.
From $U_k=\langle r^{2k}s_1\rangle$ and $V_k=\langle r^{2k+1}s_1\rangle$, we have
\[
\begin{aligned}
U_k&=U \Longleftrightarrow 2k\equiv 0\pmod p, &
V_k&=V \Longleftrightarrow 2k+1\equiv 1\pmod p, & \\
V_k&=U \Longleftrightarrow 2k+1\equiv 0\pmod p, &
U_k&=V \Longleftrightarrow 2k\equiv 1\pmod p.
\end{aligned}
\]
If $p$ is even, then $2k\equiv 0\pmod p$ holds exactly for $k=0$ and $k=p/2$, while the congruences $2k\equiv 1$ and $2k+1\equiv 0$ have no solutions.
Hence, the only surviving blocks under $\rho$ are those with $k=0$ and $k=p/2$, and
\[
\rho_{\{U, V\}}(\Phi(a_1))=\tau_U\tau_V\,\tau_U\tau_V=(\tau_U\tau_V)^2.
\]
If $p$ is odd, then $2$ is invertible modulo $p$.
Hence, $2k\equiv 0\pmod p$ forces $k=0$, while the congruences $2k+1\equiv 0$ and $2k\equiv 1$ have the unique solutions $k=(p-1)/2$ and $k=(p+1)/2$, respectively.
Therefore, $\rho(\Phi(a_1))$ receives the additional surviving factors $\tau_{V_{(p-1)/2}}=\tau_U$ and $\tau_{U_{(p+1)/2}}=\tau_V$, and again
\[
\rho_{\{U, V\}}(\Phi(a_1))=\tau_U\tau_V\tau_U\tau_V=(\tau_U\tau_V)^2.
\]
Since $\tau_U$ and $\tau_V$ do not commute, $(\tau_U\tau_V)^2$ has infinite order in $\langle\tau_U,\tau_V\rangle$.
Thus, $\Phi(a_1)$ has infinite order in $\widehat W$.
Therefore, $a_1$ has infinite order in $PC_W$.
\end{proof}

Now we are ready to show that $\langle a_1,\ldots,a_m\rangle$ is a free abelian subgroup in $PC_W$.

\begin{proposition}\label{thm:An-free-abelian}
The subgroup $\langle a_1,\ldots,a_m\rangle\leq PC_W$ is free abelian of rank $m$.
\end{proposition}

\begin{proof}
By Lemmas~\ref{lem:An-commutativity}-\ref{lem:An-infinite-order}, all $a_i$ commute pairwise, lie in $PC_W$, and have infinite order.
Therefore, it is enough to show the independence.

\medskip
(i) {\bf Type $A_n$ for $n \ge 1$}: For each $1\leq i\leq m$, let $K_i=\langle\tau_{P_i},\tau_{Q_i},\tau_{R_i}\rangle \leq\widehat W$.
For $i\neq j$, we have $\{P_i,Q_i,R_i\}\cap\{P_j,Q_j,R_j\}=\emptyset$. 
Consider the retraction $\rho_i:=\rho_{\{P_i,Q_i,R_i\}}: \widehat W\longrightarrow K_i$ defined in Remark~\ref{rmk:rho_A}.
Then $\rho_i(\Phi(a_i))=\Phi(a_i)$ and $\rho_i(\Phi(a_j))=1$ for $j \neq i$.

Suppose that $\prod_{i=1}^m a_i^{n_i}=1$ in $PC_W$.
Applying $\rho_i\circ\Phi$ gives $\Phi(a_i)^{n_i}=1$.
Since $\Phi(a_i)$ has infinite order by Lemma 4.6(i), we obtain $n_i=0$ for every $i$.
Thus, $a_1,\ldots,a_m$ are independent.
Together with their pairwise commutativity, this proves that $\langle a_1,\ldots,a_m\rangle \cong\mathbb Z^m$.

\medskip

(ii) {\bf Type $B_n$ for $n \ge 2$}:
Assume that $\prod_{i=1}^m a_i^{n_i}=1$ in $PC_W$ for some integers $n_i$.  Applying the injective homomorphism $\Phi$ gives $\prod_{i=1}^m \Phi(a_i)^{n_i}=1$ in $\widehat W$.
We show successively, starting from the largest index, that all $n_i$ vanish.
Consider first
\[
\rho_{\{U_{m,0},U_{m,1}\}}:
\widehat W\longrightarrow
\langle\tau_{U_{m,0}},\tau_{U_{m,1}}\rangle
\cong\mathbb Z_2*\mathbb Z_2
\]
as defined in Remark~\ref{rmk:rho_A}.
By the proof of Lemma~\ref{lem:An-infinite-order}(ii), $\rho_{\{U_{m,0},U_{m,1}\}}(\Phi(a_m))
=(\tau_{U_{m,0}}\tau_{U_{m,1}})^2$,
which has infinite order.

For every $j<m$, each parabolic subgroup appearing in $\Phi(a_j)$ has rank either $2j-1$ or $1$, whereas $U_{m,0}$ and $U_{m,1}$ have rank $2m-1$.
Hence none of the parabolic subgroups appearing in $\Phi(a_j)$ is equal to $U_{m,0}$ or $U_{m,1}$, and therefore
$\rho_{\{U_{m,0},U_{m,1}\}}(\Phi(a_j))=1$ for $j < m$.
Applying $\rho_{\{U_{m,0},U_{m,1}\}}$ to the relation above gives $(\tau_{U_{m,0}}\tau_{U_{m,1}})^{2n_m}=1$.
Since $\tau_{U_{m,0}}\tau_{U_{m,1}}$ has infinite order, we obtain $n_m=0$.
Thus, $\prod_{i=1}^{m-1}\Phi(a_i)^{n_i}=1$.

Repeating the same argument with $\rho_{\{U_{m-1,0},U_{m-1,1}\}},\rho_{\{U_{m-2,0},U_{m-2,1}\}},\ldots,\rho_{\{U_{2,0},U_{2,1}\}}$ successively gives $n_m=n_{m-1}=\cdots=n_2=0$.
We are then left with $\Phi(a_1)^{n_1}=1$.
Since $\Phi(a_1)$ has infinite order by Lemma~\ref{lem:An-infinite-order}(ii), we have $n_1=0$.
Hence, $n_1=\cdots=n_m=0$.
Therefore, $\{a_1,\ldots,a_m\}$ is an independent family of commuting infinite-order elements, and consequently, $\langle a_1,\ldots,a_m\rangle\cong\mathbb Z^m$.

\medskip
(iii) {\bf Type $D_n$ for $n \ge 4$}:
We first consider the case $n=2m+1$.
The subdiagram spanned by $\{s_1,\ldots,s_{n-1}\}$ is of type $A_{n-1}$.
By Lemma~\ref{lem:An-PQR}(iii), the argument in Proposition~\ref{thm:An-free-abelian}(i) applies to $a_1,\ldots,a_m$.
Hence, these elements are independent.

It remains to consider the case $n=2m$.
Assume that $\prod_{i=1}^m a_i^{n_i}=1$ in $PC_W$.
Apply Yu's embedding $\Phi$ to get $\prod_{i=1}^m \Phi(a_i)^{n_i}=1$ in $\widehat{W}$.
Keep the parabolic subgroups $P,Q$ as in Lemma~\ref{lem:An-PQR}(iii), 
and let
\[
\rho_{\{P,Q\}}: \widehat{W}\to \langle\tau_P,\tau_Q\rangle\cong \mathbb Z_2*\mathbb Z_2
\]
be the homomorphism defined in Remark~\ref{rmk:rho_A}.
For $i \le m-1$, the word defining $a_i$ involves only generators $\sigma_{I}$ with $I\subseteq I_{1,n-2}$. Hence, every parabolic subgroup appearing in $\Phi(a_i)$ is contained in $W_{I_{1,n-2}}$. In particular, none of the Coxeter generators $\tau_P,\tau_Q$ appears in $\Phi(a_i)$. Thus, $\rho(\Phi(a_i))=1$.

On the other hand, by Lemma~\ref{lem:An-PQR}(iii), we have $\rho_{\{P,Q\}}(\Phi(a_m))=(\tau_P\tau_Q)^2$, which has infinite order by Lemma~\ref{lem:An-infinite-order}(iii).
Thus, $\rho(\Phi(a_m))^{n_m}=1$ forces $n_m=0$.
Therefore, $\prod_{i=1}^{m-1} a_i^{n_i}=1$ in the free abelian group $\langle a_1,\dots,a_{m-1}\rangle\cong\mathbb Z^{m-1}$,
and hence $n_1=\cdots=n_{m-1}=0$. This proves that $\langle a_1,\dots,a_m\rangle\cong \mathbb Z^m$.

\medskip
(iv) {\bf Types $E_6$, $E_7$, and $E_8$}:
Assume that $\prod_{i=1}^m a_i^{n_i}=1$
in $PC_W$ for some integers $n_i$. For types $E_6$ and $E_7$, we have $m=3$, while for type $E_8$, we have $m=4$.
Applying $\Phi$, we obtain $\prod_{i=1}^m \Phi(a_i)^{n_i}=1$ in $\widehat W$. We show successively, starting from the largest index,
that all $n_i$ vanish.

Suppose first that $W$ is of type $E_8$. Keep $U_0$ and $V_0$ as in
Lemma~\ref{lem:An-PQR}(iv), and consider
\[
\rho_{\{U_0,V_0\}}:
\widehat W\longrightarrow
\langle\tau_{U_0},\tau_{V_0}\rangle
\cong \mathbb Z_2*\mathbb Z_2.
\]
By the proof of Lemma~\ref{lem:An-infinite-order}(iv), $\rho_{\{U_0,V_0\}}(\Phi(a_4))=(\tau_{U_0}\tau_{V_0})^2$,
which has infinite order.

By Lemma~\ref{lem:An-PQR}(iv), every parabolic subgroup appearing in
$\Phi(a_1)$, $\Phi(a_2)$, or $\Phi(a_3)$ is contained in
$W_{I_{1,7}}$ and has rank at most $5$. Since $U_0=W_{I_{1,7}}$ has rank $7$ and $V_0 \not\subseteq W_{I_{1,7}}$,
neither $U_0$ nor $V_0$ appears among the parabolic subgroups in
$\Phi(a_1)$, $\Phi(a_2)$, or $\Phi(a_3)$.
Therefore, for $i=1, 2, 3$, we have $\rho_{\{U_0,V_0\}}(\Phi(a_i))=1$, and this gives $(\tau_{U_0}\tau_{V_0})^{2n_4}=1$. Hence, we have $n_4=0$.

Next, we consider $a_3$.
Keep $U=W_{I_{1,5}}$ and $V=s_6Us_6^{-1}$ as in the proof of Lemma~\ref{lem:An-infinite-order}(iv).
By that proof, $\rho_{\{U,V\}}(\Phi(a_3))=(\tau_V\tau_U)^2$, which has infinite order.
Since all parabolic subgroups appearing in $\Phi(a_1)$ and $\Phi(a_2)$
have rank at most $3$, while $U$ and $V$ have rank $5$, we have $\rho_{\{U,V\}}(\Phi(a_1))=\rho_{\{U,V\}}(\Phi(a_2))=1$, and that $(\tau_V\tau_U)^{2n_3}=1$.
Therefore, we have $n_3=0$.

It remains to consider $a_1$ and $a_2$.
In types $E_6$, $E_7$, and $E_8$, the elements $a_1$ and $a_2$ have exactly the same form as in type $A_k$ ($k\geq 4$), and the standard parabolic subgroup $W_{I_{1,4}}$ has type $A_4$. By Proposition~\ref{thm:An-free-abelian}(i), we obtain $\langle a_1,a_2\rangle \cong \mathbb Z^2$, and thus, we have $n_2=n_1=0$.
Therefore, $n_1=\cdots=n_m=0$.
Consequently, $a_1,\ldots,a_m$ are independent. Together with the preceding lemmas, this proves that $\langle a_1,\ldots,a_m\rangle\cong\mathbb Z^m$.

\medskip
(v) {\bf Type $F_4$}: 
Assume that $a_1^{n_1}a_2^{n_2}=1$ in $PC_W$ for some integers $n_1$ and $n_2$. Applying $\Phi$, we obtain $\Phi(a_1)^{n_1}\Phi(a_2)^{n_2}=1$ in $\widehat{W}$.
Keep $U=W_{I_{1,3}}$ and $V=\omega_{I_{1,3}}\langle s_4\rangle
\omega_{I_{1,3}}^{-1}$
as in the proof of Lemma~\ref{lem:An-infinite-order}(v), 
and consider
\[
\rho_{\{U,V\}}:
\widehat W\longrightarrow
\langle\tau_U,\tau_V\rangle
\]
defined in Remark~\ref{rmk:rho_A}.
By the proof of Lemma~\ref{lem:An-infinite-order}(v), $\rho_{\{U,V\}}(\Phi(a_2))=(\tau_U\tau_V)^2$, which has infinite order.

Every parabolic subgroup appearing in $\Phi(a_1)$ is contained in
$W_{I_{1,2}}\le U$ and has rank $1$. Since $U$ has rank $3$ and
$V\not\subset U$, neither $U$ nor $V$ appears among these parabolic
subgroups. Hence, $\rho_{\{U,V\}}(\Phi(a_1))=1$.
Applying $\rho_{\{U,V\}}$ gives $(\tau_U\tau_V)^{2n_2}=1$, and thus, $n_2=0$.

We are then left with $\Phi(a_1)^{n_1}=1$.
Since $\Phi(a_1)$ has infinite order by Lemma~\ref{lem:An-infinite-order}(v), we obtain $n_1=0$.
Therefore, $a_1$ and $a_2$ are independent.
Together with the preceding lemmas, this proves that $\langle a_1,a_2\rangle\cong\mathbb Z^2$.

\medskip
(vi) {\bf Types $H_3$ and $H_4$}:
In type $H_3$, $\langle a_1\rangle$ is infinite cyclic subgroup in $PC_W$ by Lemmas~\ref{lem:An-ai-in-PCW}(v) and~\ref{lem:An-infinite-order}(v).
In type $H_4$, the same argument as in type $F_4$ shows that $\langle a_1,a_2\rangle\cong\mathbb Z^2$.

\medskip
(vii) {\bf Type $I_2(p)$ for $p\ge 5$}: By Lemma~\ref{lem:An-infinite-order}(vii), $a_1$ has infinite order. Therefore, $\langle a_1\rangle\cong\mathbb Z$.

\end{proof}

\section{Application: Non-quasi-isometry to RAAGs}\label{sec:application}

Throughout this section, except for Section~\ref{subsec:conspiciality}, we assume that $(W,S)$ is a finite Coxeter system and write $S=S_1\sqcup\cdots\sqcup S_k$ for the decomposition of the Coxeter diagram into connected components.
We also let $n_i=|S_i|$ and $n=|S|$.

Our aim is to prove Theorem~\ref{thm:mainRAAGQI} by extending
Genevois's results on cubical actions and manifold structures
for ordinary cactus groups to generalized cactus groups.
We prove that $PC_W$ acts conspicially on the Cayley graph $\mathscr C_W$ and has trivial cube-stabilizers on its cube-completion $\mathscr A_W$, which implies that $PC_W$ is torsion-free.
We then show that $\mathscr A_W$ splits as $\mathscr M_W\times [0,1]^k$, identify the vertex links of $\mathscr M_W$ in terms of nested complexes, and deduce that $\mathscr M_W$ is a \text{CAT(0)} cubical $(n-k)$-manifold.
This yields $\operatorname{asdim}(C_W)=n-k$ and leads to an obstruction to being quasi-isometric to a right-angled Artin group.

\subsection{Conspiciality of the pure action}\label{subsec:conspiciality}

For ordinary cactus groups, Genevois~\cite[Section~4]{Genevois} labeled hyperplanes by subsets of strands and used this labeling to prove the conspiciality of the action of the pure cactus groups.
These subsets also index the standard generators in his embedding into a right-angled Coxeter group.
For generalized cactus groups, we use the parabolic subgroups indexing the standard generators in Yu's embedding~ \cite[Section~2.2]{Yu} as hyperplane labels. We adapt this labeling argument to prove that the action $PC_W\curvearrowright\mathscr C_W$ is conspicial.
We first recall the definition of conspiciality and the relevant terminology concerning hyperplanes.

\begin{definition}[Conspiciality, self-intersections, self-osculations and inter-osculations]
Recall that an action on a median graph is called {\em conspicial} 
if it has no self-intersections, self-osculations, or inter-osculations.

Here, a \emph{self-intersection} is the data of a hyperplane $H$ of $X$ and an element $g\in G$ such that $H$ and $gH$ are transverse; a \emph{self-osculation} is the data of a hyperplane $H$ of $X$ and an element $g\in G$ such that $H$ and $gH$ are tangent; and an \emph{inter-osculation} is the data of two hyperplanes $H_1,H_2$ of $X$ and an element $g\in G$ such that $H_1$ is transverse to $H_2$ and $gH_1$ is tangent to $H_2$.
\end{definition}

Recall that $\mathscr{C}_W$ is the Cayley graph with respect to the generating set $\{\sigma_I\}$.
For an oriented edge $e=(h,h\sigma_I)$ of $\mathscr{C}_W$, we define
\[
\lambda(e)\ :=\ g_W(h)\,W_I\,g_W(h)^{-1}\ \le W.
\]
The label $\lambda(e)$ is independent of the orientation of $e$. Indeed, for the edge $e'=(h\sigma_I,h)$ with the opposite orientation, we have
\begin{align*}
\lambda(e')&=g_W(h\sigma_I)\,W_I\,g_W(h\sigma_I)^{-1}=g_W(h)\,\omega_I\,W_I\,\omega_I^{-1}\,g_W(h)^{-1}\\
&=g_W(h)\,W_I\,g_W(h)^{-1}=\lambda(e).
\end{align*}

We next show that all edges in the same hyperplane of $\mathscr{C}_W$ have the same label. Thus, we may write $\lambda(H)$ for the common label of a hyperplane $H$.

\begin{lemma}\label{lem:hyperplane-label}
Opposite edges in any 4-cycle of $\mathscr{C}_W$ have the same label under $\lambda$.
Consequently, $\lambda$ is constant on each hyperplane of $\mathscr{C}_W$.
\end{lemma}

\begin{proof}
By the last assertion of Lemma~\ref{lem:squares}, it is enough to see two types of 4-cycles obtained by the commuting and nesting relations.

First, assume $I\perp J$ and consider the 4-cycle
\[
(h,\ h\sigma_I,\ h\sigma_I\sigma_J,\ h\sigma_J,\ h).
\]
There are two pairs of opposite edges:
\[
\{(h,h\sigma_I),(h\sigma_J,h\sigma_J\sigma_I)\} \quad \text{and} \quad \{(h,h\sigma_J),(h\sigma_I,h\sigma_I\sigma_J=h\sigma_J\sigma_I)\}
\]
By symmetry, it suffices to consider the first pair.
By the definition of $\lambda$,
\[
\lambda(h\sigma_J,h\sigma_J\sigma_I)=g_W(h\sigma_J)W_Ig_W(h\sigma_J)^{-1}
=g_W(h)\omega_J\,W_I\,\omega_J^{-1}g_W(h)^{-1}.
\]
Since $I\perp J$, every element of $W_I$ commutes with every element of $W_J$. Hence, $\omega_JW_I\omega_J^{-1}=W_I$, and therefore, $\lambda(h\sigma_J,h\sigma_J\sigma_I)=\lambda(h,h\sigma_I)$.

Next, we assume that $I\subset J$ and consider the 4-cycle coming from $\sigma_I\sigma_J=\sigma_J\sigma_{\omega_J I\omega_J^{-1}}$:
\[
(h,\ h\sigma_I,\ h\sigma_I\sigma_J=h\sigma_J\sigma_{\omega_J I\omega_J^{-1}},\ h\sigma_J,\ h).
\]
There are two pairs of opposite edges:
\[
\{(h,h\sigma_I),(h\sigma_J,h\sigma_J\sigma_{\omega_J I\omega_J^{-1}})\} \quad \text{and} \quad \{(h,h\sigma_J),(h\sigma_I,h\sigma_I\sigma_J=h\sigma_J\sigma_{\omega_J I\omega_J^{-1}})\}.
\]
In the former case, we have
\begin{align*}
    \lambda(h\sigma_J,h\sigma_J\sigma_{\omega_J I\omega_J^{-1}})
    &=g_W(h\sigma_J)W_{\omega_J I\omega_J^{-1}}g_W(h\sigma_J)^{-1}\\
    &=g_W(h)\omega_J(\omega_JW_I\omega_J^{-1})\omega_J^{-1}g_W(h)^{-1}\\
    &=\lambda(h,h\sigma_I),
\end{align*}
where we use $\omega_J^2=1$.
Likewise, the other pair of opposite edges $(h,h\sigma_J)$ and $(h\sigma_I,h\sigma_I\sigma_J)$
satisfies $\lambda(h\sigma_I, h\sigma_I\sigma_J)=\lambda(h,h\sigma_J)$, since $I\subset J$ implies $W_I\leq W_J$ and hence $\omega_I\in W_J$.

\smallskip
Since any two edges in the same hyperplane can be connected by a sequence of opposite edges of 4-cycles, all edges in the same hyperplane have the same label.
\end{proof}

An edge belonging to a hyperplane $H$ is said to be {\em dual} to $H$.
Since any two edges dual to the same hyperplane can be connected by a sequence of opposite edges of $4$-cycles, Lemma~\ref{lem:hyperplane-label} shows that all edges dual to the same hyperplane have the same label.

We now establish three properties of the label $\lambda$ that will be used to prove the conspiciality of the action $PC_W\curvearrowright\mathscr{C}_W$.

\begin{lemma}\label{lem:local-inj}
Fix a vertex $h\in C_W$. If $I\neq J$, then $\lambda(h,h\sigma_I)\neq \lambda(h,h\sigma_J)$.
\end{lemma}

\begin{proof}
If $\lambda(h,h\sigma_I)=\lambda(h,h\sigma_J)$, then $g_W(h)W_Ig_W(h)^{-1}=g_W(h)W_Jg_W(h)^{-1}$,
Hence, we have $W_I=W_J$, and therefore $I=J$.
\end{proof}

\begin{lemma}\label{lem:pure-invariant}
For every $p\in PC_W$ and every hyperplane $H$ of $\mathscr{C}_W$, one has $\lambda(pH)=\lambda(H)$.
\end{lemma}

\begin{proof}
Let $e=(h,h\sigma_I)$ be an oriented edge dual to $H$. Then $p\cdot e=(ph,ph\sigma_I)$ is dual to $pH$ and
\[
\lambda(p\cdot e)=g_W(ph)W_Ig_W(ph)^{-1}
=g_W(p)\,g_W(h)W_Ig_W(h)^{-1}\,g_W(p)^{-1}.
\]
Since $p\in\ker(g_W)=PC_W$, we have $\lambda(pH)=\lambda(p \cdot e)=\lambda(e)=\lambda(H)$.
\end{proof}

\begin{lemma}\label{lem:crossing-by-label}
Let $H$ and $K$ be hyperplanes of $\mathscr{C}_W$, 
and let $e=(h,h\sigma_I)$ and $e'=(h,h\sigma_J)$ be distinct edges incident to $h$, dual to $H$ and $K$, respectively.
Then the following are equivalent:
\begin{enumerate}[\rm (i)]
\item The edges $e$ and $e'$ span a $4$-cycle.
\item The subsets $I$ and $J$ are orthogonal or nested.
\item The parabolic subgroups $\lambda(H)$ and $\lambda(K)$ are orthogonal or nested in $W$.    
\end{enumerate}
\end{lemma}

\begin{proof}
By Lemma~\ref{lem:squares}, the edges $e$ and $e'$ span a 4-cycle if and only if $I$ and $J$ are orthogonal or nested.
Thus, (i) and (ii) are equivalent.

To show the equivalence of (ii) and (iii), note that
\[
\lambda(H)=\lambda(h,h\sigma_I)=g_W(h)W_Ig_W(h)^{-1},
\qquad
\lambda(K)=\lambda(h,h\sigma_J)=g_W(h)W_Jg_W(h)^{-1}.
\]
Since conjugation by $g_W(h)$ is an automorphism of $W$, the subgroups
$W_I$ and $W_J$ are orthogonal or nested if and only if
$\lambda(H)$ and $\lambda(K)$ are orthogonal or nested.
Thus, (ii) and (iii) are equivalent.
\end{proof}

Now we are ready to 
prove that the action $PC_W\curvearrowright \mathscr{C}_W$ is conspicial.
We use the hyperplane language of \cite[\S2.1]{Genevois}.

\begin{theorem}\label{thm:pcw-conspicial}
The action $PC_W\curvearrowright \mathscr{C}_W$ is conspicial.
\end{theorem}

\begin{proof}
By Lemma~\ref{lem:hyperplane-label}, all edges dual to the same hyperplane have the same label,
and by Lemma~\ref{lem:pure-invariant},
$\lambda(pH)=\lambda(H)$ for every $p\in PC_W$ and every hyperplane $H$ of $\mathscr{C}_W$.
{\setlength{\leftmargini}{2em}
\begin{itemize}
\item \emph{No self-intersections.}
Suppose that there exists a hyperplane $H$ and $p \in PC_W$ such that $H$ and $pH$ are transverse.
Then there exists a vertex $h$, and there exist two distinct edges $e=(h, h\sigma_I)$ and $e'=(h, h\sigma_J)$ dual to $H$ and $pH$, respectively, and spanning a 4-cycle.
By Lemmas~\ref{lem:hyperplane-label} and~\ref{lem:pure-invariant},
\[
\lambda(e)=\lambda(H)=\lambda(pH)=\lambda(e').
\]
Lemma~\ref{lem:local-inj} then implies that $I=J$, and hence $e=e'$,
contradicting the choice of $e$ and $e'$ as distinct edges.
Therefore, no self-intersection occurs.
\item \emph{No self-osculation.}
Suppose there exist a hyperplane $H$ and $p\in PC_W$ such that $H$ and $pH$ are tangent at some vertex $h$.
Then, there exist two distinct intersecting edges $e=(h,h\sigma_I)$ and $e'=(h,h\sigma_J)$ dual to $H$ and $pH$, respectively, which do not span a 4-cycle.
By Lemmas~\ref{lem:hyperplane-label} and~\ref{lem:pure-invariant},
$\lambda(e)=\lambda(H)=\lambda(pH)=\lambda(e')$.
As in the self-intersection case, Lemma~\ref{lem:local-inj} implies that
$e=e'$, a contradiction.
Therefore, no self-osculation occurs.
\item \emph{No inter-osculations.}
Suppose that there exist hyperplanes $H,K$ and an element $p\in PC_W$ such that $H$ and $K$ are transverse, while $pH$ and $K$ are tangent.
Since $H$ and $K$ are transverse, there exist distinct dual edges meeting
at a vertex and spanning a $4$-cycle. Hence, by
Lemma~\ref{lem:crossing-by-label}, $\lambda(H)$ and $\lambda(K)$ are
orthogonal or nested.
\end{itemize}
}

Since $pH$ and $K$ are tangent, there exists a vertex $h'$, and there exist two distinct edges
\[
e_{pH}=(h',h'\sigma_{I'})
\quad\text{and}\quad
e_K=(h',h'\sigma_{J'})
\]
dual to $pH$ and $K$, respectively,
which do not span a $4$-cycle.
Then by Lemmas~\ref{lem:hyperplane-label} and~\ref{lem:pure-invariant}, we have
\[
\begin{aligned}
\lambda(H)&=\lambda(pH)=\lambda(e_{pH})=g_W(h')W_{I'}g_W(h')^{-1}, \quad \text{and} \\
\lambda(K)&=\lambda(e_K)=g_W(h')W_{J'}g_W(h')^{-1}.
\end{aligned}
\]
Since $H$ and $K$ are transverse, Lemma~\ref{lem:crossing-by-label}
implies that $\lambda(H)$ and $\lambda(K)$ are orthogonal or nested.
By Lemma~\ref{lem:pure-invariant}, $\lambda(pH)=\lambda(H)$, and hence $\lambda(pH)$ and $\lambda(K)$ are orthogonal or nested.
Therefore, $I'$ and $J'$ are orthogonal or nested.
Lemma~\ref{lem:crossing-by-label} implies that $e_{pH}$ and $e_K$
span a $4$-cycle, contradicting the tangency of $pH$ and $K$.
Therefore, no inter-osculation occurs.
Thus, the action $PC_W\curvearrowright\mathscr C_W$ is conspicial.
\end{proof}

\subsection{Trivial cube-stabilizers}


We prove that the action of $PC_W$ on $\mathscr A_W$ has trivial cube stabilizers.
For ordinary cactus groups, Genevois ordered the strand labels of the hyperplanes crossing a cube and showed that the corresponding composition of strand reversals induces a nontrivial permutation~ \cite[the proof of Theorem~4.2]{Genevois}.
We retain this approach, replacing the argument about strand permutations with one using longest elements and parabolic subgroups of the Coxeter group.
We first recall the terminology of tubes, tubings, and nested complexes~\cite{Carr_Devadoss, Manneville_Pilaud}.

\begin{definition}[Tubes, tubings, and nested complexes]
\label{def:tubes-compatible-nested}
Let $\Gamma$ be a finite graph. A nonempty subset $A\subset V(\Gamma)$ is called
{\em connected} if the induced subgraph $\Gamma[A]$ is connected.
Two connected subsets $A,B\subset V(\Gamma)$ are {\em compatible} if either they are nested, that is, $A\subset B$ or $B\subset A$, or they are disjoint and non-adjacent, that is, $A\cap B=\emptyset$ and no vertex of $A$ is adjacent to a vertex of $B$.

When $\Gamma$ is connected, a \emph{tube} of $\Gamma$ is a nonempty proper connected subset $T\subsetneq V(\Gamma)$. A \emph{tubing} is a collection of pairwise compatible tubes.
The \emph{nested complex} $N(\Gamma)$ is the simplicial complex whose vertices are the tubes of $\Gamma$ and whose simplices are tubings.
\end{definition}

\begin{proposition}\label{prop:pcw-trivial-cube-stab}
The action $PC_W\curvearrowright \mathscr A_W$ has trivial cube-stabilizers.
\end{proposition}

\begin{proof}
Suppose that $p \in PC_W$ stabilizes a cube $Q$, and choose $x \in Q^{(0)}$. Then $R:=x^{-1}Q$ is a cube containing $1$. Since $pQ=Q$, we have $x^{-1}(px) \in R^{(0)}$. Moreover, since $PC_W$ is normal in $C_W$, it follows that $x^{-1}px\in PC_W$.
Therefore, if $R^{(0)} \cap PC_W={1}$ for every cube $R$ containing $1$, then $x^{-1}px=1$, and hence $p=1$. Thus, after relabeling such a cube $R$ as $Q$, it is enough to show that $Q^{(0)}\cap PC_W=\{1\}$ for every cube $Q$ containing $1$. To see this, we claim that $g_W(h)\neq 1$ for each $h\in Q^{(0)} \setminus \{1\}$.

Fix $h\in Q^{(0)} \setminus \{1\}$, and let $F$ be the smallest face of $Q$ containing $1$ and $h$.
Thus, $1$ and $h$ are opposite vertices of $F$.
Let $\mathcal{U}$ consist of the indices $I$ for which $(1,\sigma_I)$ is an edge of $F$.
If $|\mathcal{U}|=1$, say $\mathcal U=\{M\}$, then $F$ is an edge, and hence $h=\sigma_M$. Therefore, $g_W(h)=\omega_M\neq 1$.
Thus, we may assume that $|\mathcal U|=r\ge 2$.
By Lemma~\ref{lem:squares}, the elements of $\mathcal{U}$ are pairwise orthogonal or nested.

Choose a maximal element $M \in \mathcal{U}$, and order $\mathcal{U}=\{I_1,\ldots,I_r\}$ so that $I_r=M$ and $I_i\subsetneq I_j$ implies $i<j$ for $1\le i,j\le r$.
Since the elements of $\mathcal U$ are pairwise compatible, for $k<j$, either $I_k\subsetneq I_j$ or $I_k\perp I_j$.
In either case, the defining relations give a square with vertices $(v,v\sigma_{I_k},v\sigma_{I_k}\sigma_{I_j},v\sigma_{I_j})$ in this cyclic order.
Hence, the edge opposite to $(v,v\sigma_{I_j})$ is $(v\sigma_{I_k},\,v\sigma_{I_k}\sigma_{I_j})$ which is also labeled $\sigma_{I_j}$.

We now use the preceding observation to follow the directions $I_1, \dots, I_r$ successively from~$1$. 
This gives a path in $F$
\[
1 \to \sigma_{I_1} \to \sigma_{I_1}\sigma_{I_2} \to \cdots
\to \sigma_{I_1} \cdots \sigma_{I_r},
\]
whose successive edge labels are $\sigma_{I_1}, \dots, \sigma_{I_r}$.
Since this path ends at the opposite vertex $h$, we obtain 
\[
h=\sigma_{I_1}\cdots\sigma_{I_{r-1}}\sigma_M.
\]

Let $T=\bigcup_{j<r} I_j$.
We claim that $M \setminus T \neq \emptyset$. 
Suppose to the contrary that $M\subseteq T$.
Let $J_1, \dots, J_m$ be maximal members of $\{I \in \mathcal{U} \mid I \subsetneq M\}$ with respect to inclusion. Then $M=\bigcup_{\substack{j<r\ I_j\subsetneq M}} I_j=J_1 \cup \cdots \cup J_m$.
Moreover, $J_1, \dots, J_m$ are pairwise disjoint and non-adjacent, and hence there is no edge in the Coxeter diagram joining two distinct $J_i$'s.
Therefore, if $m\ge2$, it follows that the induced Coxeter diagram on $M$ is disconnected, contradicting the fact that $M\in \mathcal{I}$. If $m=1$, then $M=J_1$, contradicting $J_1\subsetneq M$.
Thus, $M\not\subseteq T$, and hence $M \setminus T \neq \emptyset$.

Choose $s \in M \setminus T$. Since $\omega_M$ is the longest element of $W_M$, we have $\ell(\omega_M s) < \ell(\omega_M)$.
Thus, Lemma~\ref{lem:parabolic-right-descent} gives $\omega_M \notin W_T$.
This implies that $g_W(h)=\omega_{I_1}\cdots\omega_{I_{r-1}}\omega_M\neq 1$ since $\omega_{I_1}\cdots\omega_{I_{r-1}} \in W_T$, proving $Q^{(0)}\cap PC_W=\{1\}$.
\end{proof}

\begin{corollary}\label{cor:pcw-torsionfree}
The group $PC_W$ is torsion-free.
\end{corollary}

\begin{proof}
Let $g\in PC_W$ have finite order. Since $\mathscr A_W$ is a finite-dimensional $\text{CAT(0)}$ cube complex, by the Bruhat–Tits fixed point theorem, the finite group $\langle g\rangle$ fixes a point of $\mathscr A_W$.
Thus, $\langle g\rangle$ stabilizes the minimal cube containing this fixed point.
By Proposition~\ref{prop:pcw-trivial-cube-stab}, every cube-stabilizer is trivial.
Therefore, $g=1$.
\end{proof}

\subsection{A bounded factor in the cube-completion}\label{subsec:bounded-factor}

As in the discussion of cactus groups \cite[Section~7]{Genevois}, the generalized cactus group also admits a decomposition that splits off a bounded cubical factor in the cube-completion.
Let $\mathcal I^\circ=\mathcal I\setminus\{S_1,\ldots,S_k\}$, and let $C_W^\circ\le C_W$ be the subgroup generated by $\{\sigma_I\mid I\in\mathcal I^\circ\}$.

We first note that the elements $\sigma_{S_1},\ldots,\sigma_{S_k}$ generate a subgroup isomorphic to $\mathbb{Z}_2^k$.
Moreover, conjugation by $\sigma_{S_i}$ preserves the set of generators of $C_W^\circ$. Indeed, if $I\subset S_i$ then
the defining relations give $\sigma_{S_i}\sigma_I\sigma_{S_i}=\sigma_{\omega_{S_i}I\omega_{S_i}^{-1}}$.
On the other hand, if $I\subset S_j$ for some $j\neq i$, then $S_i\perp I$, and hence, $\sigma_{S_i}\sigma_I\sigma_{S_i}=\sigma_I$.
Thus, $\sigma_{S_i}C_W^\circ\sigma_{S_i}=C_W^\circ$. Consequently, $C_W^\circ$ is normal in $C_W$, and we obtain a semidirect product
decomposition
\begin{equation}\label{eq:cw-semidirect}
C_W \cong C_W^\circ\rtimes \mathbb{Z}_2^k,
\end{equation}
where the action of $\mathbb{Z}_2^k$ on $C_W^\circ$ permutes the generating set
$\{\sigma_I\mid I\in\mathcal I^\circ\}$.

Let $Y=\mathrm{Cay}(C_W^\circ,\{\sigma_I\mid I\in\mathcal I^\circ\})$ and  $Q_k=\mathrm{Cay}(\mathbb{Z}_2^k,\{\sigma_{S_1},\ldots,\sigma_{S_k}\})$ be Cayley graphs.
As the action in~\eqref{eq:cw-semidirect} preserves the generating set of $C_W^\circ$,
the Cayley graph of $C_W$ with respect to
\[
\{\sigma_I\mid I\in\mathcal I^\circ\}\ \cup\ \{\sigma_{S_1},\ldots,\sigma_{S_k}\}
\]
splits as a direct product:
\[
\mathscr C_W \cong Y\times Q_k.
\]
Since $Q_k$ is the $1$-skeleton of the $k$-cube, its cube-completion is $[0,1]^k$. Let $\mathscr M_W$ denote the cube-completion of $Y$.
Then taking cube-completions yields
\begin{equation}\label{eq:AW-splits}
\mathscr A_W \cong \mathscr M_W\times [0,1]^k.
\end{equation}
In particular, the projection $\mathscr M_W\times [0,1]^k\to \mathscr M_W$ is a
quasi-isometry, since $[0,1]^k$ has finite diameter.

\subsection{Asymptotic dimension in finite Coxeter type}


Following Genevois's analysis of cactus groups, we focus on the unbounded factor $\mathscr M_W$ in the decomposition $\mathscr A_W \cong \mathscr M_W\times [0,1]^k$ obtained in Section~\ref{subsec:bounded-factor}.
For ordinary cactus groups, Genevois describes the link of a (equivalently, any) vertex of the corresponding factor in terms of compatible intervals and proves by induction that they are spheres~\cite[the proof of Proposition~7.1]{Genevois}.

In this subsection, we follow the same strategy for generalized cactus groups: we identify the links of vertices in $\mathscr M_W$ with joins of nested complexes associated with the connected components of the Coxeter diagram.
Using the graph associahedra~\cite[Theorem~2.6]{Carr_Devadoss}, we realize each nested complex as the boundary complex of a simplicial polytope dual to the corresponding graph associahedron.
This shows that the vertex links are spheres, and hence that $\mathscr M_W$ is a cubical manifold.
As in \cite[Remark~7.6]{Genevois}, this manifold structure then allows us to compute $\operatorname{asdim}(C_W)=\operatorname{asdim}(PC_W)$ from the dimension of $\mathscr M_W$ as a $\mathrm{CAT}(0)$ cube complex.

We keep the terminology of tubes, tubings, and nested complexes from Definition~\ref{def:tubes-compatible-nested}. 
We use the convention
\[
\mathbb S^{-1}=\{\emptyset\}.
\]
Note that $N(S_i)=\{\emptyset\}=\mathbb S^{-1}$ whenever a connected component $S_i$ consists of a single vertex, and $X * \{\emptyset\}=X$ for any simplicial complex $X$.
In addition, we describe the combinatorics of vertex links in $\mathscr M_W$, and relate them to nested complexes of the Coxeter diagram.

\begin{lemma}\label{lem:link-generalized-cactus}
Let $v$ be a vertex of $\mathscr M_W$. Then
\[
\operatorname{Lk}_{\mathscr M_W}(v)
\cong
N(S_1)*\cdots *N(S_k),
\]
where $N(S_i)$ denotes the nested complex of the Coxeter diagram induced by $S_i$.
\end{lemma}

\begin{proof}
Choose a vertex $\epsilon\in ([0,1]^k)^{(0)}$ and set
\[
x:=(v,\epsilon)\in \mathscr M_W\times [0,1]^k\cong \mathscr A_W.
\]
Since links in a product of cube complexes are joins of links, we have
\[
\operatorname{Lk}_{\mathscr A_W}(x)
\cong
\operatorname{Lk}_{\mathscr M_W}(v)*
\operatorname{Lk}_{[0,1]^k}(\epsilon).
\]
The link $\operatorname{Lk}_{[0,1]^k}(\epsilon)$ is a $(k-1)$-simplex, whose vertices correspond to the $k$ factors $S_1,\ldots,S_k$ under \eqref{eq:AW-splits}.

We now analyze $\operatorname{Lk}_{\mathscr A_W}(x)$ in order to identify the remaining factor.
Since taking the cube-completion does not change the $1$-skeleton, the vertices of $\operatorname{Lk}_{\mathscr A_W}(x)$ are naturally indexed by $\mathcal I$.
Moreover, since $\mathscr A_W$ is $\text{CAT(0)}$, its links are flag.
Hence, for any finite collection $I_1, \dots, I_m \in \mathcal{I}$, the set $\{I_1,\ldots,I_m\}\subset\mathcal I$ spans a simplex in $\operatorname{Lk}_{\mathscr A_W}(x)$ if and only if each pair spans an edge.

Under the identification $\mathscr A_W\cong \mathscr M_W\times [0,1]^k$, the link $\operatorname{Lk}_{\mathscr M_W}(v)$ is the full subcomplex of $\operatorname{Lk}_{\mathscr A_W}(x)$ obtained by deleting the $k$ vertices corresponding to the connected components $S_1, \dots, S_k$.
Thus, the remaining vertices are indexed by $\mathcal I\setminus\{S_1,\ldots,S_k\}$.
For each $i$, the remaining indices contained in $S_i$ are precisely the proper nonempty connected subsets $I\subsetneq S_i$, that is, the tubes of the induced diagram on $S_i$.

Fix $i$. By Lemma~\ref{lem:squares}, two vertices indexed by proper connected subsets $I,J\subsetneq S_i$ are adjacent in the link if and only if $I$ and $J$ are orthogonal or nested. Since the link is flag, a finite collection of such vertices spans a simplex if and only if the corresponding tubes are pairwise compatible.
Therefore, the full subcomplex spanned by the vertices is the nested complex $N(S_i)$.

Finally, if $i\neq j$, then any $I\subset S_i$ and $J\subset S_j$ are orthogonal, so the corresponding vertices of $\operatorname{Lk}_{\mathscr M_W}(v)$ are adjacent.
Therefore, $\operatorname{Lk}_{\mathscr M_W}(v)\cong N(S_1)*\cdots *N(S_k)$, as desired.
\end{proof}

To understand the topology of the links from Lemma~\ref{lem:link-generalized-cactus}, we use the notion of graph associahedra in \cite[Definition~2.4]{Carr_Devadoss} and the following standard fact.

\begin{definition}[Graph associahedra]
\label{def:graph-associahedron}
Let $\Gamma$ be a finite connected graph with vertex set $V(\Gamma)$, where $|V(\Gamma)|=q$.
Let $\Delta_\Gamma$ be the $(q-1)$-simplex whose facets are labeled by the vertices of $\Gamma$. For a tube $T$ of $\Gamma$, let $F_T$ denote the face of $\Delta_\Gamma$ obtained as the intersection of the facets labeled by the vertices of $T$.
The \emph{graph associahedron} $P_\Gamma$ is obtained from $\Delta_\Gamma$ by successively truncating the faces $F_T$ corresponding to non-singleton tubes $T$, in increasing order of dimension.
\end{definition}

For an illustration of this construction, see \cite[Example~2.5]{Carr_Devadoss}.
In the definition of graph associahedra, note that the facets of $P_\Gamma$ are naturally indexed by the tubes of $\Gamma$: singleton tubes correspond to the original facets of $\Delta_\Gamma$, while non-singleton tubes correspond to the new facets created by the truncations.

\begin{lemma}
\label{lem:graph-associahedron-nested-complex}
Let $\Gamma$ be a finite connected graph with $q$ vertices.
Then the nested complex $N(\Gamma)$ is a simplicial sphere of dimension $q-2$.
\end{lemma}

\begin{proof}
The case $q=1$ is immediate from the convention $N(\Gamma)=\{\emptyset\}=\mathbb{S}^{-1}$.
Hence, assume that $q\geq 2$.
Let $P_\Gamma$ be the graph associahedron.

By \cite[Theorem~2.6]{Carr_Devadoss}, $P_\Gamma$ is a simple convex polytope.
Moreover, there is an isomorphism $\varphi$ between the tubings of
$\Gamma$, ordered by reverse inclusion, and the nonempty faces of
$P_\Gamma$, ordered by inclusion.
For each tube $T$, we denote by $\mathcal F_T$ the facet of $P_\Gamma$
corresponding to $T$, so that $\varphi(\{T\})=\mathcal F_T$.

Let $\mathcal U$ be a tubing.
For every $T\in\mathcal U$, we have $\varphi(\mathcal U) \subseteq \varphi(\{T\})=\mathcal F_T$.
Hence,
\[
\varphi(\mathcal U)
\subseteq
\bigcap_{T\in\mathcal U}\mathcal F_T.
\]
In particular, $\bigcap\limits_{T\in\mathcal U}\mathcal F_T$ is a nonempty face of $P_\Gamma$.
Let this face correspond under $\varphi$ to a tubing $\mathcal V$, that is, $\varphi(\mathcal V)=\bigcap\limits_{T\in\mathcal U}\mathcal F_T$.
Since $\varphi(\mathcal V) \subseteq \mathcal F_T=\varphi(\{T\})$ for every
$T\in\mathcal U$, it follows that $\mathcal V$ contains every
$T\in\mathcal U$. Hence, $\mathcal U \subseteq \mathcal V$.
By the inclusion-reversing property of $\varphi$, this gives
\[
\bigcap_{T\in\mathcal U}\mathcal F_T
\subseteq
\varphi(\mathcal U).
\]
Combining the two inclusions, we obtain
\[
\varphi(\mathcal U)
=
\bigcap_{T\in\mathcal U}\mathcal F_T,
\]
where, by convention, the empty intersection is $P_\Gamma$.

Now, let $T_1,\ldots,T_m$ be tubes of $\Gamma$.
If $\{T_1,\ldots,T_m\}$ is a tubing, then the equality above gives $\varphi(\{T_1,\ldots,T_m\})=\mathcal F_{T_1}\cap\cdots\cap\mathcal F_{T_m}$,
and hence $\mathcal F_{T_1}\cap\cdots\cap\mathcal F_{T_m}\neq\emptyset$.
Conversely, suppose that $\mathcal F_{T_1}\cap\cdots\cap\mathcal F_{T_m}\neq\emptyset$.
Then this intersection is a nonempty face of $P_\Gamma$, and hence
corresponds under $\varphi$ to a tubing of $\Gamma$.
Since this face is contained in
$\mathcal F_{T_i}=\varphi(\{T_i\})$ for each $i=1,\ldots,m$, the
corresponding tubing contains each $T_i$.
Therefore, $\{T_1,\ldots,T_m\}$ is a tubing.
Thus, $\mathcal F_{T_1}\cap\cdots\cap\mathcal F_{T_m}\neq\emptyset$
if and only if $\{T_1,\ldots,T_m\}$ is a tubing.

Recall that a \emph{dual polytope} $P^\ast$ of a convex polytope $P$
is a polytope together with a bijection between the faces of $P$ and
the faces of $P^\ast$ that reverses inclusion.
Let $P_\Gamma^\ast$ be a dual polytope of $P_\Gamma$.
Since $P_\Gamma$ is simple, $P_\Gamma^\ast$ is simplicial.
For each tube $T$, let $v_T$ be the vertex of $P_\Gamma^\ast$
corresponding to the facet $\mathcal F_T$ of $P_\Gamma$.
By the inclusion-reversing correspondence between the faces of
$P_\Gamma$ and $P_\Gamma^\ast$, together with the fact that
$P_\Gamma^\ast$ is simplicial, a collection of vertices
$\{v_{T_1},\ldots,v_{T_m}\}$ spans a simplex in
$\partial P_\Gamma^\ast$ if and only if $\mathcal F_{T_1}\cap\cdots\cap\mathcal F_{T_m}\neq\emptyset$.
Thus, $\{v_{T_1},\ldots,v_{T_m}\}$ spans a simplex in
$\partial P_\Gamma^\ast$ if and only if
$\{T_1,\ldots,T_m\}$ is a tubing.

Since the simplices of $N(\Gamma)$ are precisely the tubings of
$\Gamma$, the assignment $T\mapsto v_T$
induces a simplicial isomorphism $N(\Gamma)\cong\partial P_\Gamma^\ast$.
Since $\dim P_\Gamma^\ast=q-1$, its boundary complex is a simplicial
sphere of dimension $q-2$.
Therefore, $N(\Gamma)$ is a simplicial $(q-2)$-sphere.
\end{proof}

Combining Lemma~\ref{lem:link-generalized-cactus} and Lemma~\ref{lem:graph-associahedron-nested-complex}, we obtain the following.

\begin{lemma}\label{lem:manifold-factor-generalized-cactus}
The cube complex $\mathscr{M}_W$ is a $\text{CAT(0)}$ cubical manifold of dimension $n-k$.
\end{lemma}

\begin{proof}
By Lemma~\ref{lem:link-generalized-cactus}, for every vertex $v$ of $\mathscr{M}_W$,
\[
\operatorname{Lk}_{\mathscr{M}_W}(v)
\cong
N(S_1)*\cdots *N(S_k).
\]
For each component $S_i$ with $n_i\ge 2$, Lemma~\ref{lem:graph-associahedron-nested-complex} gives $N(S_i)\cong \mathbb S^{n_i-2}$.
If $n_i=1$, then there are no proper tubes in $S_i$, and we use the convention $N(S_i)=\emptyset=\mathbb S^{-1}$.
Thus,
\[
\operatorname{Lk}_{\mathscr{M}_W}(v) \cong \mathbb S^{n_1-2}*\cdots *\mathbb S^{n_k-2}.
\]
Using the fact that $\mathbb S^a*\mathbb S^b\cong \mathbb S^{a+b+1}$ for $a, b \ge -1$, we obtain
\[
\operatorname{Lk}_{\mathscr{M}_W}(v) \cong \mathbb S^{\sum_{i=1}^k(n_i-2)+(k-1)}=\mathbb S^{n-k-1}.
\]
Therefore, every vertex link of $\mathscr{M}_W$ is a simplicial sphere of dimension $n-k-1$, so $\mathscr{M}_W$ is a cubical $(n-k)$-manifold.
Finally, since $\mathscr{M}_W$ is a cubical factor of the $\text{CAT(0)}$ cube complex $\mathscr{A}_W$, it is also $\text{CAT(0)}$.
\end{proof}

Now we are ready to compute the asymptotic dimension of generalized cactus groups associated with finite Coxeter systems.

\begin{theorem}\label{thm:asdim-generalized-cactus-finite}
Let $(W,S)$ be a finite Coxeter system with connected components $S_1,\dots,S_k$, and let $n=|S|$.
Then
\[
\operatorname{asdim}(C_W)=\operatorname{asdim}(PC_W)=n-k.
\]
\end{theorem}

\begin{proof}
Since $W$ is finite, $PC_W$ has finite index in $C_W$.
Hence, $PC_W$ and $C_W$ are quasi-isometric. Since asymptotic dimension is invariant under quasi-isometry, it suffices to compute $\operatorname{asdim}(C_W)$.

Moreover, the group $C_W$ is quasi-isometric to its Cayley graph $\mathscr C_W$. Since $\mathscr C_W$ is the $1$-skeleton of the finite-dimensional cube complex $\mathscr A_W$, it is also quasi-isometric to $\mathscr A_W$. 
Since the factor $[0,1]^k$ is bounded, by~\eqref{eq:AW-splits}, it follows that $\mathscr A_W$ is quasi-isometric to $\mathscr{M}_W$.
Therefore, $\operatorname{asdim}(C_W)=\operatorname{asdim}(\mathscr M_W)$.

By Lemma~\ref{lem:manifold-factor-generalized-cactus}, $\mathscr M_W$ is a $(n-k)$-dimensional \text{CAT(0)} cube complex and a cubical $(n-k)$-manifold. 
By~\cite[1.$F_1$]{Gromov} and \cite[Theorem~4.10]{Wright}, $\operatorname{asdim}(\mathscr M_W)=n-k$.
\end{proof}

\begin{remark}\label{rem:asdim-typeA}
When $(W,S)$ is of type $A_{n-1}$, the Coxeter diagram is connected and $|S|=n-1$. Since $C_{A_{n-1}}\cong J_n$, Theorem~\ref{thm:asdim-generalized-cactus-finite} gives
\[
\operatorname{asdim}(J_n)=n-2,
\]
recovering Genevois's computation for ordinary cactus groups.
\end{remark}

As another direct consequence of Lemma~\ref{lem:manifold-factor-generalized-cactus}, one can obtain the virtual cohomological dimension, denoted by $\mathrm{vcd}$, of generalized cactus groups.

\begin{corollary}
With the notation in Theorem~\ref{thm:asdim-generalized-cactus-finite}, $\mathrm{vcd}(C_W)=n-k$.    
\end{corollary}

\subsection{Non-quasi-isometric to RAAGs}

In this subsection, we prove that if the Coxeter diagram of $(W,S)$ has a connected component of size at least $3$, the generalized cactus group $C_W$ cannot be quasi-isometric to any RAAG.

We first prove the following theorem and then combine it with Theorem~\ref{thm:maximal_rank} to prove Theorem~\ref{thm:mainRAAGQI}.

\begin{theorem}\label{thm:QI-RAAG-forces-Zd}
Let $(W,S)$ be a finite Coxeter system with connected components $S_1,\dots,S_k$ and let $n=|S|$.
Assume that $C_W$ is quasi-isometric to a RAAG.
Then $C_W$ contains a subgroup isomorphic to $\mathbb Z^{n-k}$.
\end{theorem}

\begin{proof}
Let $A_\Gamma$ be a RAAG quasi-isometric to $C_W$.
By Theorem~\ref{thm:asdim-generalized-cactus-finite}, we have $n-k=\operatorname{asdim}(C_W)=\operatorname{asdim}(A_\Gamma)$.
By~\cite[Theorem~1.2]{Tselekidis}, we have $\operatorname{asdim}(A_\Gamma)=\omega(\Gamma)$.
Here, $\omega(\Gamma)$ denotes the {\em clique number} of $\Gamma$, that is, the maximum number of vertices spanning a complete subgraph of $\Gamma$. 
Hence, $A_\Gamma$ contains an undistorted subgroup $\mathbb Z^{\omega(\Gamma)}\cong\mathbb Z^{n-k}$ of $A_\Gamma$.

Composing an isometric embedding $\mathbb Z^{n-k}\hookrightarrow A_\Gamma$ with a quasi-isometry $A_\Gamma\to C_W$ yields a quasi-isometric embedding $\phi:\mathbb Z^{n-k}\to C_W$.
Since $W$ is finite, $PC_W$ has finite index in $C_W$, hence the inclusion $PC_W\hookrightarrow C_W$ is a quasi-isometry.
Let $r : C_W \rightarrow PC_W$ be the quasi-inverse of the inclusion $PC_W\hookrightarrow C_W$. Then $f:=r \circ \phi : \mathbb Z^{n-k}\to PC_W$ is a quasi-isometric embedding.

Note that the left-multiplication action $C_W\curvearrowright \mathscr{C}_W$ is free and cocompact.
Since $PC_W\le C_W$ has finite index, the restricted action $PC_W\curvearrowright \mathscr{C}_W$ is also free and cocompact, and the induced action on the cube-completion $\mathscr{A}_W$ is cocompact.
Moreover, by Proposition~\ref{prop:pcw-trivial-cube-stab}, it has trivial cube stabilizers, and hence it is free.
Hence, the quotient
\[
X \ :=\ \mathscr{A}_W / PC_W
\]
is a compact nonpositively curved cube complex with $\pi_1(X)\cong PC_W$.

On the other hand, by Theorem~\ref{thm:pcw-conspicial}, the action $PC_W\curvearrowright \mathscr{C}_W$ is conspicial. 
Thus, the hyperplanes of $X$ are embedded and do not self-osculate
or inter-osculate.
Hence, $X$ is {\em special} in the sense of~\cite[Definition~3.2]{Haglund_Wise}.
Since $X$ is compact, it has finitely many hyperplanes.
Therefore, by \cite[Lemma~3.11]{Haglund_Wise}, there exists a finite cover $p:Y\to X$ such that every hyperplane of $Y$ is two-sided.
Since $X$ is special, it follows from \cite[Corollary~3.8]{Haglund_Wise} that the finite cover $Y$ is also special.
Hence, by \cite[Definition~3.2]{Haglund_Wise}, $Y$ is {\em $A$-special}, that is, every hyperplane is embedded and two-sided,
with no self-osculations or inter-osculations.
In particular, $Y$ is a compact special cube complex.

We also have $\pi_1(Y)\ \le\ \pi_1(X)\ \cong \ PC_W$ as a finite index subgroup.
Moreover, the universal cover $\widetilde{Y}$ of $Y$ identifies with that of $X$, so $\widetilde{Y} \cong \mathscr{A}_W$.
Fix $y_0\in \widetilde Y\cong \mathscr{A}_W$. Since $\pi_1(Y)$ acts geometrically on $\widetilde Y$, by \v{S}varc--Milnor lemma, the orbit map $\mathrm{orb}_{y_0}:\pi_1(Y)\rightarrow \widetilde Y$ defined by $\mathrm{orb}_{y_0}(g)=g\cdot y_0$ is a quasi-isometry.

Since $\pi_1(Y)$ has finite index in $PC_W$, there exists a quasi-isometry $g: PC_W\to \pi_1(Y)$ which is a quasi-inverse to the inclusion $\pi_1(Y)\hookrightarrow PC_W$.
Then the subset
\[
Q \ :=\ (\mathrm{orb}_{y_0}\circ g\circ f)(\mathbb{Z}^{n-k}) \ \subset\ \widetilde{Y} \cong \mathscr{A}_W
\]
is an $(n-k)$-quasiflat in $\widetilde{Y}$.
Here, a $d$-\emph{quasiflat} in a metric space is the image of a quasi-isometric embedding of $\mathbb Z^d$.
By \eqref{eq:AW-splits} and Lemma~\ref{lem:manifold-factor-generalized-cactus}, we have a decomposition $\mathscr{A}_W \ \cong\ \mathscr{M}_W\times [0,1]^k$,
where $\mathscr{M}_W$ is a $(n-k)$-dimensional $\text{CAT(0)}$ cube complex.
Since $[0,1]^k$ is bounded, the projection $\pi:\mathscr{M}_W\times [0,1]^k\to\mathscr{M}_W$ is a quasi-isometry.
Hence, $\pi(Q)\subset\mathscr{M}_W$ is again a $(n-k)$-quasiflat.

Since $\mathscr{M}_W$ is a $(n-k)$-dimensional
$\mathrm{CAT}(0)$ cube complex and $\pi(Q)$ is a $(n-k)$-quasiflat in
$\mathscr M_W$, \cite[Theorem~1.1]{Huang} implies that $\pi(Q)$ lies at
finite Hausdorff distance from a finite union of $(n-k)$-dimensional orthant
subcomplexes of $\mathscr M_W$.
Recall that, in Huang's terminology, an {\em orthant subcomplex} is a convex cubical subcomplex which is isometric to the Euclidean orthant $[0,\infty)^{n-k}$ with respect to the $\text{CAT(0)}$ metric.

Choose one of these orthant subcomplexes and denote it by $O\subset \mathscr M_W$.

\medskip 
\noindent {\bf Claim.} {\em The cubical structure on $O$ is the standard cubulation of $[0,\infty)^{n-k}$.}

\begin{proof}[Proof of Claim]
Identify $O$ isometrically with $[0,\infty)^{n-k}$.
Under this isometry, each cube of $O$ is mapped to a Euclidean unit cube in $[0,\infty)^{n-k}$.
Since $O$ is a cube complex,
these cubes cover $[0,\infty)^{n-k}$, and the intersection of any two of them is either empty or a common face.

We first show that the origin of $[0,\infty)^{n-k}$ is a vertex of the induced cubulation.
Otherwise, it would lie in the relative interior of a cube of positive dimension, so there would exist distinct points $a,b\in[0,\infty)^{n-k}$ with $0=(a+b)/2$.
Since all coordinates of $a$ and $b$ are nonnegative, this forces $a=b=0$, a contradiction.

Next, we show that an $(n-k)$-dimensional cube containing the origin is $[0,1]^{n-k}$.
To show this, we first note that the induced cubulation of $[0,\infty)^{n-k}$ is locally finite.
Since $Y$ is compact, its universal cover $\widetilde Y$ is
locally finite. Since $\widetilde Y\cong \mathscr M_W\times [0,1]^k$, the cubical subcomplex $\mathscr M_W$ is also locally finite. Hence, as
$O$ is a cubical subcomplex of $\mathscr M_W$, the induced cubulation
of $[0,\infty)^{n-k}$ is locally finite.

By local finiteness, a sufficiently small neighborhood of the origin is covered by finitely many cubes containing it.
Since this neighborhood is $(n-k)$-dimensional, at least one of these cubes must have dimension $n-k$.
Such an $(n-k)$-dimensional cube must be $[0,1]^{n-k}$.
To see this, let $u_1,\ldots,u_{n-k}$ be the unit edge vectors of this cube at the origin.
The vectors $u_i$ are pairwise orthogonal because the cube is Euclidean,
while, since the cube lies in $[0,\infty)^{n-k}$, all coordinates of the
$u_i$ are nonnegative.
Hence, no two of them can have nonzero entries in the same coordinate.
As there are $n-k$ unit vectors in $\mathbb R^{n-k}$, they are, up to
permutation, precisely the coordinate unit vectors.

Finally, we show that every standard unit cube occurs. Suppose that
\[
C_{\mathbf m}=\prod_{i=1}^{n-k}[m_i,m_i+1],
\qquad \mathbf m\in\mathbb Z_{\geq0}^{\,n-k},
\]
is the image of a cube of $O$ under the above isometry.
For each $i$, the relative interior of its face $x_i=m_i+1$ lies in the interior of $[0,\infty)^{n-k}$.
Hence, another $(n-k)$-cube meets $C_{\mathbf m}$ along this face on the opposite side, that is, in the direction $x_i>m_i+1$.
The unique Euclidean unit cube sharing this face with $C_{\mathbf m}$ on the side $x_i>m_i+1$ is $C_{\mathbf m+e_i}$, where $e_i$ is the $i$th coordinate unit vector.
Starting from $C_{\mathbf0}=[0,1]^{n-k}$, induction therefore shows that every standard unit cube occurs.
Since these cubes cover $[0,\infty)^{n-k}$, the induced cubulation is the standard one, proving the claim.
\end{proof}

By Claim, fix a cubical isomorphism
\[
\iota:[0,\infty)^{n-k}\xrightarrow{\cong} O
\subset \mathscr M_W.
\]
Choose a vertex $v\in ([0,1]^k)^{(0)}$ and consider the cubical embedding
\[
j_v:\mathscr M_W\hookrightarrow \mathscr M_W\times [0,1]^k,\qquad
x\mapsto (x,v).
\]
Both $\iota$ and $j_v$ are cubical isomorphisms onto their images. Hence, their composition
\[
j_v\circ\iota:[0,\infty)^{n-k}\hookrightarrow \widetilde Y
\]
is also a cubical isomorphism onto its image, and therefore a combinatorial embedding.
Thus, the image of $j_v\circ\iota$ is an $(n-k)$-dimensional combinatorial flat in $\widetilde Y$, in the sense used in the proof of \cite[Theorem~4.6]{Genevois_2}.
Since $Y$ is a compact special cube complex, that theorem implies
that $\pi_1(Y)$ contains a subgroup isomorphic to $\mathbb Z^{n-k}$. Consequently, as $\pi_1(Y)$ has finite index in $PC_W\le C_W$, it follows that $C_W$ contains $\mathbb Z^{n-k}$.
\end{proof}

We are now ready to prove Theorem~\ref{thm:mainRAAGQI}.

\begin{proof}[Proof of Theorem~\ref{thm:mainRAAGQI}]
Suppose that the generalized cactus group $C_W$ is quasi-isometric to a RAAG.
Let $n=|S|$ and let $S_1,\ldots,S_k$ be connected components of the Coxeter diagram of $(W,S)$.
By Theorem~\ref{thm:QI-RAAG-forces-Zd}, $C_W$ contains a free abelian subgroup $\mathbb{Z}^{n-k}$.
However, by Theorem~\ref{thm:maximal_rank}, we have $n-k\leq \sum_{i=1}^k \lfloor n_i/2 \rfloor$, where $n_i=|S_i|$ for each $i$.
Since $n-k=\sum_{i=1}^k(n_i-1)$ and $\lfloor n_i/2 \rfloor \le n_i-1$ with equality if and only if $n_i\le 2$, the above inequality can hold only if $n_i\le 2$ for every $i$.
Since some connected component $S_i$ has at least three vertices, we have $\lfloor n_i/2 \rfloor < n_i-1$ for at least one $i$. Hence, $\sum_{i=1}^k \lfloor n_i/2 \rfloor<n-k$, contradicting the inequality above.
Therefore, $C_W$ is not quasi-isometric to any RAAG.    
\end{proof}

Davis, Januszkiewicz, and Scott~\cite[Question~8.5.11]{Davis_Januszkiewicz_Scott_2} asked when the mock reflection groups associated with different associahedral tilings are quasi-isometric. Since generalized cactus groups arise as special cases of their mock reflection groups, this suggests the more general problem of classifying generalized cactus groups up to quasi-isometry.

For the quasi-isometry classification among generalized cactus groups, our computation of asymptotic dimension provides a partial distinction. However, asymptotic dimension does not distinguish generalized cactus groups associated with irreducible finite Coxeter groups of the same rank. Thus, we conclude this paper with the following question.

\begin{question}
Let $(W,S)$ and $(W',S')$ be irreducible finite Coxeter systems of the same rank. When are the generalized cactus groups $C_W$ and $C_{W'}$ quasi-isometric?
\end{question}

\section*{AI disclosure}

The authors used ChatGPT (OpenAI) for language editing and improvements to the presentation of the manuscript.
All AI-assisted content was reviewed, edited, and verified by the authors, who take full responsibility for the integrity of the work.

\medskip
\bibliographystyle{abbrv} 
\bibliography{reference_cactus}		
\end{document}